\documentclass[11pt]{article}

\usepackage[utf8]{inputenc} %
\usepackage[T1]{fontenc}    %
\usepackage{lmodern}
\usepackage[table]{xcolor}

\input{setup.sty}
\newcommand{\eqdst}{\stackrel{\mathsf{d}}{=}}
\newcommand{\permmapset}{\mathscr{S}}
\newcommand{\permset}{\mathscr{P}}
\newcommand{\tspace}{\mathbb{T}}

\newtheorem{manualtheoreminner}{Theorem}
\newenvironment{manualtheorem}[1]{%
	\IfBlankTF{#1}
	{\renewcommand{\themanualtheoreminner}{\unskip}}
	{\renewcommand\themanualtheoreminner{#1}}%
	\manualtheoreminner
}{\endmanualtheoreminner}

\usepackage{enumitem}
\usepackage{natbib}
\usepackage{mathrsfs}
\usepackage{setspace}

\title{Concentration Inequalities for Exchangeable Tensors and Matrix-valued Data}

\author{%
	Chen Cheng\thanks{Department of Statistics, University of Illinois Urbana-Champaign} \and Rina Foygel Barber\thanks{Department of Statistics, University of Chicago}
}

\begin{document}
\maketitle

\begin{abstract}

	We study concentration inequalities for structured weighted sums of random data, including (i) tensor inner products and (ii) sequential matrix sums. We are interested in tail bounds and concentration inequalities for those structured weighted sums under exchangeability, extending beyond the classical framework of independent terms.
	
	We develop Hoeffding and Bernstein bounds provided with structure-dependent exchangeability. Along the way, we recover known results in weighted sum of exchangeable random variables and i.i.d.\ sums of random matrices to the optimal constants. We also derive a concentration bound for combinatorial sums of matrix arrays with an improved constant in the bias term compared with that by the method of exchangeable pairs.
	
	For applications, the richer structures provide us with novel analytical tools for estimating the average effect of multi-factor response models and studying fixed-design sketching methods in federated averaging. We apply our results to these problems, and find that our theoretical predictions are corroborated by numerical evidence.  
\end{abstract}

\tableofcontents

\section{Introduction}\label{sec:intro}
A central goal in concentration inequalities involves deriving tail bounds for a centered random variable $Z - \E Z$ that equals to the weighted sum of random data, with the simplest form as
\begin{align*}
	Z = \<w, X\> := w_1 X_1 + w_2 X_2 + \cdots + w_N X_N,
\end{align*}
where $w = (w_1, \cdots, w_N)^\top$ are fixed weights in $\R$ and $X = (X_1, \cdots, X_N)^\top$ are i.i.d.\ random variables in $[-1,1]$. In the i.i.d.\ case, a number of classical results on bounding $\P(Z - \E Z \geq t)$ are well-known such as Hoeffding and Bernstein bounds, as well as Azuma-Hoeffding bound in the martingale case~\cite[Chp.~2]{wainwright2019high}. In the recent work~\citep{foygel2024hoeffding}, \citeauthor{foygel2024hoeffding} generalizes Hoeffding and Bernstein bounds to the exchangeable case, where the random data $(X_1, \cdots, X_N)$ have the same joint distribution as $(X_{\pi(1)}, \cdots, X_{\pi(N)})$ under any permutation $\pi$ on $[N]$.

It is then natural to ask, raised in part by~\citeauthor{foygel2024hoeffding}, \textit{how do such bounds extend to exchangeability with richer structures}? We consider a fixed weight (or an ordered set of weights) $W \in \tspace$ and a random data object $X \in \tspace'$ from vector spaces $\tspace, \tspace'$, where $X$ is distributionally invariant (exchangeable) under the symmetry group associated with $\tspace'$. We summarize the quantity of interest $Z$ as a bilinear form $Z:=\llangle W, X \rrangle \in \tspace''$, where $\tspace'' = \R$ or $\R^{n \times m}$. Our goal is to provide Hoeffding-type and Bernstein-type tail bounds for $\P(Z - \E Z > t)$ or $\P(\norm{Z - \E Z} > t)$. We consider two settings:
\begin{enumerate}
	\item \textbf{Mode-exchangeable tensors:} Our first setting considers a fixed $K$-mode tensor $W$, and a random tensor $X$ satisfying exchangeability along each mode (to be defined below). For intuition, here we describe a special case, with $K=2$ modes: $W$ and $X$ are both matrices, $W, X \in \R^{N \times M}$, and $X$ is \textbf{row-wise and column-wise exchangeable}, i.e., the row-wise and column-wise permuted matrix $(X_{\pi(i)\pi'(j)})_{i\in [N], j \in [M]}$ has the same distribution as $X =(X_{ij})_{i\in [N], j \in [M]}$, for any permutations $\pi$ on $[N]$ and $\pi'$ on $[M]$. We then study the concentration of the real-valued inner product $Z= \sum_{i=1}^N \sum_{j=1}^M W_{ij}X_{ij} \in \R$ (or, more generally, $Z$ is given by the tensor inner product).
	\item \textbf{Matrix-valued exchangeable data:} Let $W_1, \cdots, W_N \in \R^{p \times q}$ be fixed and let $X_1, \cdots, X_N \in \R^{q \times r}$ be random, such that $(X_1, \cdots, X_N)$ have the same joint distribution as $(X_{\pi(1)}, \cdots, X_{\pi(N)})$ under any permutation $\pi$ on $[N]$. We then establish concentration bounds on the matrix-valued sum $Z = \sum_{k=1}^N W_k X_k \in \R^{p \times r}$.
\end{enumerate}

We present an illustration of these settings in Table~\ref{tab:main-results}. The main theoretical contributions of our paper include Sec.~\ref{sec:master-tensor-bounds} where we investigate concentration inequalities for mode-exchangeable tensors, and Sec.~\ref{sec:master-matrix-valued-bounds} where we turn to exchangeable matrix-valued data. The remaining paper consists of the following. In Sec.~\ref{sec:theoretical-implications}, we provide theoretical implications of our results, recovering~\citeauthor{foygel2024hoeffding}'s result for exchangeable weighted sums and applying to the i.i.d.\ tensor-type and matrix-valued data. In particular, our tools imply tail bounds concerning combinatorial sum of matrix arrays studied in~\citet{chatterjee2007stein, mackey2014matrix}. In Sec.~\ref{sec:applications}, we apply our tools to two practical problems, and make concrete theoretical predictions with numerical simulations providing empirical evidence. We also include a detailed discussion of our contributions in Sec.~\ref{sec:contributions} at the end of the introduction.

\begin{table}[ht]
	\centering
	\begin{tabular}{l|llcc}
		\toprule 
		\rowcolor{gray!10} 
		\textbf{$W$ and $X$} &  \textbf{$\tspace$} & \textbf{$\tspace'$} & \textbf{$\llangle W, X \rrangle \mapsto Z$} & \textbf{\begin{tabular}{@{}c@{}}Exchangeability\\assumption on $X$\end{tabular}}\\
		\midrule
		\textbf{Sequence of scalars  }              & $\R^N$ & $\R^N$ & \includegraphics[width=.15\linewidth]{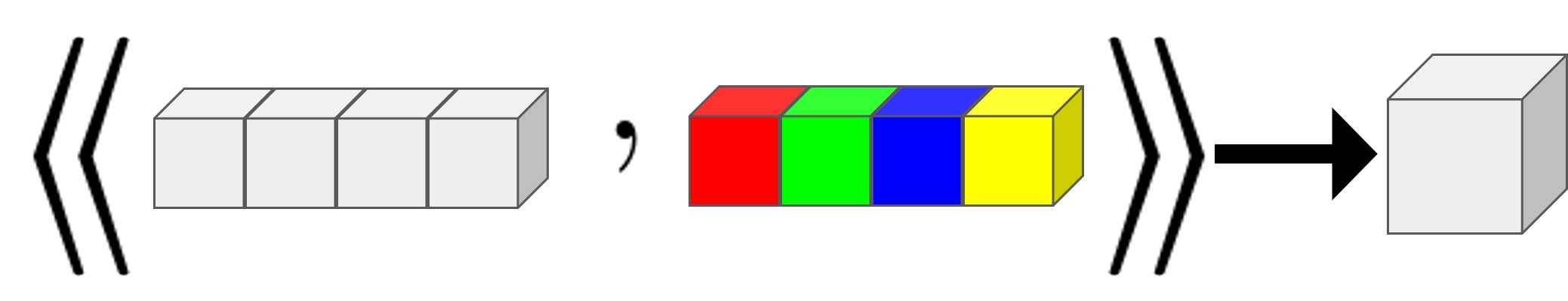}  & \includegraphics[width=.15\linewidth]{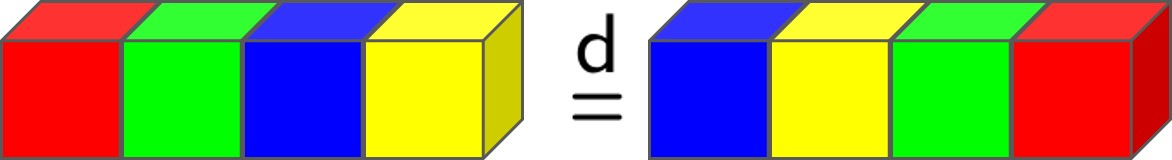} \\
		\midrule \raisebox{3\height}{\textbf{Tensors}}                 & \raisebox{2.3\height}{$\R^{N_1 \times \cdots \times N_K} $} & \raisebox{2.3\height}{$\R^{N_1 \times \cdots \times N_K}$} & \raisebox{1\height}{\includegraphics[width=.15\linewidth]{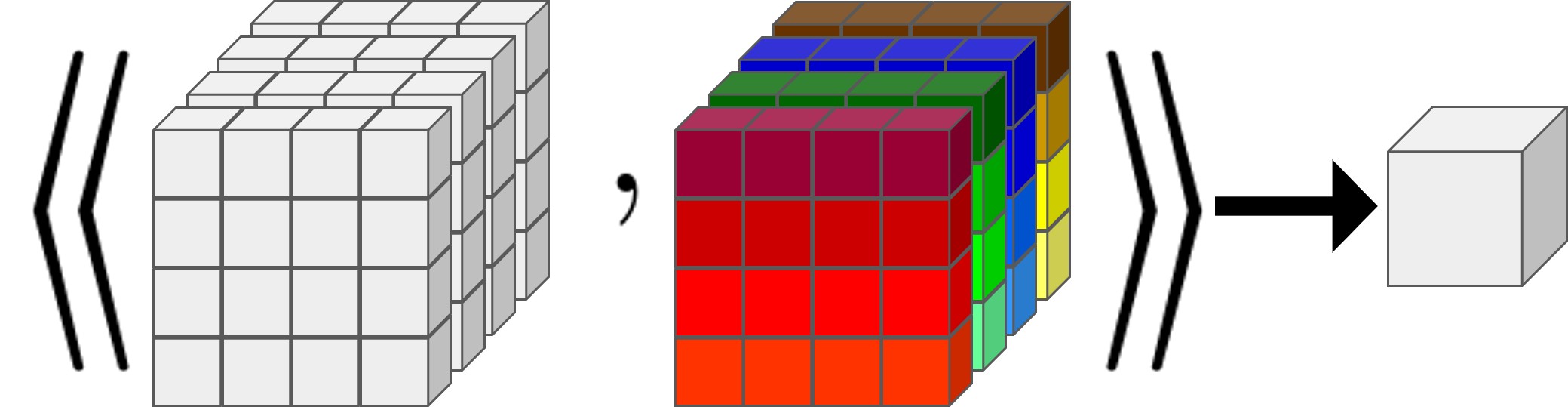}} & \includegraphics[width=.15\linewidth]{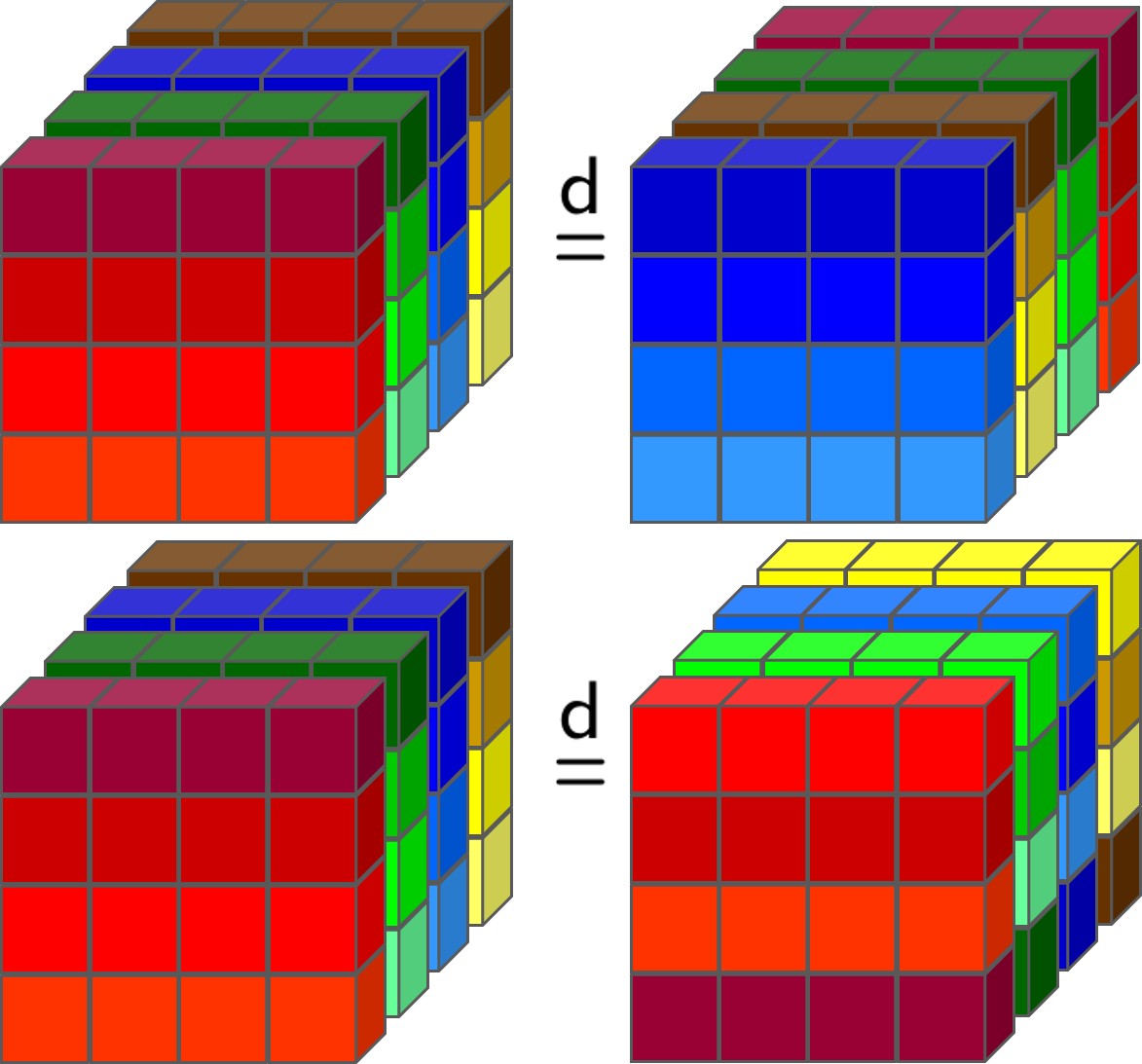} \\
		\midrule \textbf{Matrix-valued data} & $\R^{N \times p \times q}$ & $\R^{N \times q \times r}$ & \includegraphics[width=.15\linewidth]{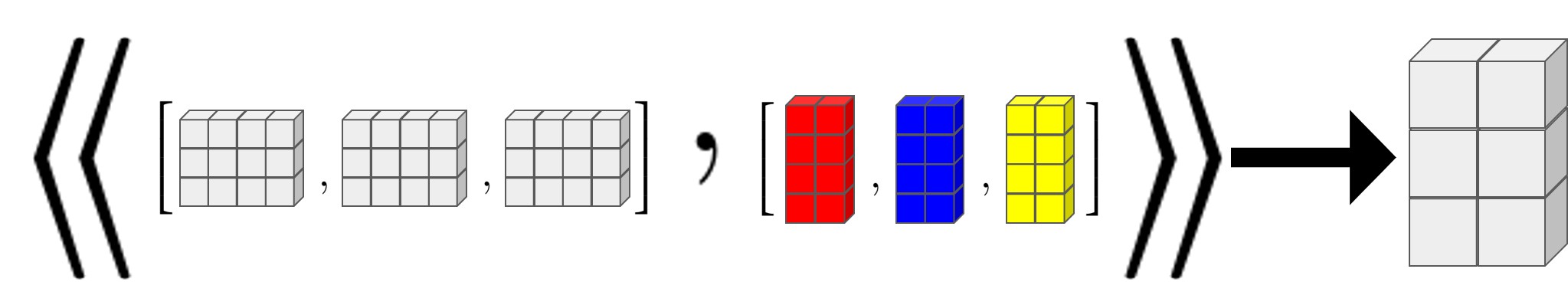} & \includegraphics[width=.15\linewidth]{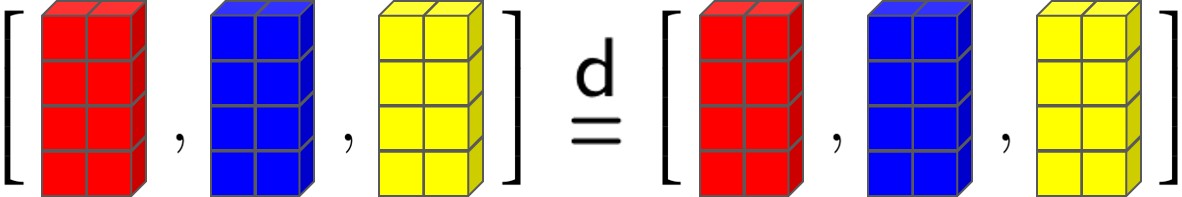} \\
		\bottomrule
	\end{tabular}
	\caption{An illustration of the settings considered in our paper, generalizing weighted sums of exchangeable scalars as studied in \cite{foygel2024hoeffding} (top row) to mode-exchangeable tensors (middle row) and exchangeable sequence of matrix-valued data (bottom row).}  \label{tab:main-results}
\end{table}

\subsection{Examples}
Before formalizing the mathematical framework and introducing the concentration bounds, we first preview two real-world examples when conventional i.i.d.\ or exchangeable assumptions break, while the richer structures of symmetry in randomness allow us to utilize the novel analytical tools. We will return to these examples later on in Section~\ref{sec:applications}.

The following first example concerns tensor-valued data that exhibit exchangeability along each mode.\smallskip

\begin{example}[Average effect of a multi-factor response] \label{example:avg-effect} Consider a multi-factor model, where factor $k \in [K]$ has discrete levels $[N_k]$. The multi-index $(i_1,\cdots, i_K), i_l \in [N_l]$ corresponds to an experimental condition with the $k$-th factor being $i_k$. For each experimental setup, running the test produces a random and possibly censored response $Y_{i_1\cdots i_K} \in \R$. Let $\sum_{i_l \in [N_l], l \in [K]} p_{i_1 \cdots i_K} = 1$ be a probability distribution $P$ on $[N_1] \times \cdots \times [N_K]$, we are interested in estimating the average effect of the multi-factor response
\begin{align*}
	\mu = \E_{(i_1, \cdots, i_K) \sim P}[Y_{i_1\cdots i_K}] = \sum_{i_l \in [N_l], l \in [K]} p_{i_1 \cdots i_K} Y_{i_1 \cdots i_K}.
\end{align*}
In practice, it may be infeasible to observe $Y_{i_1\cdots i_K}$ for all $N_1\cdot\hdots\cdot N_K$ possible combinations of factor levels.
Instead, we can consider the following estimation procedure of the average effect by sub-sampling a Cartesian product of factor-level subsets. Namely, we choose without replacement $I_l \subset [N_l]$ for all $l \in [K]$ with $|I_l| = n_l$, and let $X_{i_1\cdots i_K} = \prod_{l=1}^K \ind_{i_l \in I_l}$ be the indicator random variable of the sampled experiments---evidently, $X_{i_1 \cdots i_K}$ are exchangeable along each fixed factor level (as we shall define formally in the following in Def.~\ref{def:mode-ex}, $X$ is mode-exchangeable). In various real-world problems, this Cartesian product sub-sampling scheme can be attractive, even compared to independent sampling with the same number of samples, as changing factor setups for individual experiment can be costly while batch experiments greatly reduce cost. Several practical domains motivate such cost-reduction schemes, including Lithium-ion batteries aging under multiple factors (temperature, state-of-charge windows, charge/discharge C-rates and depth-of-discharge)~\citep{pelletier2017battery, roman2022design}, high-throughput drug combination screening in cancer (joint effect of drug A, drug B, dose levels and cell lines)~\citep{he2018methods} and semiconductor manufacturing process optimization (factors including gas flows, chamber pressure, RF power and electrode spacing)~\citep{may1991statistical}. 

Suppose we know $p_{i_1\cdots i_K}$ (or we have empirical estimates of the frequencies). Conditioned on the true responses $Y_{i_1 \cdots i_K}$ and let $W_{i_1 \cdots i_K} = p_{i_1 \cdots i_K} Y_{i_1 \cdots i_K}$, we can use the classical Horvitz-Thompson estimator for the average effect:
\begin{align*}
	\what{\mu} = \frac{N_1 \cdots N_K}{n_1 \cdots n_K} \cdot Z = \frac{N_1 \cdots N_K}{n_1 \cdots n_K} \cdot \llangle W, X \rrangle := \frac{N_1 \cdots N_K}{n_1 \cdots n_K} \sum_{i_l \in [N_l], l \in [K]} W_{i_1 \cdots i_K} X_{i_1 \cdots i_K}.
\end{align*}
The question is then to study the concentration of $\what{\mu} - \E [\what{\mu} \mid \{Y_{i_1 \cdots i_K}\}_{i_l \in [N_l], l \in [K]}]$, and in particular how to minimize the estimation error given a fixed sub-sampling budget $n_1 \cdots n_K \leq B$.
\end{example}

In the next example, we turn to sketching techniques widely used as a type of dimensional reduction method in statistical machine learning, where $W$, $X$ and the quantity of interest $\llangle W, X \rrangle$ belong to the matrix-valued venue in Table~\ref{tab:main-results}. In classical setups such as Gaussian sketch and count sketch, the sketching matrices are independent. While still being random, we consider sub-sampling from a pool of fixed-design sketching matrices as an alternative to the existing approaches, reducing the requirement of independence to exchangeability.\bigskip

\begin{example}[Fixed-design sketching matrices] \label{example:sketching}
	Consider a sequence of $N$ exchangeable parameters $\theta_k \in \R^{q \times r}$, and we are interested in the average of parameters $\wb{\theta} = \frac{1}{N} \sum_{k=1}^N \theta_k$. In distributed optimization setup where $\theta_k$ are local parameters, such as federated averaging~\citep{kairouz2021advances, cheng2021fine} and collaborative learning~\citep{cheng2023collaboratively}, computation is intensive in local agents and oftentimes there are constraints in communication cost between the central server and local agents. It is thus desirable to apply sketching techniques to compress local data $\theta_k$: we will take sketching matrices $U_k\in\R^{q'\times q}$, for some dimension $q'\ll q$, which satisfy
    \[\sum_{k=1}^N U_k^\top U_k \approx I_q.\]
    The local data $\theta_k\in\R^{q\times r}$ is then compressed to $z_k = U_k \theta_k \in \R^{q' \times r}$, and the local agents communicate $z_k$ to the central server, which then aggregates the compressed data:
	\begin{align*}
		\what{\theta} = \sum_{k=1}^N U_k^\top z_k=\sum_{k=1}^N U_k^\top U_k\theta_k.
	\end{align*}
	Since $\sum_{k=1}^N U_k^\top U_k \approx I_q$, we can then expect $\what{\theta} \approx \wb{\theta}$. 
    
    Standard approaches utilize independent sketching matrices $U_k$ such as Gaussian and Rademacher random matrices, and sub-sampled random Fourier transform (SRFT) matrices~\citep{halko2011finding}. Other popular methods include count sketch and row-column sub-sampling techniques, which can reduce computational cost since $U_k$ are sparse~\citep{charikar2002finding, mahoney2011randomized}. 
    
    In our setup, given exchangeability in local agents $\theta_1, \cdots, \theta_N$, we may apply a fixed sequence of sketching matrices $U_k$ with the exact equality $\sum_{k=1}^N U_k^\top U_k = I_q$, which is helpful in reducing the approximation error $\sum_{k=1}^N U_k^\top U_k \approx I_q$ that arises in a random sketching setup. Writing $W_k = U_k^\top U_k$ and $X_k = \theta_k$, we then have $\hat\theta=Z = \llangle W, X\rrangle = \sum_{k=1}^N W_k X_k$, and our results then establish concentration guarantees subject to matrix-valued exchangeability in Table~\ref{tab:main-results}.
\end{example}

\subsection{Related work }\label{sec:related-work}

Classical concentration inequalities for sums of independent random variables and their weighted variants, most frequently used as Hoeffding~\citep{hoeffding1963probability} and Bernstein bounds~\citep{bernstein1946theory}, are by now standard~\citep{boucheron2013concentration, wainwright2019high}. For matrix-valued sums with independent summands, the remarkable line of work by \citeauthor{tropp2012user} generalizes sums of random variables to matrix Hoeffding, Bernstein inequalities and other variants~\citep{tropp2012user, tropp2015introduction}, building on tools such as matrix Laplace transforms and operator inequalities. See the technical ideas by~\citeauthor{ahlswede2002strong} in~\cite{ahlswede2002strong} and Lieb's inequality~\citep{lieb1973convex}. We also point out recent work on operator-norm concentration for
random tensors, which typically studies the norm of a random tensor or a sum
of random tensor products under independent or related probabilistic
assumptions; see, for example,
\cite{bandeira2024geometric,al2025sharp}.

Less is known for similar tight bounds beyond independence, even in the case of exchangeable bounded random variables $X_1, \cdots, X_N$ which provides a canonical relaxation of independence while retaining strong symmetry. \citeauthor{de1929funzione}'s celebrated representation theorem~\citep{de1929funzione} suggests general tight Hoeffding and Bernstein alike bounds matching optimal constants in the i.i.d.\ case when $N \to \infty$. For example, in~\cite{ramdas2026randomized}, the authors prove a uniform Hoeffding inequality under infinite exchangeability with the optimal constant. \citeauthor{diaconis1977finite} quantifies distributional approximation of exchangeable sequences when $N$ is finite to a mixture of i.i.d.\ laws in total variation distance~\citep{diaconis1977finite}. See also a recent variant in Kullback–Leibler divergence in~\cite{gavalakis2021information}. In the unweighted case $w_1 = w_2 = \cdots = w_n =1, w_{n+1} = \cdots = w_N = 0$, typically studied as sampling without replacement, Serfling's inequality~\citep{serfling1974probability} provides sharp characterization stronger than the Hoeffding bound. Recent works~\citep{bardenet2015concentration, greene2017exponential} derive refined bounds for this particular setup.

\citeauthor{foygel2024hoeffding}'s recent work examines concentration results in the setting of finite exchangeability, and develops sharp Hoeffding and Bernstein bounds for general weighted sums of exchangeable sequences~\citep{foygel2024hoeffding}. Notably, in the recent work by \citeauthor{kim2026sharper}, the authors reduce the inflation factor from $O(\frac{\log N}{N})$ to the optimal rate $O(\frac{1}{N})$ in the scalar-valued Hoeffding bound. Our paper, directly inspired by \citeauthor{foygel2024hoeffding}'s work, extend the results to (i) mode-exchangeable tensors and (ii) matrix-valued sums formed from exchangeable sequences, which generalizes substantially to richer class of symmetry structures. A line of related route to concentration under dependence in the setup of sum of random matrices is \citeauthor{chatterjee2007stein}'s method of exchangeable pairs~\citep{chatterjee2005concentration, chatterjee2007stein, mackey2014matrix, han2024introduction}.

Sequential data $X_1, \cdots, X_N$ under exchangeability are prevalent in modern statistics, such as in permutation testing and conformal prediction~\citep{ramdas2023permutation, shafer2008tutorial, lei2018distribution, angelopoulos2023conformal}. Beyond the multi-factor model and fixed-design sketching examples, mode-exchangibility and exchangeable sums of random matrices are also useful in practical applications including U-statistics in bipartite networks~\citep{le2023u} and combinatorial nonparametric regression~\citep{han2024introduction}, where our novel tools can be potentially useful.

\subsection{Our contributions
} \label{sec:contributions}

Our starting point is the weighted
scalar concentration theory for finite exchangeable sequences developed in
\cite{foygel2024hoeffding}. We extend this theory in two distinct
directions. The tensor results address dependence induced by separate
permutations of several modes, whereas the matrix results address the
noncommutativity of matrix-valued observations and weights. Our the object of interest is
a fixed weighted linear functional of a tensor whose dependence is specified
through mode exchangeability.

For mode-exchangeable tensors, we establish Hoeffding- and
Bernstein-type concentration inequalities for weighted tensor inner products. Our
contribution is a mode-wise decomposition that successively centers and
controls the tensor along its exchangeable modes. This produces bounds that
make explicit both the exchangeability correction associated with each mode
and the dimensional inflation caused by dependence across the remaining
modes. The
resulting theory therefore identifies the concentration that is possible
under mode exchangeability alone, without entrywise
independence.

For exchangeable matrix-valued data, we establish both a
general Hoeffding bound and sharper spectral Hoeffding- and Bernstein-type
bounds under the commutativity-pair condition introduced in Assumption~\ref{assmp:commutativity}.
Classical matrix concentration inequalities treat independent random
matrices, while matrix exchangeable-pair methods such as
\cite{mackey2014matrix} are designed around a different coupling framework.
Our results instead begin with a finitely exchangeable sequence and exploit
its permutation symmetry directly. The scalar results of
\citeauthor{foygel2024hoeffding}, standard independent matrix concentration,
and combinatorial matrix sums are recovered as special or limiting cases.
Cor.~\ref{cor:matrix-valued-combinatorial} and~\ref{cor:sketching} provide explicit commutativity-pair constructions,
demonstrating that the condition applies beyond ordinarily commuting square
matrices.

The two parts of the paper are connected by a common technical
principle. Exchangeability allows the statistic to be averaged over random
permutations and represented through conditionally centered increments. In
the tensor setting, this argument is iterated one mode at a time, leading to
the modewise Hoeffding and variance terms in Section~\ref{sec:master-tensor-bounds}. In the matrix setting,
the same symmetry must be combined with Hermitian dilation and
trace-exponential arguments.

\paragraph{Notation} For any integer $n \geq 1$, $[n]$ denotes the set $\{1, \cdots, n\}$. For a vector $x \in \R^n$ and $p \geq 1$, we use $\|x\|_p$ for its $\ell_p$ norm and reserve $\|x\|$ for $p=2$. For a square matrix $A \in \R^{n \times n}$, we use $\|A\|$ for its operator norm and $\|A\|_F$ for its Frobenius norm. We use $\mathbb{S}^n$ for the vector space of all symmetric real-valued matrices in $\R^{n \times n}$ and $\mathbb{S}_+^n$ for the convex cone of all symmetric positive semidefinite matrices in $\mathbb{S}^n$. Let $\mathbb{S}_{++}^n$ be the cone of symmetric positive definite matrices. Let $\preceq$ be the partial order induced by the cone $\mathbb{S}_+^n$ and $\prec$ be that induced by $\mathbb{S}_{++}^n$. For a matrix $A \in \mathbb{R}^{n \times m}$ with $m \leq n$, let its pseudo-inverse be $A^\dagger$. We define the projection matrix onto the column space of $A$ by $\proj_A:= A(A^\top A)^\dagger A^\top$ and its orthogonal complement as $\proj_A^\perp := I_n - A(A^\top A)^\dagger A^\top$. In particular, for a unit vector $w \in \R^n$, $\proj_w = ww^\top$ and $\proj_w^\perp = I_n - ww^\top$. 
Let $\mathscr{S}_N$ be the set of all permutation mappings in $[N]$ and  $\permset_N$ be the collection of all permutation matrices in $\R^{N \times N}$. We denote by $\eqdst$ equality in distribution.

Next, we define some additional notation and terminology for working with tensors. For two tensors $A \in \R^{N_1 \times \cdots \times N_K}$ and $B \in \R^{N_1' \times \cdots \times N_{K'}'}$, let their outer product $A \otimes B \in \R^{N_1 \times \cdots \times N_K \times N_1' \times \cdots \times N_{K'}'}$ be
\begin{align*}
	(A \otimes B)_{i_1\cdots i_K i_1' \cdots i_{K'}'} = A_{i_1\cdots i_K} B_{i_1' \cdots i_{K'}'} \qquad \textrm{for all } \,\, i_l \in [N_l], i_{l'}' \in [N_{l'}'], \,\, l \in [K], l' \in [K'].
\end{align*}
We will also frequently use the inner product,
\[\<A, B\>_{\tspace} = \sum_{i_1=1}^{N_1}\hdots\sum_{i_K=1}^{N_K} A_{i_1\cdots i_K}B_{i_1\cdots i_K} \in\R.\]
Recall the $k$-mode tensor-matrix product (cf.~\cite[Sec.~2.5]{kolda2009tensor}) for $A \in \tspace$ and a matrix $B \in \R^{M \times N_k}$, $A \times_k B \in \mathbb{R}^{N_1 \times \cdots \times N_{k-1} \times M \times N_{k+1} \times \cdots \times N_K}$ whose entries are
\begin{align*}
	(A \times_k B)_{i_1\cdots i_{k-1} j i_k \cdots i_K} = \sum_{i_k=1}^{N_k} B_{j i_k} X_{i_1 \cdots i_K}.
\end{align*}
We will write
$
	\wb{A} := \frac{1}{N_1 \cdots N_K} \sum_{i_1=1}^{N_1} \cdots \sum_{i_K=1}^{N_K} A_{i_1 \cdots i_K} = \< A, \proj_{\ones_{N_{1}}} \otimes \cdots \proj_{\ones_{N_{K}}}\>
$
to denote the mean of all entries in a tensor. Analogous to the Frobenius norm for matrices, we can define the norm of a tensor $A$ in the Hilbert space $\tspace$ by $\norm{A}_{\tspace}:= \sqrt{\<A, A\>_{\tspace}}$. We also write $\lone{A} := \sum_{i_l \in [N_l], l \in [K]} |A_{i_1 \cdots i_K}|$ and $\linf{A} := \max_{i_l \in [N_l], l \in [K]} |A_{i_1 \cdots i_K}|$.

\section{Concentration inequalities for mode-exchangeable tensors} \label{sec:master-tensor-bounds}
In this section, we will work with tensor-structured data $X \in \R^{N_1 \times \cdots \times N_K}$, where $X_{i_1\cdots i_K} \in [-1,1]$ for all $i_k \in [N_k], k \in [K]$. Here $K$ is the number of modes of the tensor, with special cases $K=1$ (a vector) and $K=2$ (a matrix). Before defining the problem for the setting of a general $K$, we first single out these two special cases. For the vector case ($K=1$), as studied in \citet{foygel2024hoeffding}, we have
\[Z = \<w, X\>: = w^\top X = \sum_{i=1}^N w_i X_i.\]
In this vector setting, 
exchangeability means the joint distribution of $(X_1, \cdots, X_N)$ is invariant under any permutation of the index, i.e. $(X_1, \cdots, X_N) \eqdst(X_{\pi(1)}, \cdots, X_{\pi(N)})$ for all $\pi \in \permmapset_N$. Next, for the matrix case ($K=2)$, we have
\[ Z = \<W, X\>:= \Tr(W^\top X) = \sum_{i=1}^N \sum_{j=1}^M W_{ij} X_{ij},\] 
where now the exchangeability assumption is row-column-exchangeability: $(X_{ij})_{i \in [N], j \in [M]} \eqdst(X_{\pi(i)\pi'(j)})_{i \in [N], j \in [M]}$ for all $\pi \in \permmapset_N, \pi' \in \permmapset_M$. In each of these two problems, the weight vector $w \in \R^N$ or weight matrix $W \in \R^{N \times M}$ are fixed.

We now move to the general formulation of \textit{mode-exchangeable} tensors, which naturally generalizes the above two scenarios and allows for fine-grained analyses of more complex data structures. In this section, $W$ and $X$ both lie in $\tspace:=\R^{N_1 \times \cdots \times N_K}$. We are interested in the concentration of their inner product in the tensor space,
\[Z =   \< W, X \>_{\tspace} := \sum_{i_1=1}^{N_1} \cdots \sum_{i_K=1}^{N_K} W_{i_1\cdots i_K} X_{i_1 \cdots i_K},
\]
when $W$ is a fixed $K$-mode tensor and while $X$ is a random tensor that satisfies mode exchangeability, as defined next.

\begin{definition}[Mode exchangeability] \label{def:mode-ex}
	For a random tensor $X \in \tspace$, we say it is mode-exchangeable if for any $l \in [K]$ and permutation matrix $\Pi_l \in \permset_{N_l}$,
	\begin{align*}
		X \times_l \Pi_l \eqdst X. 
	\end{align*}	
	That is, the distribution of $X$ is invariant under permutation along any mode. Equivalently, for any set of permutation matrices $\Pi_{k} \in \permset_{N_k}, k =1, 2, \cdots , K$,
	\begin{align*}
		(X \times_1 \Pi_1 \times_2 \cdots \times_K \Pi_K) \eqdst X,
	\end{align*}
    i.e., the distribution $X$ is unchanged when we apply a permutation along each mode.
\end{definition}

In this setting, we are interested in concentration inequalities for 
\begin{align*}
    Z:= \<W, X \>_{\tspace} - \E[\<W, X \>_{\tspace} \mid \wb{X}],
\end{align*}
where $\E[\<W, X \>_{\tspace} \mid \wb{X}] = N_1\cdots N_K \cdot \wb{W}\cdot \wb{X}$. To bound the one-sided tail probability $\P\{Z>t\}$, we bound the moment generating function $\E[\exp\{\lambda Z\}]$ for appropriate values of $\lambda$. The remainder of this section develops two such bounds: first a Hoeffding type bound, and then a more refined Bernstein type bound.

Before proceeding, we first need to introduce some additional notation for computing mode-wise averages within a tensor. 
For any tensor $A \in \tspace$, we can define a sequence of tensors $\wb{A}_k \in \mathbb{R}^{N_1 \times \cdots \times N_k} =: \tspace_k$ given by
\[\wb{A}_k := A \times_{k+1} \frac{1}{N_{k+1}}\ones_{N_{k+1}} \times_{k+2} \cdots \times_K \frac{1}{N_K}\ones_{N_K}.\]
In other words, $\wb{A}_k$ is obtained by averaging over modes $k+1,\dots,K$, and has entries
\[(\wb{A}_k)_{i_1\cdots i_k} = \frac{1}{N_{k+1}\cdots N_K}\sum_{i_{k+1}=1}^{N_{k+1}}\hdots\sum_{i_K=1}^{N_K}A_{i_1\cdots i_ki_{k+1}\cdots i_K}.\]
We also define a sequence $A_k \in \tspace$ given by
\[
	A_k  := A \times_{k+1} \proj_{\ones_{N_{k+1}}} \times_{k+2} \cdots \times_K \proj_{\ones_{N_K}}, \]
i.e., we have $(A_k)_{i_1 \cdots i_k i_{k+1} \cdots i_K} = (\wb{A}_k)_{i_1 \cdots i_k}$ for any $i_1,\dots,i_K$. For $k=0$, we additionally have that $\wb{A}_0 = \wb{A} \in \R$ is the average of all entries of $A$, and $A_0 = \wb{A}_0 \cdot \ones_{N_1} \otimes \cdots \otimes \ones_{N_K}$. On the other hand, $A_K = \wb{A}_K = A$.

\subsection{The Hoeffding bound}
As in~\cite{foygel2024hoeffding}, we define a sequence of auxiliary constants $\epsilon_n$ with $\epsilon_1 = 0$, and when $n \geq 2$,
\begin{align*}
	\epsilon_n = \frac{H_n-1}{n - H_n}, \qquad  \qquad H_n := 1 + \frac{1}{2} + \cdots + \frac{1}{n}.
\end{align*}
This constant connects to the advantage term associated with the WoR mean estimator in~\cite{waudby2020confidence}, for which $A_{n} = n/(1+\epsilon_{n+1})$ in the case $t = n$. The following result gives a Hoeffding-type MGF bound for $Z=\<W, X\>_{\tspace}-\E[\<W, X\>_{\tspace} \mid \wb{X}]$.
\begin{theorem}[\textbf{Hoeffding-type MGF bound under mode-exchangeability}] \label{thm:tensor-hoeffding}
	For a fixed $W \in \tspace$ and a mode-exchangeable $X \in \tspace$, it holds for any $\lambda \in \R$ that
	\begin{subequations}
	\begin{align}
		\Ep \brk{\exp \brc{\lambda Z}} \leq \exp \brc{\frac{\lambda^2 }{2} \cdot \prod_{l=1}^K N_l \cdot \sum_{k=1}^{K} \frac{1 + \epsilon_{N_k}}{N_k}  \cdot \sigma^2_{W,k}},\qquad \textnormal{where }\sigma^2_{W, k} := \norm{W_{k} \times_{k} \proj_{\ones_{N_{k}}}^\perp }_{\tspace}^2.
	\end{align}
	Consequently for all $\delta \in (0, 1)$,
	\begin{align}
		\P \brc{Z \geq \sqrt{2 \log \frac{1}{\delta} \cdot \prod_{l=1}^K N_l \cdot \sum_{k=1}^{K} \frac{1 + \epsilon_{N_k}}{N_k} \sigma^2_{W,k} } } \leq \delta.
	\end{align}
	\end{subequations}
\end{theorem}
\noindent The proof of this result is given in Appendix~\ref{proof:tensor-hoeffding}.
\subsection{The Bernstein bound}
We define a sequence of variance parameters $\sigma^2_{X, k}$ and $\wt{\sigma}^2_{X,k}$ for the tensor $X$ for $k \in [K]$,
\begin{align*}
	\sigma^2_{X, k} & = \frac{1}{N_k} \norm{\wb{X}_{k} \times_{k} \proj_{\ones_{N_{k}}}^\perp}_{\tspace_{k}}^2 = \frac{1}{\prod_{l=k}^K N_l} \cdot \norm{X_{k} \times_{k} \proj_{\ones_{N_{k}}}^\perp}_{\tspace}^2, \\
	\wt{\sigma}^2_{X, k} & = \sigma^2_{X, k} + \epsilon_{N_{k}} \cdot \prod_{l=0}^{k-1} N_l \cdot \norm{\wb{X}_{k} \times_{k} \proj_{\ones_{N_{k}}}^\perp}_{\infty}^2 = \sigma^2_{X, k} + \epsilon_{N_{k}} \cdot \prod_{l=0}^{k-1} N_l \cdot \norm{X_{k} \times_{k} \proj_{\ones_{N_{k}}}^\perp}_{\infty}^2,
\end{align*}
analogous to the definition of $\sigma^2_{W,k}$ above. The following Bernstein-type bounds using those sample variance parameters should be viewed as empirical-Bernstein-type inequalities, which can be more favorable than the bounds based on population variance proxies, since they capture the potentially small variations within the realized samples. Define a vector $\wb{w}_k \in \R^{N_k}$ for $k \in [K]$ whose $i$-th coordinate is $\sum_{i_l \in [N_l], l \in [k-1]} \left|(\wb{W}_k \times_k \proj_{\ones_{N_k}}^\perp)_{i_1 \cdots i_{k-1} i}\right|$. For $k\in[K]$, write $c_k := (1+\epsilon_{N_k})\prod_{l=k+1}^K N_l$. We are now ready to present the Bernstein-type MGF bound. The proof is in Appendix~\ref{proof:tensor-bernstein}. 

\begin{theorem}[\textbf{Bernstein-type MGF bound under mode-exchangeability}] \label{thm:tensor-bernstein}
	For a fixed $W \in \tspace$ and a mode-exchangeable $X \in \tspace$, let $\wt{\mc{F}}_X = \sigma(\wt{\sigma}_{X,1}, \cdots, \wt{\sigma}_{X, K})$, then for $|\lambda| < \min_{k \in [K]} \frac{3}{2 c_k \linf{\wb{w}_k} }$,
	\begin{subequations}
	\begin{align}
		&\Ep \brk{\exp \brc{\lambda Z} \mid \wt{\mc{F}}_X}  \leq \exp \brc{\frac{\lambda^2}{2} \cdot \sum_{k=1}^K\frac{ c_k \cdot \sigma^2_{W,k} \cdot \wt{\sigma}_{X,k}^2  }{1 - \frac{2|\lambda|}{3} c_k \linf{\wb{w}_{k}} } } \label{eq:tensor-bernstein-a} 
	\end{align}
	Consequently for all $\delta \in (0, 1)$,
	\begin{align}
		\P \brc{Z \geq \sqrt{2 \log \frac{1}{\delta} \cdot \sum_{k=1}^K  c_k\sigma^2_{W, k} \wt{\sigma}_{X,k}^2  } + \frac{2 \max_{k \in [K]}   c_k \linf{\wb{w}_{k}} }{3} \cdot \log \frac{1}{\delta} }  \leq \delta. \label{eq:tensor-bernstein-c}
	\end{align}
	\end{subequations}
\end{theorem}

We can apply further relaxations to Eqs.~\eqref{eq:tensor-bernstein-a} and~\eqref{eq:tensor-bernstein-c} to simplify the bounds, for greater interpretability. Note that 
$\sigma^2_{W,k} =  \norm{W_{k} \times_{k} \proj_{\ones_{N_{k}}}^\perp }_{\tspace}^2 = \norm{W_{k} }_{\tspace}^2 -  \norm{W_{k} \times_{k} \proj_{\ones_{N_{k}}} }_{\tspace}^2 =  \norm{W_{k} }_{\tspace}^2 -  \norm{W_{k-1} }_{\tspace}^2$, which indicates $\sum_{k=1}^K \sigma_{W,k}^2 = \norm{W}_{\tspace}^2 - \norm{W_0}_{\tspace}^2 \leq \norm{W}_{\tspace}^2$. Therefore
	\[
	\sum_{k=1}^K  c_k\sigma^2_{W, k} \wt{\sigma}_{X,k}^2 
	\leq \norm{W}_{\tspace}^2 \cdot \max_{k \in [K]}  c_k\wt{\sigma}^2_{X, k}.
	\]
This leads immediately to the following corollary:
\begin{corollary}\label{cor:tensor-bernstein}
    In the setting of Theorem~\ref{thm:tensor-bernstein}, it also holds that
    	\begin{align}
		\P \brc{Z \geq \sqrt{2 \log \frac{1}{\delta} \cdot \norm{W}_{\tspace}^2 \cdot \max_{k \in [K]}  c_k\wt{\sigma}^2_{X, k}}+ \frac{2 \max_{k \in [K]}   c_k \linf{\wb{w}_{k}} }{3} \cdot \log \frac{1}{\delta} }  \leq \delta. 
	\end{align}
\end{corollary}

Since both the tensor and the mode-exchangeability assumption are invariant under relabeling, the results apply after any deterministic permutation of the $K$ modes. Therefore, one optimize the deterministic upper bound over all orderings and take the minimum to make such bounds independent of the ordering of the modes.

\subsubsection{Comparing exchangeability and independence: the role of dimension}\label{sec:preview_iid}
To help interpret these results, we next consider the role of the dimensions, $N_1,\dots,N_K$, in these bounds. As a baseline, suppose that the entries of $X$ are i.i.d.\ unbiased signs---that is, $X_{i_1\cdots i_K}\sim\textnormal{Unif}(\{\pm 1\})$. In this case, $Z = \<X,W\>$ has mean zero and variance $\norm{W}^2_{\tspace}$, and thus we have $|Z| = O_P(\norm{W}_{\tspace})$. On the other hand, in Corollary~\ref{cor:tensor-bernstein}, if we consider only the leading term then we have $|Z| = O_P(\norm{W}_{\tspace} \cdot \sqrt{\prod_{k=2}^K N_k})$. 

However, we will now see that this is unavoidable, due to the way that exchangeability is substantially weaker than independence as soon as $K>1$. In particular, consider the following mode-exchangeable tensor: let $Y=(Y_1,\dots,Y_{N_1})$ where $Y_1,\dots,Y_{N_1}\sim\textnormal{Unif}(\{\pm 1\})$ are i.i.d., and let $X = Y\otimes \ones_{N_2}\otimes\cdots\otimes \ones_{N_K}$. Therefore $X$ is constructed with $N_1$ independent samples (rather than $\prod_{k=1}^K N_k$). In particular, $Z = \<X,W\>_{\tspace}$ again has mean zero, but now has variance $\|W_1\|_{\tspace}^2\cdot \prod_{k=2}^K N_k$. This explains why the rate of concentration differs by a factor of $\sum_{k=2}^K N_k$ between the exchangeable and i.i.d.\ cases.

\section{Concentration inequalities for matrix-valued data} \label{sec:master-matrix-valued-bounds}

In this section, we consider concentration inequalities for sequential data of matrix-valued entries. Recall that, in this setting, we work with a sequence of weight matrices $W_1, \cdots, W_N \in \R^{p \times q}$, and random matrices $X_1, \cdots, X_N \in \R^{q \times r}$ that are assumed to be exchangeable---i.e., for any $\pi \in \permmapset_N$, it holds that $(X_1, \cdots, X_N) \eqdst (X_{\pi(1)}, \cdots, X_{\pi(N)})$. Using the tensor notation, we denote by $W \in \R^{N \times p \times q}$ and $X \in \R^{N \times q \times r}$ the tensors $(W_1, \cdots, W_N)$ and $(X_1, \cdots, X_N)$. We then consider the bilinear mapping $\llangle \cdot, \cdot \rrangle: \R^{N \times p \times q} \times \R^{N \times q \times r} \to \R^{p \times r}$,
\[ Z = \llangle W, X \rrangle := \sum_{k=1}^{N} W_k X_k.\]
For any $A \in \R^{N \times N}$, this bilinear mapping satisfies the following property (similar to that of $\<\cdot, \cdot\>_{\tspace}$):
\begin{align*}
	\llangle W \times_1 A, X \rrangle = \sum_{k=1}^N  \prn{\sum_{l=1}^N A_{kl} W_l} X_k = \sum_{k=1}^N W_k \prn{\sum_{l=1}^N A_{lk} X_l}  = \llangle W, X \times_1 A^\top \rrangle.
\end{align*} 
Let $\wb{W} = \frac{1}{N} \sum_{k=1}^N W_k$ and $\wb{X} = \frac{1}{N} \sum_{k=1}^N X_k$. Since $Z$ is now matrix-valued, in this setting we can study the concentration of $Z$ via tail bounds on $\norm{\llangle W, X \rrangle -  N\cdot \wb{W} \;\wb{X} }$ measured in operator norm. As the natural generalization of MGF in the scalar case, we will see that can derive Chernoff-type concentration bounds---including Hoeffding and Bernstein bounds---by studying the matrix MGF defined through trace exponentials (see the following Sec.~\ref{sec:Hoeffding-bound-matrix-valued}). 

\subsection{The Hoeffding bound (general case)} \label{sec:Hoeffding-bound-matrix-valued}
To begin with, we highlight a key difference between the matrix-valued case and the scalar-valued case is that the symmetry breaks. Even if $p=r$, we cannot expect $\llangle W, X \rrangle = \llangle X, W \rrangle$ to hold generically unless $p=q=r=1$. This is due to the non-commutativity of matrix multiplication. As we shall see in the following, this asymmetry introduces some challenges for deriving concentration bounds. Nevertheless, we can obtain a generic Hoeffding-type tail bound without any additional assumptions on the commutativity between $W$ and $X$ as follows.

First we recall some notation for working with matrix-valued data. For a square matrix $A \in \R^{n \times n}$, recall the matrix exponential $\exp A : = \sum_{k=0}^\infty \frac{1}{k!} A^k$.  We define the following normalized exponential operator $\wb{\exp} : \R^{n \times m} \to \R$ as
$$
	\wb{\exp}(A) := \frac{1}{n+m} \Tr \brc{\exp \begin{bmatrix}
		0 & A \\
		A^\top & 0
	\end{bmatrix}}.
$$
The enlarged symmetrized matrix in $\mathbb{S}^{n+m}$ on the right hand side is the dilated matrix of $A$. The so-called dilation trick~\citep{paulsen2002completely} serves as a powerful idea in the operator theory. These tools enable the following bound.
\begin{theorem}[\textbf{Generic Hoeffding-type bound for exchangeable matrix-valued data}] \label{thm:matrix-valued-hoeffding-simple}
	For any positive integers $N, p, q, r$, let $W_1, \cdots, W_N \in \R^{p \times q}$ be a fixed sequence of weight matrices and $X_1, \cdots, X_N \in \R^{q \times r}$ be an exchangeable sequence of random data where $\norm{X_k} \leq 1$ for all $k \in [N]$. 
    Then it holds for any $\lambda \in \R$ that
	\begin{subequations}\begin{align}
		\Ep \brk{\wb{\exp }\brc{\lambda \prn{ \llangle W, X \rrangle -  N\cdot \wb{W}\;\wb{X} }}} \leq \exp \brc{8\lambda^2 (1 + \epsilon_N)^2 \cdot \sum_{k=1}^N \norm{W_k}^2 }, \label{eq:tensor-Hoeffding-MGF-1}
	\end{align}and consequently,
for all $\delta \in (0, 1)$,
	\begin{align}
		\P \brc{\norm{\llangle W, X \rrangle -  N\cdot \wb{W}\;\wb{X}  } \geq 4 (1 + \epsilon_N) \sqrt{2 \log \frac{p + r}{\delta} \cdot \sum_{k=1}^N \norm{W_k}^2 }} \leq \delta.
	\end{align}\end{subequations}
\end{theorem}
This result holds in general, without any commutativity type conditions. However, as we will see next, under additional conditions we can derive a much stronger result---in particular, by improving the term $\sum_{k=1}^N \norm{W_k}^2 $.
\subsection{The Hoeffding bound (under commutativity conditions)}

Although commutativity does not hold for general matrix multipilcations, we can improve the tail bound derived above in settings where we may assume additional constraints in the form of Sylvester equations on the weight matrices $W_k$ and the random matrices $X_k$. We introduce the assumption as follows---and in fact, as we will show once having the assumption stated, these constraints can always be fulfilled by vectorization in an appropriately lifted space. 

\begin{assumption}[Commutativity conditions] \label{assmp:commutativity}
	For the given weight $W \in \R^{N \times p \times q}$, we say the exchangeable matrix-valued sequence $X \in \R^{N \times q \times r}$ satisfy the commutativity constraint, if there exist a positive integer $\wt{q}$, an associated weight $V \in \R^{N \times \wt{q} \times r }$ and an injection $\iota: \R^{q \times r} \to \R^{p \times \wt{q} }$ such that the following holds on the support of $X_1$ (which is the same as other $X_k$'s),
	\begin{align*}
		W_k X_1 = \iota(X_1) V_k =: \wt{X}_1 V_k,  \qquad  \qquad \text{for all }k \in [N].
	\end{align*}
	Denoting by $\wt{X} = (\wt{X}_1, \cdots, \wt{X}_N) \in \R^{N \times p \times \wt{q}}$, we say $(\wt{X}, V)$ is the commutativity pair of $(W, X)$.
\end{assumption}

\begin{remark}
	It is apparent that $\wt{X}$ is also an exchangeable sequence of matrices following from the push-forward distribution of $X$ by $\iota$. Therefore, the concentration of the quantity $\llangle W, X \rrangle$ is equivalent to studying $\llangle \wt{X}, V \rrangle$, reconciling the aforementioned asymmetry of left multiplication on the data. As an illustration of the special case when $q=r$ and $X_k = x_k I_q$ for an exchangeable scalar sequence $(x_1, \cdots, x_N)$, we can set $\wt{q} = p$  for any weight $W$ and $\wt{X}_k = x_k I_p, V_k = W_k$.
\end{remark}

\begin{remark} \label{remark:response-iid}
Suppose that $X_1,\cdots,X_N$ is the finite marginal of an infinitely exchangeable sequence. By de Finetti's theorem, its joint distribution is a mixture of i.i.d.\ distributions. The rectangular matrix Bernstein inequality \cite[Theorem~6.1.1]{tropp2015introduction} applies directly. A Bernstein-type tail bound with the standard spectral variance proxy holds without the commutativity conditions in the infinitely extendible setting, cf.~Cor.~\ref{cor:matrix-valued-ex}. The same strategy does not apply to finite exchangeability, and therefore Assumption~\ref{assmp:commutativity} is required to develop the subsequent theoretical results.
\end{remark}

It is important to point out that the existence of a commutativity pair is not the barrier. For example, one can always take $\iota(X)=I_p\otimes\mathrm{vec}(X)^\top$, which guarantees $W_kX=\iota(X)V_k$ for a suitable $V_k$, since each row of $W_kX$ is linear in $\mathrm{vec}(X)$. See Appendix~\ref{proof:sketching}, Step~1, for an explicit construction in the setting of Corollary~\ref{cor:sketching}. As we shall see in the Hoeffding- and Bernstein-type MGF bounds, it boils down to whether we can find a \textit{economical} pair $(\wt{X}, V)$ such that $\max_{k \in [N]} \norm{\wt{X}_k}, \max_{k \in [N]} \norm{V_k}$ and $\norm{\sum_{k=1}^N V_k^\top V_k}$ are all well-controlled. 

For the above reasons, the commutativity conditions apply much more broadly than simply the case of scalar random sequences in the remark above---we will see several concrete examples in Sections~\ref{sec:theoretical-implications} and~\ref{sec:applications} below.

With this assumption in place, we are now ready to give a more powerful Hoeffding-type bound for the matrix setting.  See its proof and precisely where we make use of the commutativity conditions in Appendix~\ref{proof:matrix-valued-hoeffding}.
\begin{theorem}[\textbf{Hoeffding-type MGF bound for exchangeable matrix-valued data}] \label{thm:matrix-valued-hoeffding}
	For any positive integers $N, p, q, r$, let $W_1, \cdots, W_N \in \R^{p \times q}$ be a fixed sequence of weight matrices and $X_1, \cdots, X_N \in \R^{q \times r}$ be an exchangeable sequence of random data where $\norm{X_k} \leq 1$ for all $k \in [N]$. If Assumption~\ref{assmp:commutativity} holds for the commutativity pair $(\wt{X}, V)$ where $\norm{\wt{X}_k} \leq 1$ for all $k \in [N]$,
    \begin{subequations}
	\begin{align}
		\Ep \brk{\wb{\exp }\brc{\lambda \prn{ \llangle W, X \rrangle -  N\cdot \wb{W}\;\wb{X} }}} & \leq \exp \brc{8 \lambda^2 (1 + \epsilon_N) \cdot \max \brc{\norm{\sum_{k=1}^N W_k W_k^\top}, \norm{\sum_{k=1}^N V_k^\top V_k} }}. \label{eq:tensor-Hoeffding-MGF-2}
	\end{align}
	
	Consequently, for all $\delta \in (0,1)$,
	\begin{align}
		\P \brc{\norm{\llangle W, X \rrangle -  N\cdot \wb{W}\;\wb{X}  } \geq 4  \sqrt{2 \log \frac{p + r}{\delta} \cdot (1 + \epsilon_N) \max \brc{\norm{\sum_{k=1}^N W_k W_k^\top}, \norm{\sum_{k=1}^N V_k^\top V_k} } }} \leq \delta.
	\end{align}
	\end{subequations}
\end{theorem}

To compare these bounds to the results of the previous theorem, we can see that the tail bounds established in Theorem~\ref{thm:matrix-valued-hoeffding} rely on $\max\{\norm{\sum_{k=1}^N W_k W_k^\top},\norm{\sum_{k=1}^N V_k^\top V_k}\}$, rather than $\sum_{k=1}^N\norm{W_k}^2$ as in Theorem~\ref{thm:matrix-valued-hoeffding-simple}. This is potentially a much tighter bound: indeed, we always have 
\[\max\left\{\norm{\sum_{k=1}^N W_k W_k^\top},\norm{\sum_{k=1}^N V_k^\top V_k}\right\} \leq \max\brc{\sum_{k=1}^N\norm{W_k}^2, \sum_{k=1}^N \norm{V_k}^2}\]
and the gap in this inequality may be as large as a factor of $O(N)$ (e.g., if each $W_k$ and $V_k$ are rank-1 matrices). However, the commutativity conditions assumed in Theorem~\ref{thm:matrix-valued-hoeffding} can be restrictive, where as Theorem~\ref{thm:matrix-valued-hoeffding-simple} always applies offering greater flexibility.

\subsection{The Bernstein bound (under commutativity conditions)}
We then move on to deriving Bernstein-type bounds for exchangeable matrix-valued data when the commutativity conditions hold. Let
\begin{align*}
	\Sigma = \frac{1}{N} \sum_{k=1}^N (X_k  - \wb{X}) (X_k - \wb{X})^\top, \qquad \wt{\Sigma} = \frac{1}{N} \sum_{k=1}^N (\wt{X}_k - \wb{\wt{X}})^\top (\wt{X}_k - \wb{\wt{X}}),
\end{align*} 
and we can define the variance parameters $\sigma_X^2 = \norm{\Sigma}$ and $\sigma_{\wt{X}}^2 = \norm{\wt{\Sigma}}$, as well as
\begin{align*}
	\wt{\sigma}_X^2 = \sigma_X^2 + \epsilon_N \max_{k \in [N]} \norm{X_k - \wb{X}}^2, \qquad \qquad \wt{\sigma}_{\wt{X}}^2 = \sigma_{\wt{X}}^2 + \epsilon_N \max_{k \in [N]} \norm{{\wt{X}}_k - \wb{{\wt{X}}}}^2.
\end{align*}

\begin{theorem}[\textbf{Bernstein-type MGF bound for exchangeable matrix-valued data}] \label{thm:matrix-valued-bernstein}
	Let Assumption~\ref{assmp:commutativity} hold for $N \geq 2$. If $\norm{X_k} \leq 1$ and $\norm{\wt{X}_k} \leq 1$ for all $k \in [N]$, we have for $w = (\norm{W_1}, \cdots, \norm{W_N})^\top$, $v = (\norm{V_1}, \cdots, \norm{V_N})^\top$ and $|\lambda| < \frac{3}{2 \min \{\norm{w}_\infty, \norm{v}_\infty\} (1 + \epsilon_N)}$ that
	\begin{subequations}
	\begin{align}
		&\Ep \brk{\wb{\exp }\brc{\lambda \prn{ \llangle W, X \rrangle -  N\cdot \wb{W}\;\wb{X} }} \mid \wt{\sigma}_X^2, \wt{\sigma}_{\wt{X}}^2} & \nonumber \\
		& \leq \exp \brc{\frac{\lambda^2 ( 1+ \epsilon_N)^2}{2\prn{1 - \frac{2|\lambda|(1 + \epsilon_N)}{3}  \min \{\norm{w}_\infty, \norm{v}_\infty\} }} \cdot \max \brc{\wt{\sigma}_X^2 \norm{\sum_{k=1}^N W_k W_k^\top}, \wt{\sigma}^2_{\wt{X}} \norm{\sum_{k=1}^N V_k^\top V_k} }}.
	\end{align}
	Consequently, for all $\delta \in (0, 1)$,
	\begin{align}
	& \P \Bigg\{\norm{\llangle W, X \rrangle -  N\cdot \wb{W}\;\wb{X}  } \geq (1 + \epsilon_N) \sqrt{2 \log \frac{p + r}{\delta} \cdot\max \brc{\wt{\sigma}_X^2 \norm{\sum_{k=1}^N W_k W_k^\top}, \wt{\sigma}^2_{\wt{X}} \norm{\sum_{k=1}^N V_k^\top V_k} } } \nonumber \\
	& \qquad + \frac{2 \min \{\norm{w}_\infty, \norm{v}_\infty\} }{3}  (1 + \epsilon_N) \log \frac{p+r}{\delta} \Bigg\} \leq \delta.
\end{align}
	\end{subequations}
\end{theorem}
We defer its proof to Appendix~\ref{proof:matrix-valued-bernstein}.

\section{Theoretical implications} \label{sec:theoretical-implications}
In this section, we examine several implications of our general results, deriving bounds for a range of different special cases.

\subsection{Weighted sum of exchangeable random variables}
We first apply our results in Sec.~\ref{sec:master-tensor-bounds} and \ref{sec:master-matrix-valued-bounds} where the random terms consist of an exchangeable sequence of random variables $\xi_1, \cdots, \xi_N \in [-1, 1]$. Specifically, we consider the following general weighted sum of exchangeable random variables, where the weights may be scalar-valued or matrix-valued:
\begin{align*}
	&\textbf{\textrm{(Scalar-valued weights)}} & & \<w, \xi\> = \sum_{k=1}^N w_k \xi_k, \\
	& \textbf{\textrm{(Matrix-valued weights)}} & & \llangle W, \xi I_r\rrangle = \sum_{k=1}^N W_k \xi_k.
\end{align*}
To relate these sums to our general notation from before, in the first case we simply have $X_k = \xi_k$, while in the second case we are working with matrix-valued random variables $X_k = \xi_k I_r$.

Let $\wb{\xi} := \frac{1}{N} \sum_{k=1}^N \xi_k$, $\sigma^2_{\xi} := \frac{1}{N} \sum_{k=1}^N (\xi - \wb{\xi})^2$ and $\wt{\sigma}^2_{\xi} := \sigma^2_{\xi} + \epsilon_N \norm{\xi - \wb{\xi}}_\infty^2 \leq \sigma^2_{\xi} + 4 \epsilon_N$. For the scalar-valued weights case, we apply Theorems~\ref{thm:tensor-hoeffding} and \ref{thm:tensor-bernstein} with $K=1$. Note that these bounds exactly recover (with the same constants) the main results in \citet[Thm.~3.1 \& 3.3]{foygel2024hoeffding}.

\begin{corollary}[\textbf{Scalar weighted sum of exchangeable r.v.s}]\label{cor:scalar} For $w=(w_1, \cdots, w_N)^\top \in \R^N$ and all $\delta \in (0,1)$, assume $\wb{w} = \frac{1}{N}\sum_{k=1}^N w_k = 0$. 
\begin{enumerate}
	\item \textbf{(Hoeffding bound)} For any $\lambda \in \R$, it holds that
	\begin{align*}
		\Ep \brk{\exp \brc{\lambda \<w, \xi\>}} \leq \exp \brc{\frac{\lambda^2 (1 + \epsilon_N)}{2} \cdot \norm{w}^2}, 
	\end{align*}
	and consequently $\P \brc{\<w, \xi\> \geq \sqrt{2 (1+\epsilon_N) \log 1/\delta} \cdot \norm{w} } \leq \delta$.
	\item \textbf{(Bernstein bound)} In addition, we have for $|\lambda| < \frac{3}{2 \linf{w}(1 + \epsilon_N)}$
	\begin{align*}
		\Ep \brk{\exp \brc{\lambda \<w, \xi\>} \mid \wt{\sigma}^2_{\xi}} \leq \exp \brc{\frac{\lambda^2 (1 + \epsilon_N)}{2 \prn{1 - \frac{2|\lambda|}{3} \linf{w}(1+\epsilon_N)} } \cdot \wt{\sigma}^2_{\xi} \norm{w}^2}, 
	\end{align*}
	and consequently $\P \brc{\<w, \xi\> \geq \sqrt{2 (1+\epsilon_N) \log 1/\delta} \cdot \wt{\sigma}_\xi \norm{w} + \frac{2}{3} \linf{w} (1+\epsilon_N) \log 1/\delta } \leq \delta$.
\end{enumerate}
\end{corollary}
\begin{proof}
	We take $\tspace = \R^N, W = w - \wb{w} \ones_N$ and $X= \xi$ in Theorems~\ref{thm:tensor-hoeffding} and \ref{thm:tensor-bernstein}.
\end{proof}

Next we apply Theorems~\ref{thm:matrix-valued-hoeffding} and \ref{thm:matrix-valued-bernstein} to a sequence of fixed matrices $W_1, \cdots, W_N$. The result that follows recovers the Hoeffding bound in~\citet[Cor.~7.2]{tropp2012user} and Bernstein bound in~\citet[Thm.~6.1.1]{tropp2015introduction} up to a multiplicative factor $(1+\epsilon_N)$, as \citeauthor{tropp2012user}'s results are on independent random variables. (In the next section, we discuss how to recover transport exchangeable results to i.i.d.\ results, which removes the $(1+\epsilon_N)$ factor.)

The numerical constant for the Hoeffding bound however, comparing to the scalar-valued weights case, is indeed improvable provided $\xi_k \eqdst - \xi_k$. We refer the readers to~\cite[Remark~7.4]{tropp2012user} for further discussions. In addition, in the work of~\citeauthor{mackey2014matrix}, the authors employ the method of exchangeable pairs to establish a Hoeffding-type inequality with the optimal constant, without distributional symmetry given i.i.d.\ data. The numerical constant for the Bernstein bound is indeed optimal.

\begin{corollary}[\textbf{Matrix weighted sum of exchangeable r.v.s}] \label{cor:matrix-valued-ex} For $W_1, \cdots, W_N \in \R^{p \times r}$ and all $\delta \in (0,1)$, assume $\wb{W} = \frac{1}{N}\sum_{k=1}^N W_k = 0$ and let $\sigma^2_W := \max \{\norm{\sum_{k=1}^N W_k W_k^\top}, \norm{\sum_{k=1}^N W_k^\top W_k}\}$.
	\begin{enumerate}
		\item \textbf{(Hoeffding bound)} For any $\lambda \in \R$, it holds that
		\begin{align*}
			\Ep \brk{\wb{\exp} \brc{\lambda \sum_{k=1}^N W_k \xi_k}} \leq \exp \brc{8\lambda^2 (1 + \epsilon_N) \cdot \sigma_W^2}, 
		\end{align*}
		and consequently $\P \brc{\norm{\sum_{k=1}^N W_k \xi_k} \geq 4\sqrt{2 (1+\epsilon_N) \log \frac{p+r}{\delta}} \cdot \sigma_W} \leq \delta$.
		\item \textbf{(Bernstein bound)} Additionally, suppose $\norm{W_k} \leq L$ for all $k \in [N]$. We have for $|\lambda| < \frac{3}{2 L(1 + \epsilon_N)}$
		\begin{align*}
			\Ep \brk{\wb{\exp} \brc{\lambda \sum_{k=1}^N W_k \xi_k} \mid \wt{\sigma}^2_{\xi}} \leq \exp \brc{\frac{\lambda^2 (1 + \epsilon_N)}{2 \prn{1 - \frac{2|\lambda|L}{3} (1+\epsilon_N)} } \cdot \wt{\sigma}^2_{\xi} \sigma_W^2}, 
		\end{align*}
		and consequently
		\begin{align*}
			\P \brc{\norm{\sum_{k=1}^N W_k \xi_k} \geq \sqrt{2 (1+\epsilon_N) \log \frac{p+r}{\delta}} \cdot  \wt{\sigma}_\xi \sigma_W  + \frac{2L}{3} (1+\epsilon_N) \log \frac{p+r}{\delta} } \leq \delta.
		\end{align*}
	\end{enumerate}
\end{corollary}
\begin{proof}
	We first verify Assumption~\ref{assmp:commutativity} by taking $X_k = \xi_k I_r$, $\wt{X}_k = \xi_k I_p$ and $V_k = W_k$. The MGF inequalities then follow from Theorems~\ref{thm:matrix-valued-hoeffding} and \ref{thm:matrix-valued-bernstein}. The tail bounds are consequences of Lemma~\ref{lem:MGF-tail-matrix}.
\end{proof}

\subsection{Connections to the i.i.d.\ setting}
Next, we explore the connection between our results derived under the assumption of exchangeability, and the more standard setting of i.i.d.\ data. 

To introduce this question, first recall the the scalar setting studied in~\cite{foygel2024hoeffding}, where the aim is to establish concentration of the weighted sum $w_1X_1 + \dots + w_NX_N$ for exchangeable random variables $X_1,\dots,X_N\in\R$. It is shown in~\cite{foygel2024hoeffding} that if $X_1,\dots,X_N$ can be embedded into longer exchangeable sequence $X_1,\dots,X_{\wt{N}}$ (for $\wt{N}\geq N$), then the concentration bounds of Corollary~\ref{cor:scalar} hold with $\epsilon_{\wt{N}}$ in place of $\epsilon_N$. At the extreme, if the $X_i$'s are i.i.d., we can take $\wt{N}$ to be arbitrarily large; since $\epsilon_{\wt{N}}\to0$ as $\wt{N}\to\infty$, this means that the results of Corollary~\ref{cor:scalar} recover the usual Hoeffding and Bernstein bounds for sums of independent random variables. In this section, we will see that the same type of conclusion holds for the tensor and matrix settings considered in this paper.

It is worth pointing out that merely $N \to \infty$ is insufficient to approximate a mixture-of-i.i.d.\ distribution, as de Finneti's theorem only applies to infinitely exchangeable sequence. Indeed, by the finite de Finitte's theorem~\cite[Thm.~1]{diaconis1980finite}, one typically will need the embedded exchangeable sequence $(X_1, \cdots, X_{\wt{N}})$ to be much longer, i.e.\ $\wt{N} \gg N^2$ to ensure small total variation error between $(X_1, \cdots, X_N)$ and the family of mixture-of-i.i.d.\ distributions of length $N$. Otherwise, a general finite exchangeable distribution of length $N$ can have constant error from a mixture-of-i.i.d.\ distribution.

\subsubsection{Comparing to the i.i.d.\ case for mode-exchangeable tensors}
As previewed in Section~\ref{sec:preview_iid}, for the setting of a mode-exchangeable tensor, we can compare to a setting where the data is ``i.i.d.\ along one mode''---that is, we sample i.i.d.\ $(K-1)$-mode tensors.
Specifically, suppose that $X \in [-1,1]^{N_1 \times \cdots \times N_K}$ is i.i.d.\ along the $K$-th mode: that is, the $(K-1)$-mode tensors
\[\mathscr{X}_l:= X \times_K e_l = (X_{i_1\dots i_{K-1}l})_{i_1\in[N_1],\dots,i_{K-1}\in[N_{K-1}]}\in\tspace_{K-1}\]
are i.i.d.\ for each $l=1,\dots,N_K$ following a probability distribution $P$ on $\tspace_{K-1}$, whose real-valued mean and variance parameters are
\begin{align*}
    \mu_P & = \E_{\mathscr{X} \sim P} [\wb{\mathscr{X}}] , \qquad \sigma_P^2  = \E_{\mathscr{X} \sim P} \brk{\norm{\mathscr{X} - \E_{\mathscr{X} \sim P} [\mathscr{X}]}_{\tspace_{K-1}}^2}.
\end{align*}
Similarly, we can define the $N_K$ weight tensors $\mathscr{W}_1, \cdots, \mathscr{W}_{N_K} \in \tspace_{K-1}$. In this case, the relevant inner product can be written as a sum of $N_K$ independent terms as
\begin{align*}
    Z = \langle W, X \rangle_{\tspace_K} = \sum_{l=1}^{N_K} \langle \mathscr{W}_l, \mathscr{X}_l \rangle_{\tspace_{K-1}} .
\end{align*}
Let $w \in \R^{N_K}$ be the vector whose $i$-th coordinate is $\sum_{i_l \in [N_l], l \in[K-1]} |W_{i_1 \cdots i_{K-1} i}| = \lone{\mathscr{W}_i}$. In order to invoke classical results concerning sums of independent random variables, we observe that the $l$-th summand $\langle \mathscr{W}_l, \mathscr{X}_l \rangle_{\tspace_{K-1}}$ satisfies $\left|\langle \mathscr{W}_l, \mathscr{X}_l \rangle_{\tspace_{K-1}}\right| \leq \lone{\mathscr{W}_l} \cdot \linf{\mathscr{X}_l} \leq w_l$ and
$$\var(\langle \mathscr{W}_l, \mathscr{X}_l \rangle_{\tspace_{K-1}}) = \E \brk{\prn{\langle \mathscr{W}_l, \mathscr{X}_l - \E\brk{\mathscr{X}_l}\rangle_{\tspace_{K-1}}}^2} \leq \norm{\mathscr{W}_l}_{\tspace_{K-1}}^2 \cdot \sigma_P^2.$$
By applying the classical Hoeffding inequality for sums of independent random variables, we obtain for all $\lambda \in \R$,
\begin{align*}
\E[\exp\{\lambda(Z-\E[Z])\}] & \leq \exp \brc{\frac{\lambda^2 }{2} \cdot \sum_{l=1}^{N_K} \lone{\mathscr{W}_l}^2} ,\end{align*}
while  the classical Bernstein bound yields for all $|\lambda| < \frac{3}{2 \linf{w}}$ the following MGF bound,
\begin{align*}
\E[\exp\{\lambda(Z-\E[Z])\}] & \leq \exp \brc{ \frac{\lambda^2 }{2 \prn{1 - \frac{2|\lambda|}{3} \linf{w}}} \cdot \sum_{l=1}^{N_K} \norm{\mathscr{W}_l}_{\tspace_{K-1}}^2 \cdot \sigma_P^2}.
\end{align*}
To better relate to our bounds, we apply the Cauchy-Schwarz inequality $\lone{\mathscr{W}_l}^2 \leq \prod_{l=1}^{K-1} N_l \cdot \norm{\mathscr{W}_l}_{\tspace_{K-1}}^2$ and the identity $\sum_{l=1}^{N_K} \norm{\mathscr{W}_l}_{\tspace_{K-1}}^2 = \norm{W}_{\tspace}^2$ to arrive at the following.
\begin{subequations}
\begin{align}
    \E[\exp\{\lambda(Z-\E[Z])\}] & \leq \exp \brc{\frac{\lambda^2 }{2} \cdot \prod_{l=1}^{K-1} N_l \cdot \norm{W}_{\tspace}^2} & & \text{for all }\lambda \in \R, \label{eqn:hoeffding_iid_tensor} \\
    \E[\exp\{\lambda(Z-\E[Z])\}] & \leq \exp \brc{ \frac{\lambda^2 }{2 \prn{1 - \frac{2|\lambda|}{3} \linf{w}}} \cdot \norm{W}_{\tspace}^2 \sigma_P^2} & & \text{for all }|\lambda| < \frac{3}{2\linf{w}}. \label{eqn:bernstein_iid_tensor}
\end{align}
\end{subequations}
\subsubsection{Interpolation between exchangeability and i.i.d.\ bounds} 
Next we will see how our results, under exchangeability, agree with these bounds in the limit. Suppose that $\wt{X}\in\R^{N_1\times\dots\times N_{K-1}\times \wt{N}_K} =: \wt{\tspace}$ is mode-exchangeable for some $\wt{N}_K\geq N_K$, and $X\in\R^{N_1\times\dots\times N_K}$ is the appropriate sub-tensor of $\wt{X}$. We will also need an additional assumption:
\begin{definition}[Balanced weights along mode-$K$]
	We say $W$ is balanced along mode-$K$ if $\wb{W}_{K-1} = \wb{W} \ones_{N_1} \otimes \cdots \otimes \ones_{N_{K-1}}$.
\end{definition}
Then, assuming that $W$ is balanced along mode-$K$, the MGF bound in Theorem~\ref{thm:tensor-hoeffding} admits significant simplification through
$$
    \sigma_{W,k} = \norm{W_k \times_k \proj_{\ones_{N_k}}^\perp}_{\tspace}^2 = \norm{\wb{W}_{K-1} \times_k \proj_{\ones_{N_k}}^\perp}_{\tspace_{K-1}}^2 = 0, \qquad \textnormal{for all }k \in [K-1].
$$
and thus (see details of the reduction from $\wt{\tspace}$ to $\tspace$ in Appendix~\ref{proof:tensor-iid})
$$
\E[\exp\{\lambda(Z-\E[Z\mid\wb{\wt{X}}])\}]  \leq \exp \brc{\frac{\lambda^2 }{2} \cdot \prod_{l=1}^{K-1} N_l \cdot (1+\epsilon_{\wt{N}_K})\norm{W}_{\tspace}^2}\qquad \textnormal{for all }\lambda \in \R.
$$
In particular, taking $\wt{N}_K\to\infty$ (i.e., if $X$ is i.i.d.\ along the $K$-th mode), we see that this exactly recovers the classical calculation~\eqref{eqn:hoeffding_iid_tensor}, with the correct constants. Similarly, for the Bernstein bound conditional on $\wt{\sigma}_{\wt{X},K}^2$ for the tensor case and $\wt{\sigma}_{\xi}^2$ for the matrix case below, assuming again that $W$ is balanced along mode-$K$, Theorem~\ref{thm:tensor-bernstein} yields for all $|\lambda| < \frac{3}{2(1+\epsilon_{\wt{N}_K})\linf{\wb{w}_K}}$ the following MGF bound (with details in Appendix~\ref{proof:tensor-iid})
\begin{align*}
     \E[\exp\{\lambda(Z-\E[Z\mid\wb{\wt{X}}])\}]  \leq \exp \brc{ \frac{\lambda^2 \cdot (1 + \epsilon_{\wt{N}_K}) }{2 \prn{1 - \frac{2|\lambda|}{3} \linf{\wb{w}_K} (1 + \epsilon_{\wt{N}_K})}} \cdot \norm{W}_{\tspace}^2 \wt{\sigma}_{\wt{X}, K}^2}, 
\end{align*}
which again exactly recovers the classical calculation~\eqref{eqn:bernstein_iid_tensor}, with the correct constants, if we take $\wt{N}_K\to\infty$ and the convergence of sample variance $\wt{\sigma}_{\wt{X}, K}^2 \to \sigma_P^2$ almost surely.

In other words, we can interpret our results as providing a natural interpolation between the exchangeable and i.i.d.\ settings, where if the dimension $\wt{N}_K$ along mode-$K$ is simply equal to $N_K$ (the finite exchangeable setting) we have a weaker concentration bound, while if $\wt{N}_K\to\infty$ (which limits to the setting of data that is ``i.i.d.\ along mode-$K$''), then we recover the standard bounds for i.i.d.\ data. Notably, \citeauthor{ramdas2026randomized} demonstrate that, under the assumption of infinite exchangeability, it is possible to obtain a uniform bound for all $N_K \geq M$, for an arbitrarily prescribed threshold $M$, as established in~\cite{ramdas2026randomized}. This observation suggests the potential for strengthening our results to ensure uniform control within our framework.  
\begin{remark}
    Indeed, if $W$ is balanced along mode-$K$, letting $\mu_P := \E_{\mathscr{X} \sim P}[\wb{\mathscr{X}}] \in [-1, 1]$, $Z-\E Z$ admits the following simplification
    $$Z - \E Z = \sum_{l=1}^{N_K} \<\mathscr{W}_l, \mathscr{X}_l - \E_{\mathscr{X} \sim P}[\mathscr{X}] \>_{\tspace_{K-1}} = \<W, X - \mu_P \ones_{\tspace}\>_{\tspace}.$$
    We defer the details to Appendix~\ref{proof:tensor-iid}.
\end{remark}

\subsubsection{Connections to the i.i.d.\ matrix concentration inequalities}
Similarly, we can generalize Theorems~\ref{thm:matrix-valued-hoeffding} and \ref{thm:matrix-valued-bernstein} by embedding into a longer exchangeable sequence of random matrices $X_1, X_2, \cdots, X_{\wt{N}}$ with or without the commutativity conditions, and approaching i.i.d.\ results by taking $\wt{N} \to \infty$. As the general case provides limited insight with heavy notations, for simplicity, we only sketch the generalization to Cor.~\ref{cor:matrix-valued-ex} in the previous section where we have $W_1, \cdots, W_N \in \R^{p \times r}$ and $X_k = \xi_k I_r$ in what follows. First we recall the i.i.d.\ case where $\xi_1, \cdots, \xi_N \simiid P$ on $[-1, 1]$. For $Z=\llangle W, X \rrangle = \sum_{k=1}^N \xi_k W_k$, \citeauthor{tropp2015introduction}'s well-established i.i.d.\ results~\citep{tropp2012user, tropp2015introduction} imply that for the parameter
\begin{align*}
    \sigma^2_W = \max \brc{\norm{\sum_{k=1}^N W_k W_k^\top}, \norm{\sum_{k=1}^N W_k^\top W_k}}, \qquad L =  \max_{k \in [N]} \norm{W_k},
\end{align*}
we have
\begin{subequations}
\begin{align}
    \E[\wb{\exp}\{\lambda(Z-\E[Z])\}] & \leq \exp \brc{8\lambda^2 \cdot \sigma_W^2} & & \text{for all }\lambda \in \R, \label{eqn:hoeffding_iid_matrix} \\
    \E[\wb{\exp}\{\lambda(Z-\E[Z])\}] & \leq \exp \brc{ \frac{\lambda^2 }{2 \prn{1 - \frac{2|\lambda|L}{3}}} \cdot \sigma_W^2 \var(\xi)} & & \text{for all }|\lambda| < \frac{3}{2L}. \label{eqn:bernstein_iid_matrix}
\end{align}
\end{subequations}
Now we see given only exchangeability, but if we can embed $\xi_1, \cdots, \xi_N$ into a longer sequence of length $\wt{N}$, our results in Cor.~\ref{cor:matrix-valued-ex} (with $\wb{W}=0$) immediately imply
\begin{align*}
    \E[\wb{\exp}\{\lambda(Z-\E[Z])\}] & \leq \exp \brc{8\lambda^2 (1 + \epsilon_{\wt{N}})\cdot \sigma_W^2} & & \text{for all }\lambda \in \R, \\
    \E[\wb{\exp}\{\lambda(Z-\E[Z])\}] & \leq \exp \brc{ \frac{\lambda^2 (1 + \epsilon_{\wt{N}})}{2 \prn{1 - \frac{2|\lambda|L}{3}(1 + \epsilon_{\wt{N}})}} \cdot \sigma_W^2 \wt{\sigma}_{\xi}^2} & & \text{for all }|\lambda| < \frac{3}{2(1+\epsilon_{\wt{N}})L}. 
\end{align*}
Taking $\wt{N}\to\infty$ and we have $\epsilon_{\wt{N}} \to 0$ and $\wt{\sigma}_\xi^2 \to \var(\xi)$ almost surely. We therefore recover the standard i.i.d.\ matrix concentration inequalities in~\eqref{eqn:hoeffding_iid_matrix} and~\eqref{eqn:bernstein_iid_matrix}, again with the correct constants.

\subsection{Combinatorial sum of matrix arrays} \label{sec:combinatorial-array}
As the final theoretical demonstration of our tools, we revisit the combinatorial sum of matrices~\citep{mackey2014matrix, han2024introduction}. \citeauthor{mackey2014matrix} considered a deterministic array of Hermitian matrices $\{A_{i,j}\}_{i, j \in [N]}$, and studied the concentration of the random matrix $\sum_{k=1}^N A_{k, \pi(k)}$, using the celebrated Chatterjee's method of exchangeable pairs~\citep{chatterjee2007stein} (first appeared in \citeauthor{chatterjee2005concentration}'s PhD thesis~\citep{chatterjee2005concentration}). We generalize to the rectangular case, assuming $A_{i,j} \in \R^{m \times n}$. Overloading the tensor product $\otimes$ in the previous sections by Kronecker product of matrices. For $A = (a_{ij}) \in \R^{m \times n}, B = (b_{ij}) \in \R^{m' \times n'}$, let $A \otimes B \in \R^{mm' \times nn'}$ be
\begin{align*}
	A \otimes B := \begin{bmatrix}
		a_{11} B & \cdots & a_{1n} B \\
		\vdots & \ddots & \vdots \\
		a_{m1} B & \cdots & a_{mn} B
	\end{bmatrix}.
\end{align*}
We make the crucial observation that we can find commutativity pair $(W, X)$ and $(\wt{X}, V)$ satisfying Assumption~\ref{assmp:commutativity}, such that $\sum_{k=1}^N A_{k, \pi(k)} = \sum_{k=1}^N W_k X_k = \sum_{k=1}^N \wt{X}_k V_k$. Concretely, 
\begin{align*}
	& W_k = \begin{bmatrix} A_{k,1} & \cdots & A_{k, N} \end{bmatrix} \in \R^{m \times nN}, & & X_k = e_{\pi(k)} \otimes I_n \in \R^{nN \times n}; \\
	&\wt{X}_k = e_{\pi(k)}^\top \otimes I_m \in \R^{m \times mN}, & & V_k = \begin{bmatrix} A_{k,1} \\ \vdots \\ A_{k, N} \end{bmatrix} \in \R^{mN \times n}.
\end{align*}
$X_1, \cdots, X_N$ are then naturally a sequence of exchangeable matrices. We now present the Bernstein-type concentration inequalities for the combinatorial sum of matrix arrays.
\begin{corollary}[\textbf{Combinatorial matrix Bernstein inequality}] \label{cor:matrix-valued-combinatorial} 
	Suppose the array of matrices $\{A_{i,j}\}_{i, j \in [N]}$ satisfies $\sum_{i, j \in [N]} A_{i,j} = 0$ and $\norm{A_{i,j}} \leq R$. It holds for 
	\begin{align*}
		\sigma^2 = \max \brc{\norm{\sum_{i,j \in [N]} A_{i, j} A_{i, j }^\top}, \norm{\sum_{i,j \in [N]} A_{i, j}^\top A_{i, j }}}
	\end{align*}
	that when $|\lambda| < \frac{3}{2(1+\epsilon_N) R}$,
	\begin{align*}
		&\Ep \brk{\wb{\exp }\brc{\lambda \sum_{k=1}^N A_{k, \pi(k)}} } \leq \exp \brc{\frac{\lambda^2 ( 1+ \epsilon_N)^2}{2\prn{1 - \frac{2|\lambda|(1 + \epsilon_N)  R}{3}}} \cdot \prn{\frac{1}{N} + \epsilon_N \cdot \prn{1 - \frac{1}{N}}} \sigma^2},
	\end{align*}
	and consequently for $\delta \in (0, 1)$,
	\begin{align*}
	\P \brc{\norm{\sum_{k=1}^N A_{k, \pi(k)}} \geq (1+\epsilon_N) \sigma \sqrt{2 \prn{\frac{1}{N} + \epsilon_N \cdot \prn{1 - \frac{1}{N}}} \log \frac{m+n}{\delta}} + \frac{2 (1+\epsilon_N) R}{3} \log \frac{m+n}{\delta}}  \leq \delta.
	\end{align*}
\end{corollary}
We defer the proof to Appendix~\ref{proof:matrix-valued-combinatorial}. This result recovers~\citet[Cor.~10.3]{mackey2014matrix} up to an additional $O(\sqrt{\log N})$ factor in its variance term, and with a better constant in the bias term. The combinatorial matrix Bernstein inequality also yields a byproduct of a weaker Berstein-type tail bound immediately that requires no commutativity conditions. See discussions for such results in Appendix~\ref{proof:matrix-valued-combinatorial}.

\section{Applications} \label{sec:applications}

We next return to the two example applications introduced in Section~\ref{sec:intro}, to explore the implications of our results in these concrete settings. Section~\ref{sec:applications_example1} revisits Example~\ref{example:avg-effect}, while Sections~\ref{sec:dst-sketching} and~\ref{sec:rtfa} develop the problems introduced in Example~\ref{example:sketching}. See codes for experiments in this section here\footnote{\hyperlink{blue}{https://github.com/Moriartycc/exchangeable-tensors-matrix-valued-data}}. 

\subsection{Experimental design for estimating multi-factor average effect}\label{sec:applications_example1}
We revisit Example~\ref{example:avg-effect}, where our goal is to estimate a weighted mean
$$\mu = \sum_{i_l \in [N_l], l \in [K]} p_{i_1 \cdots i_K} Y_{i_1 \cdots i_K}$$
while only subsampling a block of the tensor $Y$: we will only observe $Y_{i_1\cdots i_K}$ for indices $(i_1,\dots,i_K)\in I_1\times\dots\times I_K$. We will use the Horvitz-Thompson estimator $\hat{\mu}$ as in Example~\ref{example:avg-effect}. By construction, $\what{\mu}$ has mean $\mu$, and variance
$$\sigma^2 = \prod_{l=1}^K N_l \cdot \sum_{i_l \in [N_l], l \in [K]} p_{i_1 \cdots i_K}^2 Y_{i_1 \cdots i_K}^2 - \mu^2.
$$
\paragraph{Tail bounds for the estimation error} The next estimation error bound then follows from Theorem~\ref{thm:tensor-bernstein}. See its proof in Appendix~\ref{proof:avg-effect}.
\begin{corollary} \label{cor:avg-effect}
	For $k \in [K]$, let $B_k > 0$ be that for all $i_k \in [N_k]$, $\sum_{i_l \in [N_l], l \neq k} p_{i_1 \cdots i_K} |Y_{i_1 \cdots i_K}| \leq B_k$, and $\alpha_k \in (0,1]$ be the sub-sampling rates such that $\prod_{l=1}^k n_l = \alpha_k \prod_{l=1}^k N_l$. For all $\delta \in (0,1)$, 
	\begin{subequations}
	\begin{align}
		\P \Bigg\{\what{\mu} - \mu &\geq  \sigma \sqrt{2 \log \frac{1}{\delta} \cdot \max_{k \in [K]} \frac{1+ \epsilon_{N_{k}}}{N_k \alpha_k} \cdot \prn{1 - \frac{n_k}{N_k}+ \frac{\epsilon_{N_k}}{\alpha_k} \cdot \ind_{n_k < N_k}}} \nonumber \\
		&\hphantom{{}\geq{}} {}+ \frac{2 \max_{k \in [K]} B_k(1 + \epsilon_{N_k})}{3 \alpha_K} \log \frac{1}{\delta} \Bigg\}  \leq \delta. \label{eq:tail-bound-avg-effect-general}
	\end{align}

	For an $\alpha \in (0,1)$ such that $N_K\alpha \geq 1$ is an integer, we have the following tail bound under the Cartesian sub-sampling configuration $(n_1, \cdots, n_K) = (N_1, \cdots, N_{K-1}, N_K \alpha )$:
	\begin{align} \label{eq:avg-effect-one-mode}
		& \P \Bigg\{\what{\mu} - \mu \geq  \sigma \sqrt{2 \log \frac{1}{\delta} \cdot \frac{1+ \epsilon_{N_{K}}}{N_K \alpha} \cdot \prn{1 - \alpha + \frac{\epsilon_{N_K}}{\alpha}}}   +  \frac{2 \max_{k \in [K]} B_k(1 + \epsilon_{N_k})}{3 \alpha} \log \frac{1}{\delta} \Bigg\}  \leq \delta.
	\end{align}
	\end{subequations}
\end{corollary}
Informally speaking (assuming the variance term dominates in the preceding tail bounds), we can infer from Cor.~\ref{cor:avg-effect} that under a sub-sampling rate constraint $\alpha \geq 1/N_K$, we can control deviation with $1-\delta$ probability not exceeding
\begin{align*}
	O \prn{\sigma \sqrt{\log \frac{1}{\delta}} \cdot \frac{\sqrt{\log N_K}}{\alpha N_K} \vee \sqrt{\frac{1-\alpha}{\alpha N_K}}},
\end{align*}
recovering convergence rate for $N_K$ independent Bernoulli random variables with $p=\alpha$. 

As a comparison, we can consider a different sampling strategy: suppose $X_{i_1 \cdots i_K}$ are all i.i.d.\ Bernoulli random variables with $p=\alpha$. We expect the tail bound to be
\begin{align*}
	O \prn{\sigma \sqrt{\log \frac{1}{\delta}} \cdot \sqrt{\frac{1-\alpha}{\alpha N_1 \cdots N_K}}}
\end{align*}
by standard concentration inequality. We recommend using Cartesian-product sub-sampling primarily in situations where one dimension is much larger than all the others combined, i.e., when $N_K \gg N_1 N_2 \cdots N_{K-1}$. More generally, as mentioned in Section~\ref{sec:intro}, Cartesian product sub-sampling may be needed for practical reasons of implementation even if it does not lead to better concentration.
(Of course, since we are free to reorder $(N_1, \cdots, N_K)$ in Cor.~\ref{cor:avg-effect}, from Eq.~\eqref{eq:avg-effect-one-mode}, we can see that for any $i, j \in [K]$, the tail bound is tighter for sub-sampling along mode-$i$ against mode-$j$ under the same rate $\alpha$ if $N_i > N_j$.)

\paragraph{Numerical experiment}  To corroborate our theoretical predictions, we perform numerical simulations. For $K=3$, we construct the weight tensor $W_{ijk} = (ijk)^2 / (N_1 N_2 N_3)^2 \in [0,1]$. Consider the three Cartesian sub-sampling procedures
\begin{align*}
	I \times J \times K \subset [N_1] \times [N_2] \times [N_3]
\end{align*} 
such that (i) $|I|=0.4N_1, |J|=N_2, |K|=N_3$; (ii) $|I|=N_1, |J|=0.4N_2, |K|=N_3$; (iii) $|I|=N_1, |J|=N_2, |K|=0.4N_3$; and (iv) i.i.d.\ sub-sampling with $\mathsf{Bernoulli}(0.4)$, i.e., each entry is observed with probability $0.4$. Those four schemes correspond to the same $40\%$ sampling rate. 

Our theory in Cor.~\ref{cor:avg-effect} predicts estimation errors of $O(1/\sqrt{N_1})$, $O(1/\sqrt{N_2})$, $O(1/\sqrt{N_3})$ for sampling strategies (i), (ii), (iii), respectively, while the estimation error for i.i.d.\ sampling (iv) is $O(1/\sqrt{N_1 N_2 N_3})$. We report the simulation results and recover the predictions in Fig.~\ref{fig:exp-1} under two setups: (A) $(N_1, N_2, N_3) = (5,10, N)$ and (B) $(N_1, N_2, N_3) = (0.1N, 0.2N, 0.5N)$, with $N$ in $\{10, 20, \cdots, 500\}$. As predicted, these results show that i.i.d.\ sampling (iv) leads to the lowest error; however, when $N_1,N_2$ are small while $N_3$ is large (as in setting (A)), subsampling along mode $3$ (iii) has much lower error than subsampling across either of the other two modes as in (i) and (ii), as predicted; on the other hand, when $N_1,N_2,N_3$ are comparable (as in setting (B)), sampling strategies (i), (ii), (iii) show similar performance, although the relative ordering again agrees with the ordering of the dimensions, $N_1<N_2<N_3$.
\begin{figure}[htbp!]
	\centering
    {
    \footnotesize
	\begin{tabular}{c}
		\includegraphics[width=.8\linewidth]{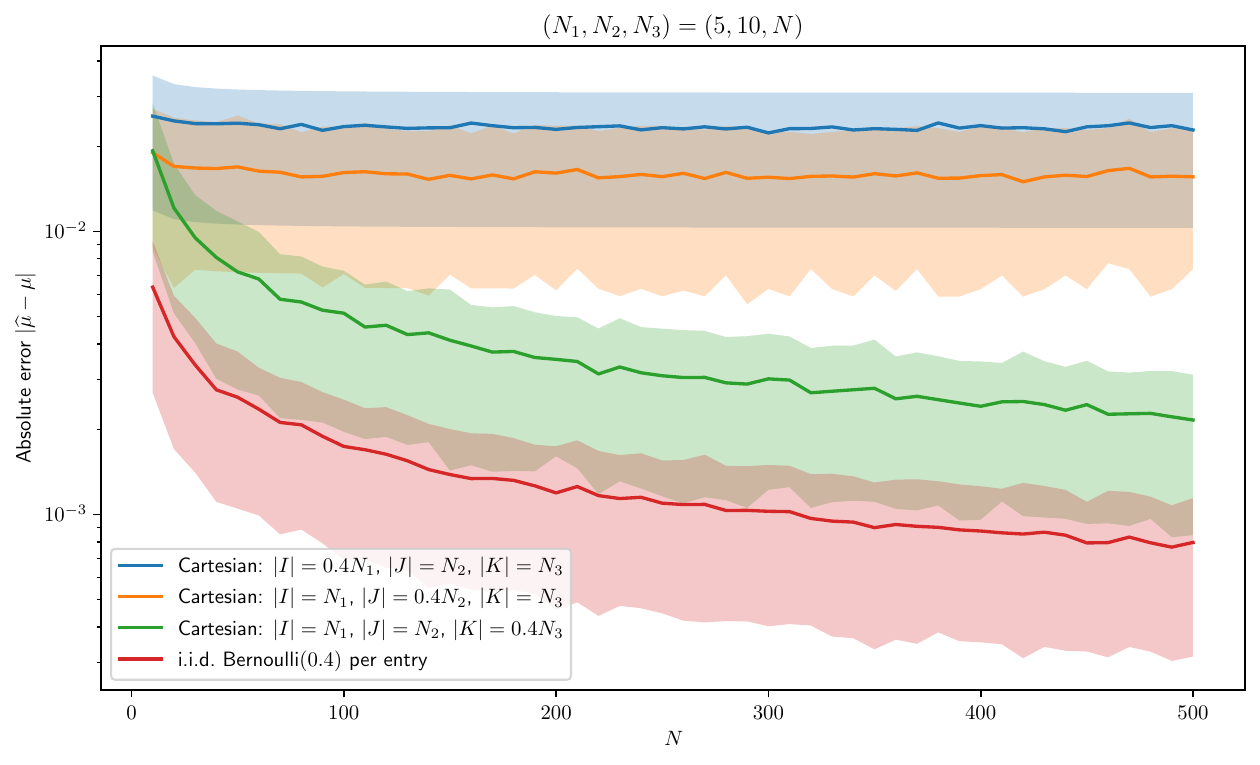} \\ 
        Experiment under setting (A) \\
        \includegraphics[width=.8\linewidth]{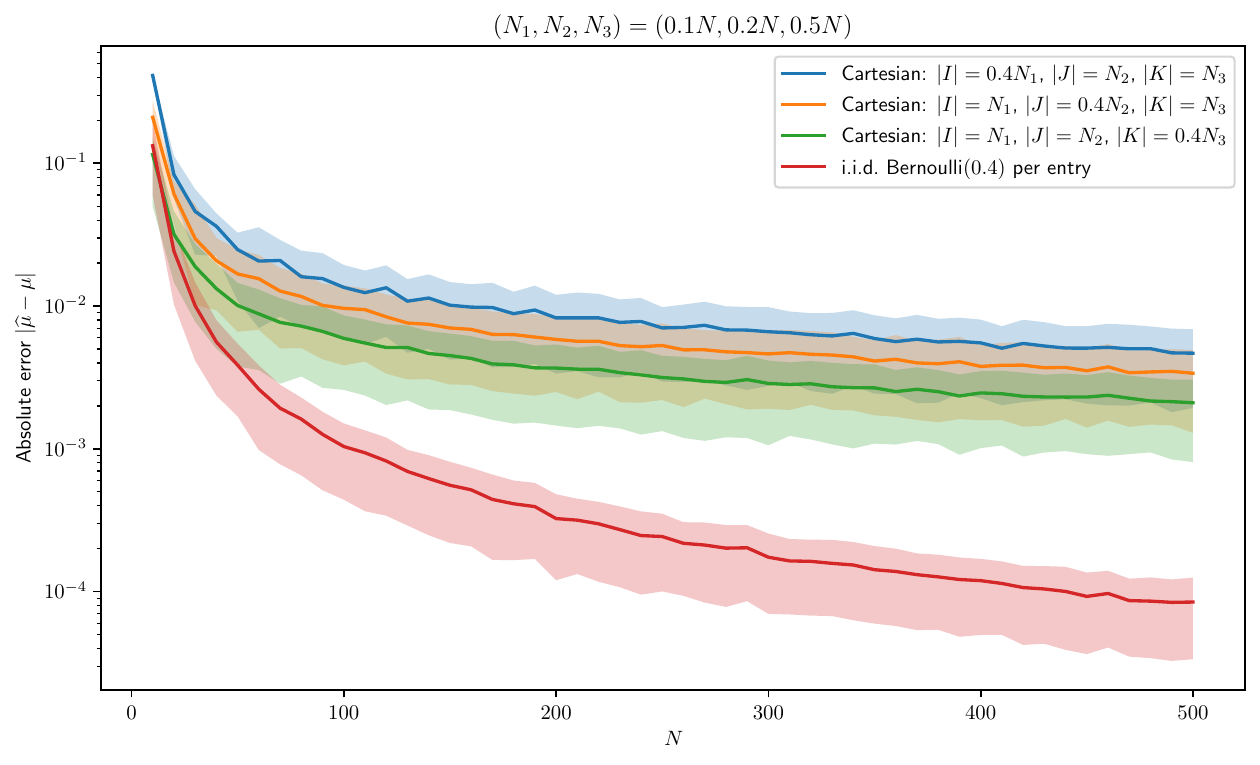} \\
		  Experiment under setting (B) 
	\end{tabular}
    }
	\caption{Numerical simulations of the sub-sampling procedures. Under both $(N_1, N_2, N_3)$ configurations, we simulate Cartesian product sub-sampling with (i) $|I|=0.4N_1, |J|=N_2, |K|=N_3$; (ii) $|I|=N_1, |J|=0.4N_2, |K|=N_3$; (iii) $|I|=N_1, |J|=N_2, |K|=0.4N_3$; and also include (iv) i.i.d.\ sub-sampling with $\mathsf{Bernoulli}(0.4)$. For each $N$ we run $T=1,000$ independent trials for each of the procedure, and report the mean of the absolute estimation errors of the Horvitz-Thompson estimator $|\what{\mu} - \mu|$, with error bands between the 25th and 75th quantiles.}  \label{fig:exp-1}
\end{figure}

\subsection{Fixed-design sketching by discrete sine transforms} \label{sec:dst-sketching}
Next, we apply our mathematical tools to Example~\ref{example:sketching}. Consider $X_k = \theta_k \in \R^{q \times r}$ a sequence of exchangeable matrices, where our goal is to estimate $\wb{\theta}=\frac{1}{N}\sum_{k=1}^N\theta_k$ while reducing communication cost via sketching. 

Let $U\in\R^{q\times q}$ be a fixed orthogonal matrix. Let $D_1,\dots,D_N$ be nonnegative diagonal matrices, each with at  most $q'$ nonzero entries, such that $\sum_{k=1}^N D_k = I_q$. For each $k$, we can construct a matrix $U_k\in\R^{q'\times q}$ such that $U_k^\top U_k = U^\top D_k U$---that is, the rows of $U_k$ are constructed by selecting (and rescaling) the rows of $U$ corresponding to nonzero entries of $D_k$.
These fixed matrices $U_1,\dots,U_N$ will then be used as the sketching matrices.

\begin{corollary} \label{cor:sketching}
	Let $\theta_1, \cdots, \theta_N$ be an exchangeable sequence of data matrices in $\R^{q \times r}$. Let $U_1,\dots,U_N\in\R^{q'\times q}$ be fixed matrices constructed as above.
    Assume there are constants $L, M > 0$ such that $\norm{\theta_k}_F \leq 1, \norm{U_k^\top U_k} = \norm{D_k} \leq L$ and $\norm{\sum_{k=1}^N U_k^\top U_k U_k^\top U_k} = \norm{\sum_{k=1}^N D_k^2} \leq M$, and define the variance parameter
	\begin{align*}
		\wt{\sigma}_\theta^2 := \frac{1}{N} \sum_{k=1}^N \norm{\theta_k - \wb{\theta}}_F^2 + \epsilon_N \max_{k \in [N]} \norm{\theta_k - \wb{\theta}}_F^2.
	\end{align*} 
	For all $\delta \in (0,1)$, as long as $q' \geq H_N = \sum_{k=1}^N k^{-1} = O(\log N)$, it holds
	\begin{align}\label{eqn:sketching_tail_bound}
		\P \brc{\norm{\sum_{k=1}^N U_k^\top U_k \theta_k - \wb{\theta}} \geq \wt{\sigma}_\theta \sqrt{2 (1+\epsilon_N) q'M \log \frac{q+r}{\delta}}  + \frac{2L}{3} (1+\epsilon_N) \log \frac{q+r}{\delta} } \leq \delta.
	\end{align}
\end{corollary}
We defer its proof to Appendix~\ref{proof:sketching} using an explicit construction of commutativity pairs. We point out that we present the result using $q'$ to show the benefit of dimension reduction by sketching in the straightforward fashion---however, Cor.~\ref{cor:sketching} can still provide tight bound in $N$ if $q' \asymp q$. For example, when $D_k = I_q/N$ (i.e., $q'=q$, and $M=1/N$), the dominant term in the tail bound~\eqref{eqn:sketching_tail_bound} scales as $q'M = q/N$.

To see how our result applies in a concrete setting, consider the orthogonal discrete sine transform (DST) matrix $U \in \R^{q \times q}$ whose entries are 
\begin{align*}
	U_{ij} = \sqrt{\frac{2}{q+1}}\sin \prn{\frac{\pi ij}{q+1}}.
\end{align*}
Let $U = \begin{bmatrix} u_1 & \cdots & u_q \end{bmatrix}^\top$, we define $U_k = \sqrt{q/(q'N)} \cdot \begin{bmatrix} u_{(k-1)q'+1} & \cdots & u_{(k-1)q' + q'} \end{bmatrix}^\top$ using the periodic index convention, that is, we define $u_{i'} = u_{i}$ for all $i' \in \mathbb{N}_+$ and $i \in [q]$ such that $i' \equiv i \pmod q$. When $q'N$ is a multiple of $q$, it is elementary to check $\sum_{k=1}^N U_k^\top U_k = U^\top U = I_q$. If in addition $H_N \leq q' \leq q$, Corollary~\ref{cor:sketching} then applies for $D_k$ whose $i$-th diagonal element is $q/(q'N)$ if $(k-1)q' < i \leq (k-1)q'+q'$ and zero otherwise. The parameters applying Corollary~\ref{cor:sketching} are then $L = M = q/(q'N)$, implying
$$
    \P \brc{\norm{\sum_{k=1}^N U_k^\top U_k \theta_k - \wb{\theta}} \geq \wt{\sigma}_\theta \sqrt{\frac{2 q(1+\epsilon_N)}{N} \cdot \log \frac{q+r}{\delta}}  + \frac{2q(1+\epsilon_N)}{3q'N}  \log \frac{q+r}{\delta} } \leq \delta.
$$
It is striking from the above bound that (i) there is no additional fluctuation arising from sampling the i.i.d.\ random variables $U_k$, since the leading stochastic term $O(\wt{\sigma}_\theta/\sqrt{N})$ depends solely on the averaged data matrices; (ii) the dominant noise term is independent of the dimension $q'$ of the compression space; and (iii) the bound becomes tighter in the higher-order noise term of order $O(1/N)$ as $q'$ increases.\smallskip

\paragraph{Numerical experiment} Let $q=1000, r =1, N=50$. We construct a set of exchangeable vectors $\theta_1, \cdots, \theta_N \simiid \normal(\mu, I_q/q)$ conditional on an independently sampled $\mu \sim \normal(0, \Sigma/q)$ where $\Sigma$ is the $\mathsf{AR}(1)$ covariance matrix $\Sigma_{ij} = \rho^{|i-j|}$. On the same dataset and a $q'$ such that $q'N$ is a multiple of $q$, we test four sketching schemes: 
\begin{enumerate}
    \item[(A)] the aforementioned fixed-design sketching matrices;
    \item[(B)] randomly sub-sampled without replacement DST sketching matrices $U_k^\top U_k =q/(q'N) \cdot U \wt{D}_k U$, where the diagonal entries of the matrices $\wt{D}_k$ are mutually independent random variables distributed according to $\mathsf{Bernoulli}(q'/q)$;
    \item[(C)] randomly sub-sampled with replacement DST matrices. Instead of (B), for each $\wt{D}_k$, independently generate $q$ random pairs $(i_l, \xi_l), l \in [q]$ where $i_l$ is uniform in $[q]$ and $\xi_l$ is from $\mathsf{Bernoulli}(q'/q)$. Then we construct $\diag(\wt{D}_k)_i = \sum_{l=1}^q \ind_{\xi_l=1, i_l=i}$;
    \item[(D)] random Gaussian sketching matrices, with $U_k$'s rows independently drawn from $\normal(0, I_q/(q'N))$;
\end{enumerate}
We report our numerical simulations in Fig.~\ref{fig:exp-1.5}, where we can clearly see the advantage of fixed-design sketching by (A), especially compared to the sub-sampling rows of DST with replacement by (C) and the fully randomized Gaussian sketching by (D).  
\begin{figure}[htbp!]
	\centering
    \includegraphics[width=.8\linewidth]{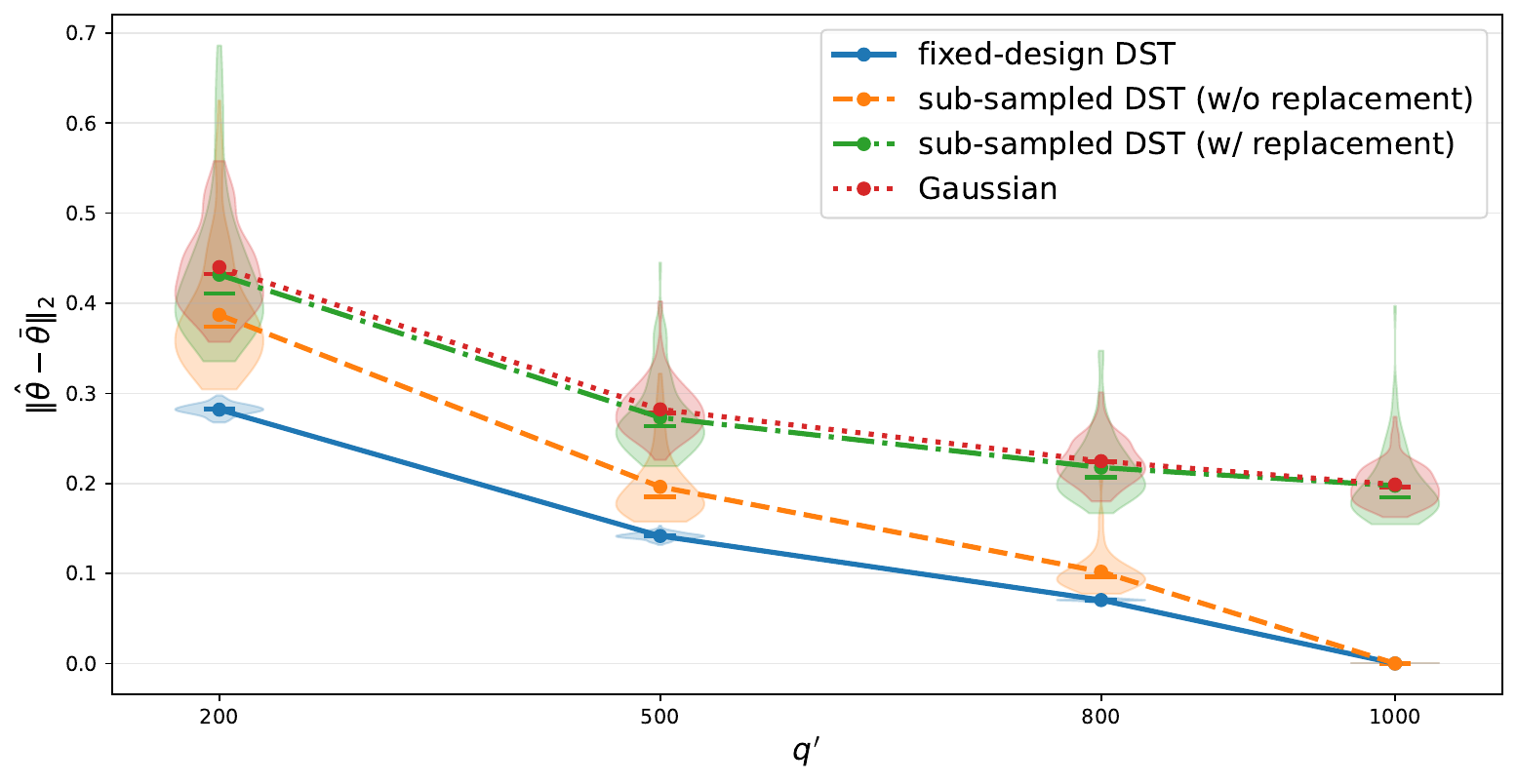}
	\caption{Violin plots of the error statistic $\|\what{\theta} - \wb{\theta}\|$ when $\rho = 0.99$, comparing the sketching schemes: (A) fixed-design DST, (B) sub-sampled DST (without replacement), (C) sub-sampled DST (with replacement) and (D) Gaussian sketching. We run the experiments for $q' \in \{200, 500, 800, 1000\}$ over $100$ trials. The bars ``\textendash'' represent the medians and dots ``\textperiodcentered'' represent the means. The curves connect the means between different values of $q'$. }  \label{fig:exp-1.5}
\end{figure}

\subsection{Communication-efficient ridge-tuned federated averaging} \label{sec:rtfa}

 Building on the simulation example of fixed-design sketching, we now turn to a concrete scenario in which federated averaging of mutually dependent yet exchangeable parameters $\theta_1, \cdots, \theta_N$ becomes practically relevant. Briefly mentioned in Example~\ref{example:sketching}, federated learning (FL) concerns a collection of $N$ clients or devices collaboratively solving machine learning models through local computation and coordination with a central server~\citep{kairouz2021advances}. 
 
 We study a generalization of the vanilla federated averaging~\citep{mcmahan2017communication} in~\citep{li2021ditto, cheng2021fine}, using broadcasts of the central server to fine-tune personalized models on local clients by a ridge-tuned penalty relative to the global parameter. See~\cite{t2020personalized} for similar ideas of communication with the central sever and interaction with local models to achieve personalization. Instead of the static averaging setup in Sec.~\ref{sec:dst-sketching}, at each round $t=0,1,2,\cdots, T$, we consider a global parameter $\wb{\theta}_t \in \R^{q \times r}$ initialized in $\wb{\theta}_0$ and local parameters $\theta_{k, t} \in \R^{q \times r}$ with initialization of i.i.d.\ in $t=0$ for $k \in [N]$. To this end, we introduce a sequence of publicly known fixed-design sketching matrices $U_k \in \R^{q' \times q}$ where $q' \ll q$, and $k \in [N]$ that satisfy $\sum_{k=1}^N U_k^\top U_k = I_q$. The algorithm then proceeds as follows in Algorithm~\ref{alg:RTFA}.
\begin{algorithm}
	\caption{Fixed-design sketching RTFA} \label{alg:RTFA}
	\begin{algorithmic}[1]
		\STATE {\bf Given:} Publicly known sketching matrices $U_1, \cdots, U_N \in \R^{q' \times q}$.
		\STATE {\bf Result:} The global parameter $\wb{\theta}_T$ and local parameters $\theta_{k, T}$.
		\FOR{$t \leftarrow 0$ {\bf to} $T-1$}
		\STATE The server broadcasts $\wb{\theta}_t$ to local agents.
		\STATE The server generates a fresh random permutation $\pi$ on $[N]$ and broadcasts to the agents.
		\FOR{$k \in [N]$ in parallel}
		\STATE Let $\theta_{k,t}^{(0)} = \theta_{k,t}$.
		\STATE Run $m_t$ steps of local gradient descents with step size $\eta$ on local agents such that
		\begin{align*}
			\theta_{k, t}^{(i+1)}  \leftarrow  \theta_{k, t}^{(i)} - \eta \prn{ \nabla L_k(\theta_{k, t}^{(i)}) + \lambda(\theta_{k, t}^{(i)} - \wb{\theta}_t) }, \qquad i=0, 1, \cdots, m_t-1.  
		\end{align*}
		\STATE $\theta_{k, t+1} \leftarrow \theta_{k, t}^{(m_t)}$. Upload the sketched parameter $\phi_{k,t+1} \to U_{\pi(k)} \theta_{k, t+1} \in \R^{p \times r}$.
		\ENDFOR
		\STATE The central server aggregates $\wb{\theta}_{t+1} = \sum_{k=1}^N U_{\pi(k)}^\top \phi_{k, t+1} = \sum_{k=1}^N U_k^\top U_k \theta_{\pi^{-1}(k), t+1}$.
		\ENDFOR
	\end{algorithmic}
\end{algorithm}

We point out that at each round $t$, $(\theta_{\pi^{-1}(1), t}, \cdots, \theta_{\pi^{-1}(N), t})$ are exchangeable but \emph{not} mutually independent---since they share mutual information from the previous aggregated estimate, $\wb{\theta}_t$. The fixed sequence of sketching matrices allow us to consider broader class of structured sketching design such as the Fourier basis, beyond generating fresh Gaussian random matrices~\citep{bartan2023distributed} or count sketches~\citep{ivkin2019communication} at each step to ensure independence, while greatly reduce uploading communication costs.

\paragraph{Numerical experiment} Let $q=1000, r =1, n= 200, N=50$ and $\theta^\star = e_1 = (1, 0, \cdots, 0)^\top \in \R^q$. Assume for each local agent, it observes data $X_k$ whose rows are i.i.d.\ from $\normal(0, I_q/q)$ and $y_k = X_k \theta^\star + \epsilon_k \in \R^n$. Suppose $\epsilon_1, \cdots, \epsilon_N \simiid \normal(0, \sigma^2 I_n)$, and the local loss functions are $L_k(\theta) = \ltwo{X_k \theta - y_k}^2$. To simplify the experiment, we set $m_t = \infty$ in Algorithm~\ref{alg:RTFA}, i.e. solving the optimization problem exactly at each round with closed-form updates
$$
	\theta_{k, t} = (X_k^\top X_k + \lambda I_q)^{-1} (X_k^\top y_k + \lambda \wb{\theta}_{t-1}), \qquad \text{for all } k \in [N].
$$
Without communication constraint, the stationary point when $t \to \infty$ for $\wb{\theta}_t$ is
\begin{align*}
	\wb{\theta}^\star = \frac{1}{N} \sum_{k=1}^N (X_k^\top X_k + \lambda I_q)^{-1} (X_k^\top y_k + \lambda \wb{\theta}^\star ),
\end{align*}
and thus
\begin{align*}
	\wb{\theta}^\star = \brc{\sum_{k=1}^N X_k^\top X_k (X_k^\top X_k + \lambda I_q)^{-1}}^{-1} \brc{\sum_{k=1}^N (X_k^\top X_k + \lambda I_q)^{-1} X_k^\top y_k}.
\end{align*}
Consider two sketching setups: (i) fixed design DST sketching matrices introduced in Sec.~\ref{sec:dst-sketching}; (ii) random Gaussian sketching matrices. We use (ii) as the baseline, with $U_k \in \R^{q' \times q}$ randomly generated at each round $t$ whose rows are i.i.d.\ vectors from $\normal(0, I_q/(q'N))$, in which case $\sum_{k=1}^N \E[U_k^\top U_k] = I_q$. Starting from zero initialization $\wb{\theta}_0 = 0$, we run Algorithm~\ref{alg:RTFA} setting $m_t=\infty$ and plot the error curves $\|\wb{\theta}_t - \wb{\theta}^\star\|$ at each iteration for the two sketching schemes. See Fig.~\ref{fig:exp-2} for the simulation results, which demonstrate using a fixed design sketching reduces variance inherent to Gaussian sketching matrices, greatly improving the federated averaging estimation accuracy.

\definecolor{mplblue}{RGB}{31,119,180}
\definecolor{mplorange}{RGB}{255,127,14}

\begin{figure}[htbp!]
	\centering
    { \footnotesize
    \begin{tabular}{c c}
	\includegraphics[width=.45\linewidth]{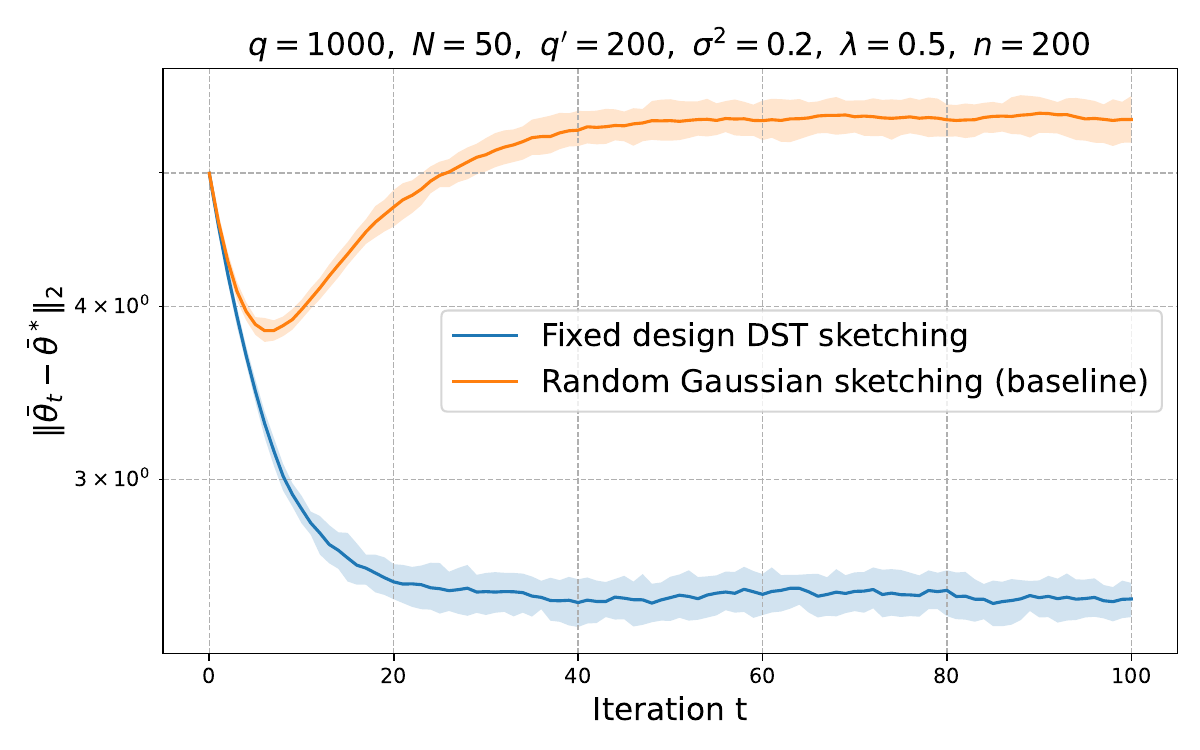} & 
	\includegraphics[width=.45\linewidth]{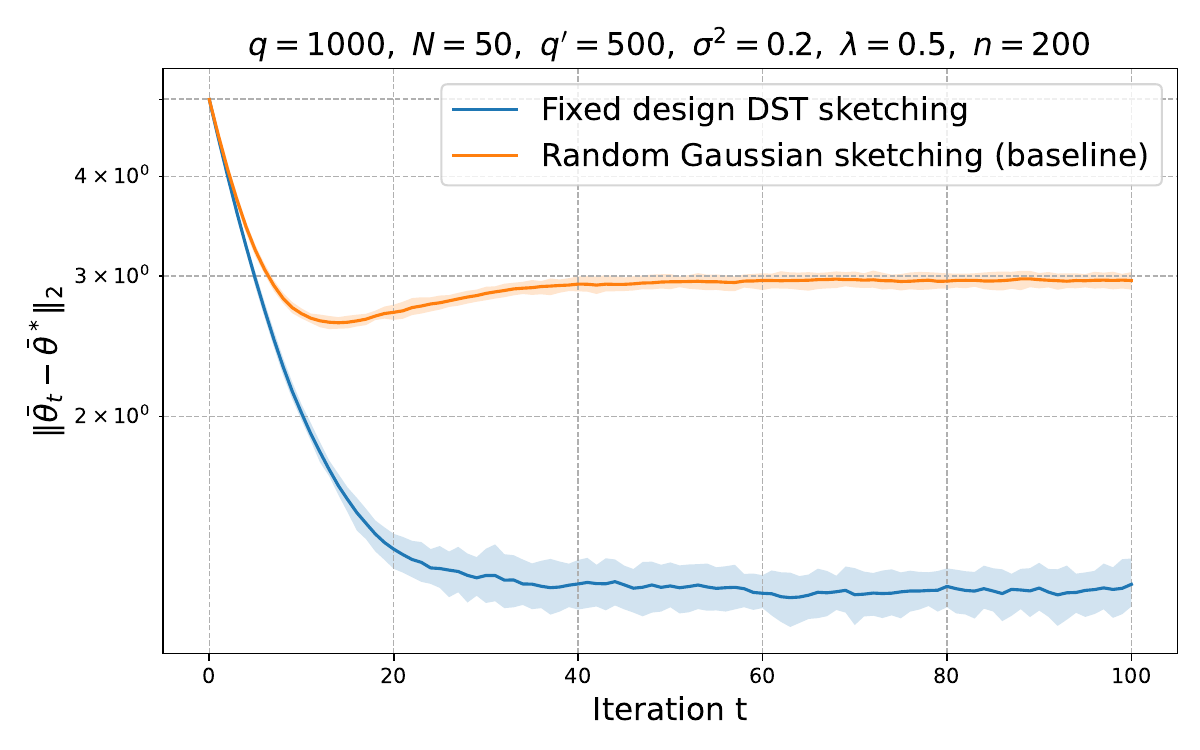} \\
    Experiment for $q'=200$ & Experiment for $q'=500$ \\
	\includegraphics[width=.45\linewidth]{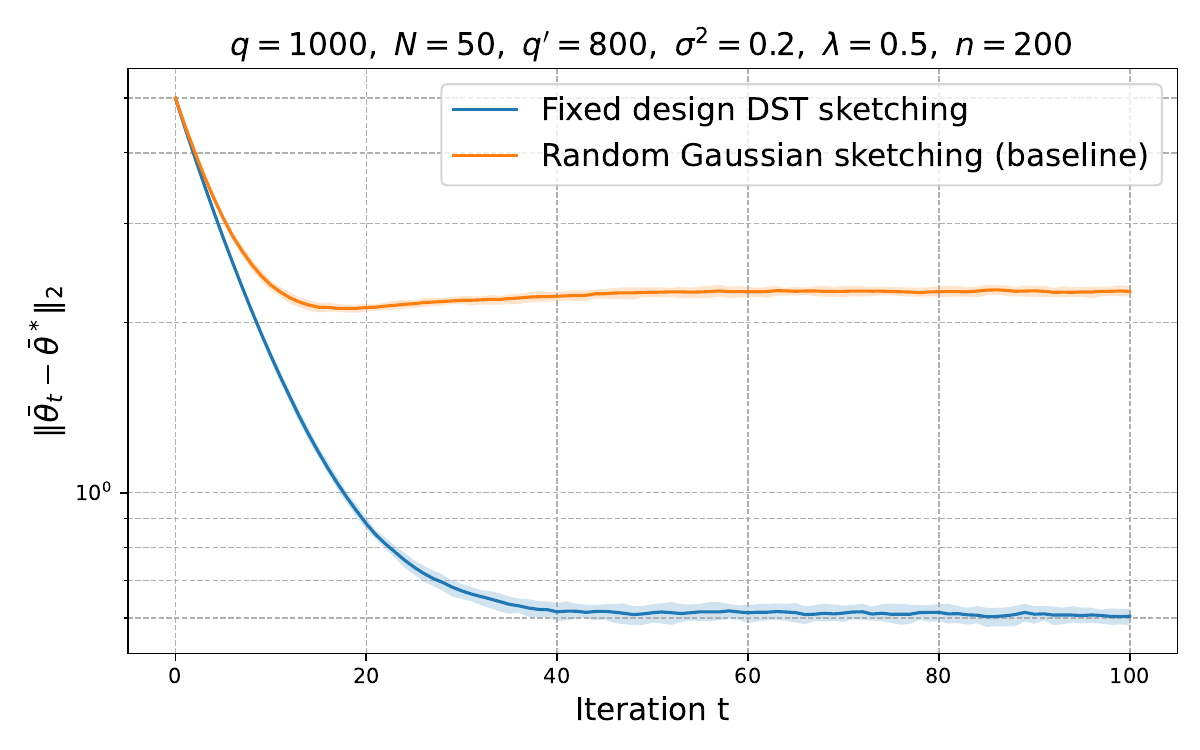} &
	\includegraphics[width=.45\linewidth]{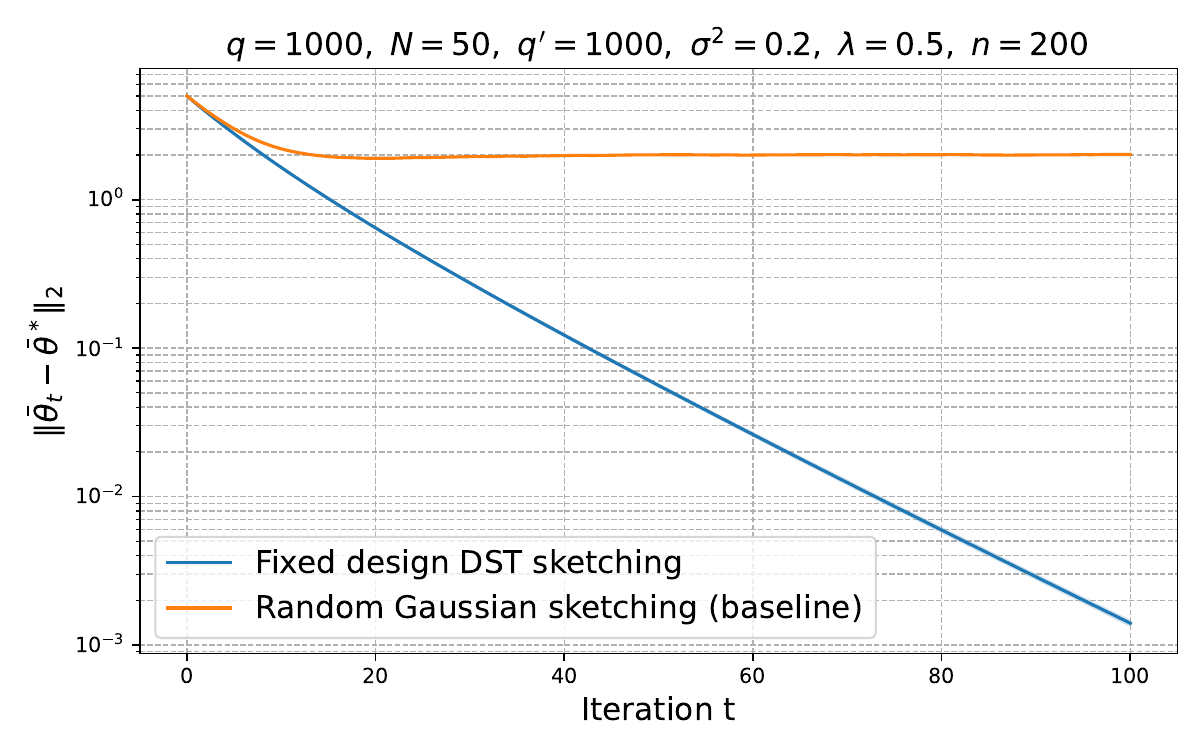} \\
    Experiment for $q'=800$ & Experiment for $q'=1000$ 
    \end{tabular}
    }
	\caption{Numerical simulations of sketching RTFA for (i) {\color{mplblue}blue: Fixed design DST sketching} vs. (ii) {\color{mplorange}{orange: Gaussian sketching}}. We run the experiments for $\sigma^2 = 0.2$, $\lambda = 0.5$ and $q' \in \{200, 500, 800, 1000\}$ over $100$ trials.  We plot the mean error curves of  $\|\wb{\theta}_t - \wb{\theta}^\star\|$ at each iteration $t$ with error bands between the 25th and 75th quantiles.}  \label{fig:exp-2}
\end{figure}

\section{Discussion}

We provide concentration bounds of Hoeffding and Bernstein types in mode-exchangeable tensors and weighted sums of exchangeable matrix-valued data. The bounds we obtain recover and improve previously known results, notably in the i.i.d.\ case, weighted sum of random scalars, and combinatorial sums of matrix arrays. We are able to leverage the richer class of distributional symmetry studied in our paper to analyze wider class of problems, including estimating average effect of multi-factor models and fixed-design sketching applying to federated averaging. Still, intriguing questions remain as future directions.

\begin{description}
	\item[Nonlinear exchangeable sums.] The results in our paper focus on concentration inequalities for the bilinear mapping $Z = \llangle W, X \rrangle$. It is interesting to ask if concentration occurs for a general nonlinear mapping $Z = f(W, X)$ under the same assumptions of exchangeability in tensors or matrix-valued data. In the unweighted and independent case, McDiarmid's inequality~\citep{mcdiarmid1989method} holds and suggests similar results should hold for the weighted and exchangeable case, under mild assumptions such as Lipschitz continuity of $f$. As a concrete example, for matrix-valued data, developing concentration bounds for nonlinear sum $Z = \sum_{k=1}^N A_k W_k B_k$ where $X =(A_k, B_k)$ is of crucial interest in the analyses for problems like blind deconvolution~\citep{chen2020noisy}.  
	
	\item[Nonasymptotic optimality.] As is previously raised in~\cite{foygel2024hoeffding}, an interesting question to ask is the nonasymptotic optimality for the multiplicative terms $(1+\epsilon_{N_k})$ in our theoretical results. In the unweighted case of sampling without replacement of random variables, Serfling's inequality~\citep{serfling1974probability} suggests such inflated constants are unnecessary (and indeed, sampling without replacement can lead to concentration that is even be stronger than the i.i.d.\ bounds); on the other hand,~\cite{foygel2024hoeffding} shows that for a general exchangeable setting some inflation is unavoidable. Indeed, recent work~\cite{kim2026sharper} shows the optimal rate in the scalar-valued Hoeffding bound is $O(\frac{1}{n})$, providing evidence for sharper results in the matrix-valued and tensor cases.
    In the setting of tensor data and matrix data studied here, are there settings where this inflation may be unnecessary (analogous to sampling without replacement)? And, what are concrete examples such that the inflated tails occur? A separate question is that of the role of the tensor structure: our bounds (cf.~Theorems~\ref{thm:matrix-valued-hoeffding} and~\ref{thm:matrix-valued-bernstein}) for multi-way tensors in general depend on the ordering of the modes $(N_1, \cdots, N_K)$. It is of course possible to optimize the bounds by choosing the best ordering, but it also suggests potentiality for improvement to results independent of the mode ordering.
	
	\item[Richer symmetry structures for matrix-valued data.] Our paper generalizes \citeauthor{foygel2024hoeffding}'s results to two separate branches shown in Table.~\ref{tab:main-results}. A natural future direction is to develop general results of sums like $Z=\sum_{i=1}^{N_1} \cdots \sum_{i=1}^{N_K} W_{i_1 \cdots i_K} X_{i_1 \cdots i_K}$ where $W_{i_1 \cdots i_K}$ and $X_{i_1 \cdots i_K}$ are all matrix-valued, and $X_{i_1 \cdots i_K}$ are mode-exchangeable, providing a unified framework allowing richer symmetry structures for matrix-valued data. This can be helpful in simplifying proofs for combinatorial sum of matrix arrays in Sec.~\ref{sec:combinatorial-array}, as naturally we can take $K=2$, $X_{ij} = A_{\pi(i), \pi'(j)}$ and $W_{ij} = \ind_{i=j}$, avoiding the careful construction of commutativity pairs in our paper.
	
	\item[Other symmetry groups and distributional invariances.] A higher level view of exchangeability is the relaxation of independence to distributional symmetry. Mode-wise and row-column-wise exchangeability are special cases of invariance under a group action and weaker symmetries, compared to the sequential exchangeability. However, the notion of distributional symmetry under group action can be much more general, such as cyclic group symmetry when $(X_1, \cdots, X_N) \eqdst (X_2, \cdots, X_N, X_1) \eqdst \cdots \eqdst (X_N, X_1, \cdots, X_{N-1})$. This occurs for $(X_1, \cdots, X_N) \sim \normal(0, \Sigma)$ where $\Sigma$ is a general semi-positive definite Toeplitz matrix, where the usual exchangeability breaks. The future direction is to develop concentration bounds given a general set of group action.
\end{description}

\section*{Acknowledgment}
R.F.B. was partially supported by the National Science Foundation via grant DMS-2023109. C.C. and R.F.B. were additionally supported by the Office of Naval Research via grant N00014-24-1-2544. We would like to thank Lester Mackey for discussions on Stein's method of exchangeable pairs. Chen Cheng would also like to thank Gary Cheng for helpful discussions.

\paragraph{Disclosure of AI use:} The GPT-6 Pro model was used in the revision process to audit existing proofs, make minor corrections, and fix typos. The authors take full responsibility for the contents of the paper.

\bibliography{bib}
\bibliographystyle{abbrvnat}

\newpage

\appendix

\section{Technical lemmas} \label{sec:technical}

This section contains technical lemmas used throughout the main proof body in the appendix. We may not explicitly refer to the lemmas when invoked, while keeping them here for completeness. 

\begin{lemma} \label{lem:tensor-matrix-product-property}
	For two tensors $A, B \in \tspace:= \R^{N_1 \times \cdots \times N_K}$, an integer $l \in [K]$ and two matrices $C ,D \in \R^{N_l \times N_l}$, it holds
	\begin{align*}
		A \times_l C \times_l D & = A \times_l DC, \\
		\<A \times_l C, B\>_\tspace & = \<A , B \times_l C^\top\>_\tspace.
	\end{align*}
\end{lemma}
\begin{proof}
	An elementary linear algebra exercise.
\end{proof}

\begin{lemma} \label{lem:perm-avg-matrix}
	For any matrix $A \in \R^{n \times n}$ such that $A\ones_n = 0$, it holds that
	\begin{align*}
		\frac{1}{n!} \sum_{\Pi_n \in \permset_n} \Pi_n^\top A \Pi_n = \frac{\Tr(A)}{n-1} \cdot \proj_{\ones_n}^\perp.
	\end{align*}
\end{lemma}
\begin{proof}
	Let $A'$ be the left hand side quantity. From $\Pi_n^\top A' \Pi_n = A'$ for any $\Pi_n \in \permset_n$, $A'$ must have identical off-diagonal entries and identical diagonal entries, namely $A' = a I_n + b \ones_n \ones_n^\top$. On the other hand, 
	\begin{align*}
		A' \ones_n = \frac{1}{n!} \sum_{\Pi_n \in \permset_n} \Pi_n^\top A \Pi_n \ones_n = A' \ones_n = \frac{1}{n!} \sum_{\Pi_n \in \permset_n} \Pi_n^\top A \ones_n = 0
	\end{align*}
	enforces $A' = c \proj_{\ones_n}^\perp$. To find the constant $c$, we take trace on both sides which implies
	\begin{align*}
		(n-1) c & = c \Tr (\proj_{\ones_n}^\perp) = \frac{1}{n!} \sum_{\Pi_n \in \permset_n} \Tr(\Pi_n^\top A \Pi_n ) = \frac{1}{n!} \sum_{\Pi_n \in \permset_n} \Tr( A \Pi_n \Pi_n^\top) = \Tr(A).
	\end{align*}
\end{proof}

The next lemma concerns convex and concave properties of the normalized exponential operator $\wb{\exp}$ (i.e. the trace exponentials). Notably, the third statement originates from the fundamental Lieb's theorem~\cite[Thm.~6]{lieb1973convex} on convex trace functions. We refer to one of its corollary below, which is in~\cite[Cor.~3.3]{tropp2012user} and~\cite[Cor.~3.4.2]{tropp2015introduction}. The first two results are in~\cite[Sec.~2]{tropp2012user} and the last one is~\cite[Lemma 7.6]{tropp2012user}.

\begin{lemma} \label{lem:convex-concave-exp}
	Suppose $A$ and  $B$ are symmetric matrices in $\mathbb{S}^n \subset \R^{n \times n}$.
	\begin{enumerate}
		\item The mapping $A \mapsto \Tr \exp A$ is convex in $\mathbb{S}^n$, monotonically increasing w.r.t. the partial order $\prec$ induced by $\mathbb{S}_{++}^n$.
		\item (Golden-Thompson inequality) $\Tr \exp(A + B) \leq \Tr \prn{\exp A \exp B}$. 
		\item (Lieb's inequality) Suppose $A$ is fixed and $B$ is random, then $\E \Tr \exp(A + B) \leq \Tr \exp (A + \log \E \exp B)$. Here $\log$ is defined on the positive definite cone $\mathbb{S}_{++}^n$ such that $\log \exp B = B$ for all $B \in \mathbb{S}^n$.
		\item (Symmetrization) Suppose $A$ is fixed and $B$ is random with $\E B = 0$. Let $\epsilon$ be a Rademacher independent of $B$, then $\E \Tr \exp(A + B) \leq \E \Tr \exp(A + 2 \epsilon B)$.
	\end{enumerate}
\end{lemma}

For a scalar random variable $Z$, the next lemma connects MGF bounds to one-sided tail event bounds. See~\cite[Sec.~2.3]{boucheron2013concentration} for the first result and \cite[Sec.~2.4]{boucheron2013concentration} for the second.
\begin{lemma} \label{lem:MGF-tail}
	For any random variable $Z \in \R$, if for some constants $a, b > 0$:
	\begin{enumerate}
		\item $\E[\exp\{\lambda Z\}] \leq \exp \{\lambda^2 a^2/2\}$ for all $\lambda \in \R$, it holds
		\begin{align*}
			\P \brc{Z \geq a \sqrt{2 \log \frac{1}{\delta}}}  \leq \delta, \qquad \text{for all }\delta \in (0, 1);
		\end{align*}
		\item $\E[\exp\{\lambda Z\}] \leq \exp \{\lambda^2 a^2/(2-2b|\lambda|)\}$ for all $|\lambda| < 1/b$, it holds
		\begin{align*}
			\P \brc{Z \geq a \sqrt{2 \log \frac{1}{\delta}} + b \log \frac{1}{\delta}}  \leq \delta, \qquad \text{for all }\delta \in (0, 1).
		\end{align*}
	\end{enumerate}
\end{lemma}

When $Z \in \R^{n \times m}$, \citet{ahlswede2002strong} proposed using trace exponentials $\wb{\exp}$ and matrix MGF bounds to the operator norm bounds. Indeed, we can establish the following lemma. We point out that it recovers the two-sided version of preceding Lemma~\ref{lem:MGF-tail} when $n=m=1$.
\begin{lemma} \label{lem:MGF-tail-matrix}
	For any random matrix $Z \in \R^{n \times m}$, if for some constants $a, b > 0$:
	\begin{enumerate}
		\item $\E[\wb{\exp}\{\lambda Z\}] \leq \exp \{\lambda^2 a^2/2\}$ for all $\lambda \in \R$, it holds
		\begin{align*}
			\P \brc{\norm{Z} \geq a \sqrt{2 \log \frac{n+ m}{\delta}}}  \leq \delta, \qquad \text{for all }\delta \in (0, 1);
		\end{align*}
		\item $\E[\wb{\exp}\{\lambda Z\}] \leq \exp \{\lambda^2 a^2/(2-2b|\lambda|)\}$ for all $|\lambda| < 1/b$, it holds
		\begin{align*}
			\P \brc{\norm{Z} \geq a \sqrt{2 \log \frac{n+m}{\delta}} + b \log \frac{n+m}{\delta}}  \leq \delta, \qquad \text{for all }\delta \in (0, 1).
		\end{align*}
	\end{enumerate}
\end{lemma}
\begin{proof}
	We give a short proof of the lemma, which essentially proves Lemma~\ref{lem:MGF-tail} as well. Denote by $Z$'s singular values as $\sigma_1, \cdots, \sigma_{n \wedge m} \geq 0$, one can easily check that the dilated matrix satisfies
	\begin{align*}
		\mathsf{spec} \prn{\begin{bmatrix} 0 & A \\ A^\top & 0 \end{bmatrix}} \subset \{0, \sigma_1, \cdots, \sigma_{n \wedge m}, -\sigma_1, \cdots, -\sigma_{n \wedge m}\}.
	\end{align*}
	Thus $\wb{\exp}\{\lambda Z\} \geq \frac{1}{n+ m} \exp(\lambda \norm{Z})$ when $\lambda > 0$. As a consequence, applying the Chernoff bound,
	\begin{align*}
		\P \brc{\norm{Z} \geq t} \leq (n + m) \inf_{\lambda > 0} e^{-\lambda t} \E \brk{\wb{\exp}\{\lambda Z\}}.
	\end{align*}
	Substitution of $\lambda = \frac{t}{(n+m)a^2}$ in the first case and $$\lambda = \frac{2}{\sqrt{(n+m)^2a^4 + 2(n+m)a^2 bt} + (n+m)a^2 + 2bt}$$ in the latter completes the proof.
\end{proof}

\section{Proofs for Section~\ref{sec:master-tensor-bounds}} \label{proof:master-tensor-bounds}

\subsection{Proof of Theorem~\ref{thm:tensor-hoeffding}} \label{proof:tensor-hoeffding}
Let $\mc{F}_k = \sigma(\wb{X}_k)$ for $k=0, \cdots, K$, since $$(\wb{X}_{k})_{i_1\cdots i_k} = \frac{1}{N_{k+1}} \sum_{i_{k+1}=1}^{N_{k+1}}(\wb{X}_{k+1})_{i_1\cdots i_k i_{k+1}} ,$$
we have the filtration $\sigma(\wb{X}) = \mc{F}_0 \subset \cdots \subset \mc{F}_K = \sigma(X)$. Define the matrix $\mathsf{A}_n \in \R^{n \times n}$ whose $i$-th row is $(\zeros_{i-1}, 1, -\frac{1}{n-i} \cdot \ones_{n-i})$ when $ i \in [n-1]$, and its $n$-th row being $\zeros_n$. As a special case, $\mathsf{A}_1 = 0$. We also make use of the auxiliary matrix $\mathsf{B}_{n} \in \R^{n \times n}$, whose $i$-th row is $(\zeros_{i-1}, \frac{1}{n+1-i} \cdot \ones_{n+1-i})$, with the properties $\mathsf{A}_n^\dagger = I_n-\mathsf{B}_n^\top $ and $\mathsf{A}_n^\top (I_n - \mathsf{B}_n) = \proj_{\ones_n}^\perp$. Both matrices appear in~\cite{foygel2024hoeffding}. We have the following lemma for any fixed $U \in \tspace$ and its associated sequence $\wb{U}_k \in \tspace_k$.
\begin{lemma} \label{lem:tensor-hoeffding-plain}
	For any $k =0, 1, \cdots, K-1$ and the filtration $\{\mc{F}_k\}_{k=0}^K$ defined above, let $N_0=1$ and we have
	\begin{align*}
		\Ep \brk{\exp \brc{\lambda \prn{ \<\wb{U}_{k+1}, \wb{X}_{k+1}\>_{\tspace_{k+1}} - N_{k+1}  \<\wb{U}_k, \wb{X}_k\>_{\tspace_k}} } \mid \mc{F}_k} \leq \exp \brc{\frac{\lambda^2}{2} \cdot \prod_{l=0}^k N_l \cdot \norm{\wb{U}_{k+1} \times_{k+1} \mathsf{A}_{N_{k+1}} }_{\tspace_{k+1}}^2}.
	\end{align*} 
\end{lemma}
We defer its proof to Appendix~\ref{proof:tensor-hoeffding-plain}. Having established the above lemma, we can refine it into the next result.
\begin{lemma} \label{lem:tensor-hoeffding-jensen}
	For any $k =0, 1, \cdots, K-1$ and the filtration $\{\mc{F}_k\}_{k=0}^K$ defined above, let $N_0=1$ and $W \in \tspace$ be the weight tensor. We have
	\begin{align*}
		& \Ep \brk{\exp \brc{\lambda \prn{ \<\wb{W}_{k+1}, \wb{X}_{k+1}\>_{\tspace_{k+1}} - N_{k+1}  \<\wb{W}_k, \wb{X}_k\>_{\tspace_k}} } \mid \mc{F}_k} \nonumber \\
		& \leq \exp \brc{\frac{\lambda^2 (1+ \epsilon_{N_{k+1}})}{2} \cdot  \prod_{l=0}^k N_l \cdot \norm{\wb{W}_{k+1} \times_{k+1} \proj_{\ones_{N_{k+1}}}^\perp }_{\tspace_{k+1}}^2}.
	\end{align*} 
\end{lemma}
\begin{proof}
	From $\mathsf{A}_n \ones_n = 0$, we have $\mathsf{A}_{n}\proj_{\ones_{n}}^\perp = \mathsf{A}_{n}$ and $(\mathsf{A}_{n}^\top \mathsf{A}_{n})^\dagger \proj_{\ones_{n}}^\perp = (\mathsf{A}_{n}^\top \mathsf{A}_{n})^\dagger$. In fact, we can derive explicitly from $I_n - \mathsf{B}_n = \mathsf{A}_n (\mathsf{A}_n^\top \mathsf{A}_n)^\dagger$ that
	\begin{align*}
		(\mathsf{A}_{n}^\top \mathsf{A}_{n})^\dagger = (I_n - \mathsf{B}_{n}^\top)(I_n - \mathsf{B}_{n}).
	\end{align*}
	For any $\Pi_{k+1} \in \permset_{N_{k+1}}$, since $X \times_{k+1} \Pi_{k+1} \eqdst X$, we can substitute $\wb{U}_{k+1} = \wb{W}_{k+1} \times_{k+1} (\mathsf{A}_{N_{k+1}}^\top \mathsf{A}_{N_{k+1}})^\dagger \Pi_{k+1}$ and replace $X$ with $X \times_{k+1} \Pi_{k+1} $ in Lemma~\ref{lem:tensor-hoeffding-plain}. Simplifying through the following calculations,
	\begin{align*}
		& \<\wb{U}_{k+1}, \wb{(X \times_{k+1} \Pi_{k+1})}_{k+1}\>_{\tspace_{k+1}} - \<\wb{U}_k, \wb{(X \times_{k+1} \Pi_{k+1})}_k\>_{\tspace_k} =  \<\wb{U}_{k+1}, \wb{(X \times_{k+1} \Pi_{k+1})}_{k+1} \times_{k+1} \proj_{\ones_{N_{k+1}}^\perp}\>_{\tspace_{k+1}} \\
		& = \< \wb{W}_{k+1} \times_{k+1} (\mathsf{A}_{N_{k+1}}^\top \mathsf{A}_{N_{k+1}})^\dagger \Pi_{k+1}, \wb{X}_{k+1} \times_{k+1} \Pi_{k+1} \times_{k+1}  \proj_{\ones_{N_{k+1}}^\perp}\>_{\tspace_{k+1}} \\
		& \stackrel{\mathrm{(i)}}{=} \< \wb{W}_{k+1} \times_{k+1} \Pi_{k+1}^\top \proj_{\ones_{N_{k+1}}^\perp} (\mathsf{A}_{N_{k+1}}^\top \mathsf{A}_{N_{k+1}})^\dagger \Pi_{k+1}, \wb{X}_{k+1}\>_{\tspace_{k+1}} \\
		& \stackrel{\mathrm{(ii)}}{=} \< \wb{W}_{k+1} \times_{k+1} \Pi_{k+1}^\top (\mathsf{A}_{N_{k+1}}^\top \mathsf{A}_{N_{k+1}})^\dagger \Pi_{k+1}, \wb{X}_{k+1}\>_{\tspace_{k+1}}, \\
		&  \norm{\wb{U}_{k+1} \times_{k+1} \mathsf{A}_{N_{k+1}} }_{\tspace_{k+1}}^2 = \norm{\wb{W}_{k+1} \times_{k+1} (\mathsf{A}_{N_{k+1}}^\top \mathsf{A}_{N_{k+1}})^\dagger \Pi_{k+1} \times_{k+1} \mathsf{A}_{N_{k+1}} }_{\tspace_{k+1}}^2 \\
		& = \<\wb{W}_{k+1} \times_{k+1} \mathsf{A}_{N_{k+1}}(\mathsf{A}_{N_{k+1}}^\top \mathsf{A}_{N_{k+1}})^\dagger \Pi_{k+1}, \wb{W}_{k+1} \times_{k+1} \mathsf{A}_{N_{k+1}}(\mathsf{A}_{N_{k+1}}^\top \mathsf{A}_{N_{k+1}})^\dagger \Pi_{k+1}\>_{\tspace_{k+1}} \\
		&  \stackrel{\mathrm{(i)}}{=}  \<\wb{W}_{k+1} \times_{k+1}  \Pi_{k+1}^\top (\mathsf{A}_{N_{k+1}}^\top \mathsf{A}_{N_{k+1}})^\dagger \Pi_{k+1}, \wb{W}_{k+1} \>_{\tspace_{k+1}},
	\end{align*}
	where in (i) we apply Lemma~\ref{lem:tensor-matrix-product-property} and in (ii) we use $\mathsf{A}_n \ones_n = 0$, we can thus obtain
	\begin{align*}
		&\Ep \brk{\exp \brc{\lambda \< \wb{W}_{k+1} \times_{k+1} \Pi_{k+1}^\top (\mathsf{A}_{N_{k+1}}^\top \mathsf{A}_{N_{k+1}})^\dagger \Pi_{k+1}, \wb{X}_{k+1}\>_{\tspace_{k+1}}} \mid \mc{F}_k} \\
		&\leq \exp \brc{\frac{\lambda^2}{2} \cdot \prod_{l=0}^k N_l \cdot \<\wb{W}_{k+1} \times_{k+1}  \Pi_{k+1}^\top(\mathsf{A}_{N_{k+1}}^\top \mathsf{A}_{N_{k+1}})^\dagger \Pi_{k+1}, \wb{W}_{k+1} \>_{\tspace_{k+1}}}.
	\end{align*} 
	Next we apply Jensen's inequality, averaging over all permutation matrices in $\Pi_{k+1} \in \permset_{N_{k+1}}$ yields
	\begin{align*}
		&\Ep \brk{\exp \brc{\lambda \< \wb{W}_{k+1} \times_{k+1} \frac{1}{N_{k+1}!} \sum_{\Pi_{k+1} \in \permset_{N_{k+1}}} \Pi_{k+1}^\top (\mathsf{A}_{N_{k+1}}^\top \mathsf{A}_{N_{k+1}})^\dagger \Pi_{k+1}, \wb{X}_{k+1}\>_{\tspace_{k+1}}} \mid \mc{F}_k} \\
		&\leq \exp \brc{\frac{\lambda^2}{2} \cdot \prod_{l=0}^k N_l \cdot \<\wb{W}_{k+1} \times_{k+1}  \frac{1}{N_{k+1}!} \sum_{\Pi_{k+1} \in \permset_{N_{k+1}}}  \Pi_{k+1}^\top(\mathsf{A}_{N_{k+1}}^\top \mathsf{A}_{N_{k+1}})^\dagger \Pi_{k+1}, \wb{W}_{k+1} \>_{\tspace_{k+1}}}.
	\end{align*} 
	Invoking Lemma~\ref{lem:perm-avg-matrix} gives further simplification of
	\begin{align*}
		\frac{1}{N_{k+1}!} \sum_{\Pi_{k+1} \in \permset_{N_{k+1}}} \Pi_{k+1}^\top (\mathsf{A}_{N_{k+1}}^\top \mathsf{A}_{N_{k+1}})^\dagger \Pi_{k+1} & = \frac{\Tr((\mathsf{A}_{N_{k+1}}^\top \mathsf{A}_{N_{k+1}})^\dagger)}{N_{k+1}-1} \cdot \proj_{\ones_{N_{k+1}}}^\perp.
	\end{align*}
	We complete the proof through
	\begin{align*}
		\Tr((\mathsf{A}_{N_{k+1}}^\top \mathsf{A}_{N_{k+1}})^\dagger) & = \Tr((I_{N_{k+1}} - \mathsf{B}_{N_{k+1}}^\top)(I_{N_{k+1}} - \mathsf{B}_{N_{k+1}})) = \norm{I_{N_{k+1}} - \mathsf{B}_{N_{k+1}}^\top}_F^2 \\
		& = \sum_{i=1}^{N_{k+1}} \brc{\prn{\frac{{N_{k+1}}-i}{{N_{k+1}}+1-i}}^2 + ({N_{k+1}}-i) \cdot \frac{1}{({N_{k+1}}+1-i)^2}} = {N_{k+1}} - H_{N_{k+1}},
	\end{align*}
	and replacing $\lambda$ with $\lambda (1 + \epsilon_{N_{k+1}})$.
\end{proof}

Returning to the main proof, the conclusion follows from applying $K$ times of tower rule and Lemma~\ref{lem:tensor-hoeffding-jensen}. Indeed by substituting
\begin{align*}
	 \<W_k, X_k\>_{\tspace} & =  \<\wb{W}_k \otimes \ones_{N_{k+1}} \cdots \otimes \ones_{N_K}, \wb{X}_k \otimes \ones_{N_{k+1}} \cdots \otimes \ones_{N_K}\>_{\tspace} = \prod_{l=k+1}^K N_l \cdot \<\wb{W}_k, \wb{X}_k\>_{\tspace_k},  \\
	\norm{\wb{W}_k \times_k \proj_{\ones_{N_k}}^\perp }_{\tspace_k}^2 & = \frac{1}{\prod_{l=k+1}^K N_l} \norm{W_k \times_k \proj_{\ones_{N_k}}^\perp }_{\tspace}^2,
\end{align*} 
and replacing $\lambda$ with $(\prod_{l=k+2}^K N_l) \cdot \lambda $ in Lemma~\ref{lem:tensor-hoeffding-jensen}, it follows
	\begin{align*}
	& \Ep \brk{\exp \brc{\lambda \prn{ \prod_{l=k+2}^K N_l \cdot  \<\wb{W}_{k+1}, \wb{X}_{k+1}\>_{\tspace_{k+1}} - \prod_{l=k+1}^K N_l \cdot  \<\wb{W}_k, \wb{X}_k\>_{\tspace_k}} } \mid \mc{F}_k} \nonumber \\
	& \leq \exp \brc{\frac{\lambda^2 (1+ \epsilon_{N_{k+1}})}{2} \cdot \prn{\prod_{l=k+2}^K N_l}^2 \cdot \prod_{l=0}^k N_l \cdot \frac{1}{\prod_{l=k+2}^K N_l} \cdot \norm{W_{k+1} \times_{k+1} \proj_{\ones_{N_{k+1}}}^\perp }_{\tspace}^2} \nonumber \\
	& = \exp \brc{\frac{\lambda^2}{2} \cdot \prod_{l=0}^K N_l \cdot \frac{1 + \epsilon_{N_{k+1}}}{N_{k+1}} \cdot \norm{W_{k+1} \times_{k+1} \proj_{\ones_{N_{k+1}}}^\perp }_{\tspace}^2}.
\end{align*} 
Denoting by $M_k := \prod_{l=k+1}^K N_l \cdot \< \wb{W}_k, \wb{X}_k\>_{\tspace_k}$, we then finish the proof by $N_0=1$ and
\begin{align*}
	&\Ep \brk{\exp \brc{\lambda \prn{ \<W, X\>_{\tspace} - \prod_{l=1}^K N_l \cdot \wb{W}\;\wb{X} }}}  = \Ep \brk{\exp \brc{\lambda \prn{ \<\wb{W}_{K}, \wb{X}_{K}\>_{\tspace_{K}} - \prod_{l=1}^K N_l \cdot \<\wb{W}_0, \wb{X}_0\>_{\tspace_0} }}} \\
	& =  \Ep \brk{\Ep \brk{\exp \brc{\lambda \prn{ M_K - M_{K-1} }} \mid \mc{F}_{K-1}} \cdot \exp \brc{\lambda \prn{ M_{K-1} - M_0 }}} \\
	& \leq \exp \brc{\frac{\lambda^2}{2} \cdot \prod_{l=0}^K N_l \cdot \frac{1 + \epsilon_{N_K}}{N_K} \cdot \norm{W_K \times_K \proj_{\ones_{N_K}}^\perp }_{\tspace}^2}\cdot \Ep \brk{\exp \brc{\lambda \prn{ M_{K-1} - M_0 }}} \\
	& \leq \cdots \leq \exp \brc{\frac{\lambda^2 }{2} \cdot \prod_{l=0}^K N_l \cdot \sum_{k=1}^{K} \frac{1 + \epsilon_{N_k}}{N_k}  \cdot \norm{W_{k} \times_{k} \proj_{\ones_{N_{k}}}^\perp }_{\tspace}^2}.
\end{align*}
We can then obtain the MGF bound by identifying $\prod_{l=1}^K N_l \cdot \wb{W}\;\wb{X} = \E[\<W, X\>_{\tspace} \mid \wb{X}]$. The tail bound follows from applying Lemma~\ref{lem:MGF-tail}.

\subsection{Proof of Theorem~\ref{thm:tensor-bernstein}} \label{proof:tensor-bernstein}

Similar to the proof in Appendix~\ref{proof:tensor-hoeffding}, we begin by constructing the filtration $\{\mc{F}_k\}_{k=0}^K$ for this proof as $\mc{F}_k = \sigma(\wb{X}_k, \{\sigma_{X,l}^2, \wt{\sigma}_{X,l}^2\}_{l\in[K]})$ when $k \leq K-1$ and $\mc{F}_K = \sigma(X)$. Since $\sigma_{X,k+1}^2,\wt{\sigma}_{X,k+1}^2 \in \sigma(\wb{X}_{k+1}) \subset \mc{F}_{k+1}$, we indeed have $ \mc{F}_0 \subset \cdots \subset \mc{F}_K$. The variance parameters are invariant under permutations within each mode, so conditioning on them preserves the mode exchangeability used below. We then have the Bernstein type lemma conditioned on $\mc{F}_k$.

\begin{lemma} \label{lem:tensor-bernstein-jensen}
	For any $k =0, 1, \cdots, K-1$ and the filtration $\{\mc{F}_k\}_{k=0}^K$ defined above, let $N_0=1$ and we have for $|\lambda| < \frac{3}{2 \linf{\wb{w}_{k+1}} (1 + \epsilon_{N_{k+1}})}$ that
	\begin{align*}
		& \Ep \brk{\exp \brc{\lambda \prn{ \<\wb{W}_{k+1}, \wb{X}_{k+1}\>_{\tspace_{k+1}} - N_{k+1} \<\wb{W}_k, \wb{X}_k\>_{\tspace_k}} } \mid \mc{F}_k} \nonumber \\
		& \leq \exp \brc{\frac{\lambda^2 (1+ \epsilon_{N_{k+1}})}{2 \prn{1 - \frac{2|\lambda|}{3} \linf{\wb{w}_{k+1}}  (1+ \epsilon_{N_{k+1}})}} \cdot  \norm{\wb{W}_{k+1} \times_{k+1} \proj_{\ones_{N_{k+1}}}^\perp }_{\tspace_{k+1}}^2 \cdot \wt{\sigma}_{X,k+1}^2}.
	\end{align*} 
\end{lemma}
A proof is in Appendix~\ref{proof:tensor-bernstein-jensen}. Applying the preceding lemma $K$ times, with $\lambda \prod_{l=k+1}^K N_l$ in place of $\lambda$ at mode $k$, yields
\begin{align*}
\Ep \brk{\exp \brc{\lambda \prn{ \<W, X\>_{\tspace} - \prod_{l=1}^K N_l \cdot \wb{W}\;\wb{X} }} \mid \mc{F}_0} \leq \exp \brc{\frac{\lambda^2}{2} \cdot \sum_{k=1}^K\frac{ c_k \cdot \norm{W_{k} \times_{k} \proj_{\ones_{N_{k}}}^\perp }_{\tspace}^2 \cdot \wt{\sigma}_{X,k}^2  }{1 - \frac{2|\lambda|}{3} c_k \linf{\wb{w}_{k}} } } .
\end{align*}
The proof for the MGF bound is done taking conditional expectation w.r.t. $\wt{\mc{F}}_{X}$ on both sides, using $\wt{\sigma}_{X, 1}^2, \cdots, \wt{\sigma}_{X, K}^2 \in \wt{\mc{F}}_X \subset \mc{F}_0$. Further from
\begin{align*}
	& \frac{\lambda^2}{2} \cdot \sum_{k=1}^K\frac{ c_k \cdot \norm{W_{k} \times_{k} \proj_{\ones_{N_{k}}}^\perp }_{\tspace}^2 \cdot \wt{\sigma}_{X,k}^2  }{1 - \frac{2|\lambda|}{3} c_k \linf{\wb{w}_{k}} } \\
	& \leq \frac{\lambda^2 \sum_{k=1}^K  c_k \cdot \norm{W_{k} \times_{k} \proj_{\ones_{N_{k}}}^\perp }_{\tspace}^2 \cdot \wt{\sigma}_{X,k}^2 }{2 \prn{1 - \frac{2}{3} \max_{k \in [K]} c_k \linf{\wb{w}_{k}}  \cdot |\lambda| }},
\end{align*}
the tail bound follows from applying Lemma~\ref{lem:MGF-tail}.

\subsection{Proof of Lemma~\ref{lem:tensor-hoeffding-plain}} \label{proof:tensor-hoeffding-plain}
We first construct a finer filtration $\{\mc{F}_{k,i}\}_{i=0}^{N_{k+1}}$ bridging from $\mc{F}_{k,0} = \sigma(\wb{X}_k) = \mc{F}_k$ to $\mc{F}_{k, N_{k+1}} = \sigma(\wb{X}_{k+1}) = \mc{F}_{k+1}$,
where for $i = 1, 2, \cdots, N_{k+1}$,
\begin{align*}
	\mc{F}_{k,i} = \sigma(\mc{F}_{k,i-1}, \{(\wb{X}_{k+1})_{i_1\cdots i_k i}\}_{i_l \in [N_l], l \in [k]}) = \sigma(\wb{X}_k, \{(\wb{X}_{k+1})_{i_1\cdots i_k i_{k+1}}\}_{i_l \in [N_l], l \in [k], i_{k+1} \in [i]}).
\end{align*} 
For any fixed $i_l \in [N_l], l \in [k]$, we use the shorthand $i_1^k$ for the index group $i_1\cdots i_k$ and denote by $(\wb{X}_{k+1})_{i_1^k, \geq i} = \frac{1}{N_{k+1} + 1 - i} \sum_{i_{k+1}=i}^{N_{k+1}} (\wb{X}_{k+1})_{i_1^k i_{k+1}}$. It is clear that $(\wb{X}_{k+1})_{i_1^k, \geq i} = \frac{1}{N_{k+1} + 1 - i} (N_{k+1} (\wb{X}_k)_{i_1^k}-\sum_{i_{k+1}=1}^{i-1} (\wb{X}_{k+1})_{i_1^k i_{k+1}} \in \mc{F}_{k,i-1}$, and by mode-exchangeability along the $(k+1)$-mode,
\begin{align*}
	\mathbb{E} \brk{(\wb{X}_{k+1})_{i_1^k i} - (\wb{X}_{k+1})_{i_1^k, \geq i}  \mid \mc{F}_{k, i-1}} = 0.
\end{align*}
This allows us to apply the standard Hoeffding's lemma~(e.g., see~\cite[Prop.~2.5]{wainwright2019high}) to a fixed $\wb{V} \in \tspace_{k+1}$,
\begin{align*}
	& \mathbb{E} \brk{\exp \brc{\lambda \sum_{i_1^k \in [N_1] \times \cdots \times [N_k]} \wb{V}_{i_1^k i} \cdot \prn{ (\wb{X}_{k+1})_{i_1^k i} - (\wb{X}_{k+1})_{i_1^k, \geq i} } } \mid \mc{F}_{k, i-1}} \\
	& \leq \exp \brc{\frac{\lambda^2}{2} \cdot \prn{\sum_{i_1^k \in [N_1] \times \cdots \times [N_k]} \left|\wb{V}_{i_1^k i}\right|}^2} \leq \exp \brc{\frac{\lambda^2}{2} \cdot \prod_{l=0}^k N_l \cdot \sum_{i_1^k \in [N_1] \times \cdots \times [N_k]} \wb{V}_{i_1^k i}^2},
\end{align*}
where the last inequality is a consequence of Cauchy-Schwartz. By telescoping products and sequentially conditioning on $\mc{F}_{k, i-1}$, it then follows
\begin{align*}
	& \mathbb{E} \brk{\exp \brc{\lambda \sum_{i_1^k \in [N_1] \times \cdots \times [N_k]} \sum_{i=1}^{N_{k+1}} \wb{V}_{i_1^k i} \cdot \prn{ (\wb{X}_{k+1})_{i_1^k i} - (\wb{X}_{k+1})_{i_1^k, \geq i} } } \mid \mc{F}_{k}}  \leq \exp \brc{\frac{\lambda^2}{2} \cdot \prod_{l=0}^k N_l \cdot \norm{\wb{V}}_{\tspace_{k+1}}^2}.
\end{align*}
For a tensor $A \in \tspace_{k+1}$ and the first $k$ indices $i_1^k$, we represent $A_{i_1^k} := (A_{i_1^k 1}, \cdots, A_{i_1^k N_{k+1}})^\top \in \R^{N_{k+1}}$, which allows us to write 
\begin{align}
	& \sum_{i_1^k \in [N_1] \times \cdots \times [N_k]} \sum_{i=1}^{N_{k+1}} \wb{V}_{i_1^k i} \cdot \prn{ (\wb{X}_{k+1})_{i_1^k i} - (\wb{X}_{k+1})_{i_1^k, \geq i}} = 	\sum_{i_1^k \in [N_1] \times \cdots \times [N_k]} \wb{V}_{i_1^k}^\top (I_{N_{k+1}}  - \mathsf{B}_{N_{k+1}} ) (\wb{X}_{k+1})_{i_1^k} \nonumber \\
	& = \< \wb{V} , \wb{X}_{k+1} \times_{k+1} (I_{N_{k+1}} - \mathsf{B}_{N_{k+1}} ) \>_{\tspace_{k+1}} = \< \wb{V} \times_{k+1} (I_{N_{k+1}}  - \mathsf{B}_{N_{k+1}} )^\top , \wb{X}_{k+1}  \>_{\tspace_{k+1}}. \label{eq:telescoping-rhs}
\end{align}
Finally, we substitute $\wb{V} = \wb{U}_{k+1} \times_{k+1} \mathsf{A}_{N_{k+1}} $ and use $\wb{U}_{k+1} \times_{k+1} \mathsf{A}_{N_{k+1}}  \times_{k+1} (I_{N_{k+1}}  - \mathsf{B}_{N_{k+1}} )^\top = \wb{U}_{k+1} \times_{k+1} (I_{N_{k+1}}  - \mathsf{B}_{N_{k+1}} )^\top\mathsf{A}_{N_{k+1}}  = \wb{U}_{k+1} \times_{k+1} \proj_{\ones_{N_{k+1}}}^\perp$,
we obtain
\begin{align*}
	& \< \wb{V} \times_{k+1} (I_{N_{k+1}}  - \mathsf{B}_{N_{k+1}} )^\top , \wb{X}_{k+1}  \>_{\tspace_{k+1}}  = 	\< \wb{U}_{k+1} \times_{k+1} \proj_{\ones_{N_{k+1}}}^\perp , \wb{X}_{k+1}  \>_{\tspace_{k+1}} \\
	& = 	\< \wb{U}_{k+1}, \wb{X}_{k+1}  \times_{k+1} \proj_{\ones_{N_{k+1}}}^\perp  \>_{\tspace_{k+1}} = 	\< \wb{U}_{k+1}, \wb{X}_{k+1}  - \wb{X}_k \otimes \frac{1}{N_{k+1}} \ones_{N_{k+1}}\>_{\tspace_{k+1}} \nonumber \\
	& = \< \wb{U}_{k+1}, \wb{X}_{k+1}\>_{\tspace_{k+1}} - N_{k+1} \< \wb{U}_k \otimes \frac{1}{N_{k+1}}  \ones_{N_{k+1}}, \wb{X}_k \otimes \frac{1}{N_{k+1}} \ones_{N_{k+1}}\>_{\tspace_{k+1}} \nonumber \\
	& = \< \wb{U}_{k+1}, \wb{X}_{k+1}\>_{\tspace_{k+1}} - N_{k+1}  \< \wb{U}_k, \wb{X}_k\>_{\tspace_k}.
\end{align*}
We then conclude the proof with the above displays
\begin{align*}
	& \Ep \brk{\exp \brc{\lambda \prn{ \<\wb{U}_{k+1}, \wb{X}_{k+1}\>_{\tspace_{k+1}} - N_{k+1}  \<\wb{U}_k, \wb{X}_k\>_{\tspace_k}} } \mid \mc{F}_k} \nonumber \\
	& = \mathbb{E} \brk{\exp \brc{\lambda \sum_{i_1^k \in [N_1] \times \cdots \times [N_k]} \sum_{i=1}^{N_{k+1}} \wb{V}_{i_1^k i} \cdot \prn{ (\wb{X}_{k+1})_{i_1^k i} - (\wb{X}_{k+1})_{i_1^k, \geq i} } } \mid \mc{F}_{k}} \leq \exp \brc{\frac{\lambda^2}{2} \cdot \prod_{l=0}^k N_l \cdot \norm{\wb{V}}_{\tspace_{k+1}}^2} \nonumber \\
	& = \exp \brc{\frac{\lambda^2}{2} \cdot \prod_{l=0}^k N_l \cdot \norm{\wb{U}_{k+1} \times_{k+1} \mathsf{A}_{N_{k+1}}}_{\tspace_{k+1}}^2} 
\end{align*}

\subsection{Proof of Lemma~\ref{lem:tensor-bernstein-jensen}} \label{proof:tensor-bernstein-jensen}
Similarly to the proof in Appendix~\ref{proof:tensor-hoeffding-plain}, for $i = 1, 2, \cdots, N_{k+1}$, we let
\begin{align*}
	\mc{F}_{k,i} & = \sigma(\mc{F}_{k,i-1}, \{(\wb{X}_{k+1})_{i_1\cdots i_k i}\}_{i_l \in [N_l], l \in [k]}) \nonumber \\
	& = \sigma(\wb{X}_k, \{\sigma^2_{X, l}, \wt{\sigma}_{X, l}^2\}_{l \in [K]}, \{(\wb{X}_{k+1})_{i_1\cdots i_k i_{k+1}}\}_{i_l \in [N_l], l \in [k], i_{k+1} \in [i]}).
\end{align*} 
For any permutation $\pi_{k+1} \in \permset_{N_{k+1}}$ and its associated permutation matrix $\Pi_{k+1} \in \permmapset_{N_{k+1}}$, we define
\begin{align*}
	\wb{V}_{k+1}^{\pi_{k+1}} := \wb{W}_{k+1} \times_{k+1} \proj_{\ones_{N_{k+1}}}^\perp \times_{k+1} \Pi_{k+1},
\end{align*}
and consider the random variables $Z_i^{\pi_{k+1}} := \sum_{i_1^k \in [N_1] \times \cdots \times [N_k]} \wb{V}_{i_1^k i} \cdot \prn{ (\wb{X}_{k+1})_{i_1^k i} - (\wb{X}_{k+1})_{i_1^k, \geq i} } $. From the following bound
\begin{align*}
	\left| Z_i^{\pi_{k+1}} \right|  & \leq  2\sum_{i_1^k \in [N_1] \times \cdots \times [N_k]} \left|\wb{V}_{i_1^k i}\right| =2\sum_{i_1^k \in [N_1] \times \cdots \times [N_k]} \left|\prn{\wb{W}_{k+1} \times \proj_{\ones_{N_{k+1}}}^\perp}_{i_1^k {\pi_{k+1}}(i)}\right| \leq 2\linf{\wb{w}_{k+1}}, 
\end{align*}
and the fact that conditioned on $\mc{F}_{k, i-1}$, by mode-exchangeability, $\mathbb{E} [Z_i^{\pi_{k+1}} \mid \mc{F}_{k,i-1}] = 0$ and
\begin{align*}
	& \var \prn{Z_i^{\pi_{k+1}} \mid \mc{F}_{k,i-1}} \nonumber \\
    & \leq \prn{\sum_{i_1^k \in [N_1] \times \cdots \times [N_k]} \wb{V}_{i_1^k i}^2 } \cdot \frac{1}{N_{k+1}+1-i} \sum_{l=i}^{N_{k+1}} \prn{\sum_{i_1^k \in [N_1] \times \cdots \times [N_k]} \prn{ (\wb{X}_{k+1})_{i_1^k l} - (\wb{X}_{k+1})_{i_1^k, \geq i} }^2} \\
	& \leq \prn{\sum_{i_1^k \in [N_1] \times \cdots \times [N_k]} \wb{V}_{i_1^k i}^2 } \cdot \frac{1}{N_{k+1}+1-i} \sum_{l=i}^{N_{k+1}} \underbrace{\prn{\sum_{i_1^k \in [N_1] \times \cdots \times [N_k]} \prn{ (\wb{X}_{k+1})_{i_1^k l} - (\wb{X}_{k})_{i_1^k} }^2}}_{=: \wt{X}_{k+1,l}} \\
	& = \underbrace{\prn{\sum_{i_1^k \in [N_1] \times \cdots \times [N_k]} \prn{\wb{W}_{k+1} \times \proj_{\ones_{N_{k+1}}}^\perp}_{i_1^k {\pi_{k+1}}(i)}^2 }}_{=:\wt{W}_{k+1, {\pi_{k+1}}(i)}} \cdot \frac{1}{N_{k+1}+1-i} \sum_{l=i}^{N_{k+1}} \wt{X}_{k+1,l},
\end{align*}
we can conclude from the Bernstein's lemma~(e.g., see~\cite[Prop.~2.10]{wainwright2019high}) that for $|\lambda| < \frac{3}{2 \linf{\wb{w}_{k+1}}}$,
\begin{align*}
	& \mathbb{E} \brk{\exp \brc{\lambda \sum_{i_1^k \in [N_1] \times \cdots \times [N_k]} \wb{V}_{i_1^k i} \cdot \prn{ (\wb{X}_{k+1})_{i_1^k i} - (\wb{X}_{k+1})_{i_1^k, \geq i} } } \mid \mc{F}_{k, i-1}} \\
	& \leq \exp \brc{\frac{\lambda^2 }{2 \prn{1 - \frac{2|\lambda|}{3} \linf{\wb{w}_{k+1}}  }} \cdot \wt{W}_{k+1, {\pi_{k+1}}(i)} \cdot \frac{1}{N_{k+1}+1-i} \sum_{l=i}^{N_{k+1}} \wt{X}_{k+1,l}}.
\end{align*}
Denoting by $\wt{W}_{k+1}$ and $\wt{X}_{k+1}$ the vectors $(\wt{W}_{k+1, 1}, \cdots, \wt{W}_{k+1, N_{k+1}})^\top$ and $(\wt{X}_{k+1, 1}, \cdots, \wt{X}_{k+1, N_{k+1}})^\top$ in $\R^{N_{k+1}}$, the above display has the compact form of
\begin{align*}
	& \mathbb{E} \brk{\exp \brc{\lambda \sum_{i_1^k \in [N_1] \times \cdots \times [N_k]} \wb{V}_{i_1^k i} \cdot \prn{ (\wb{X}_{k+1})_{i_1^k i} - (\wb{X}_{k+1})_{i_1^k, \geq i} } } \mid \mc{F}_{k, i-1}} \\
	& \leq \exp \brc{\frac{\lambda^2 }{2 \prn{1 - \frac{2|\lambda|}{3} \linf{\wb{w}_{k+1}}  }} \cdot (\Pi_{k+1} \wt{W}_{k+1})^\top e_i e_i^\top \mathsf{B}_{N_{k+1}} \wt{X}_{k+1}}.
\end{align*}
For the remainder of the proof, we write $\mathbb{E}_k\brk{\cdot}:=\mathbb{E}\brk{\cdot\mid\mc{F}_k}$ for this proof. Since $\sum_{l=i}^{N_{k+1}}\wt{X}_{k+1,l}=N_{k+1}\sigma_{X,k+1}^2-\sum_{l<i}\wt{X}_{k+1,l}$, the variance bound above is $\mc{F}_{k,i-1}$-measurable. By iterating conditional expectations and using Eq.~\eqref{eq:telescoping-rhs}, we obtain
\begin{align*}
& \mathbb{E}_k \Bigg[\exp\Bigg\{\lambda \< \wb{W}_{k+1} \times_{k+1} \proj_{\ones_{N_{k+1}}}^\perp \times_{k+1} (I_{N_{k+1}} - \mathsf{B}_{N_{k+1}})^\top \Pi_{k+1}, \wb{X}_{k+1} \>_{\tspace_{k+1}} \\*
& \qquad {}- \frac{\lambda^2}{2 \prn{1 - \frac{2|\lambda|}{3} \linf{\wb{w}_{k+1}}}} \cdot \wt{W}_{k+1}^\top \Pi_{k+1}^\top \mathsf{B}_{N_{k+1}} \wt{X}_{k+1}\Bigg\}\Bigg] \leq 1.
\end{align*}
Conditionally on $\mc{F}_k$, $(\wb{X}_{k+1},\wt{X}_{k+1})\eqdst(\wb{X}_{k+1}\times_{k+1}\Pi_{k+1},\Pi_{k+1}\wt{X}_{k+1})$, we can again apply Jensen's inequality to the entire exponent to obtain
\begin{align*}
& \mathbb{E}_k \Bigg[\exp\Bigg\{\lambda \< \wb{W}_{k+1} \times_{k+1} \proj_{\ones_{N_{k+1}}}^\perp \times_{k+1} \frac{1}{N_{k+1}!} \sum_{\Pi_{k+1} \in \permset_{N_{k+1}}} \Pi_{k+1}^\top (I_{N_{k+1}} - \mathsf{B}_{N_{k+1}})^\top \Pi_{k+1}, \wb{X}_{k+1} \>_{\tspace_{k+1}} \\*
& \qquad {}- \frac{\lambda^2}{2 \prn{1 - \frac{2|\lambda|}{3} \linf{\wb{w}_{k+1}}}} \cdot \wt{W}_{k+1}^\top \frac{1}{N_{k+1}!} \sum_{\Pi_{k+1} \in \permset_{N_{k+1}}} \Pi_{k+1}^\top \mathsf{B}_{N_{k+1}} \Pi_{k+1} \wt{X}_{k+1}\Bigg\}\Bigg] \leq 1.
\end{align*}
Given $(I_{N_{k+1}}  - \mathsf{B}_{N_{k+1}} ) \ones_{N_{k+1}} = 0$, we refer to Lemma~\ref{lem:perm-avg-matrix} and compute
\begin{align*}
	 &\frac{1}{N_{k+1}!} \sum_{\Pi_{k+1} \in \permset_{N_{k+1}}} \Pi_{k+1}^\top (I_{N_{k+1}}  - \mathsf{B}_{N_{k+1}} )^\top \Pi_{k+1} = \frac{N_{k+1} - \Tr(B_{N_{k+1}})}{N_{k+1} - 1} \cdot \proj_{\ones_{N_{k+1}}}^\perp = \frac{1}{1 + \epsilon_{N_{k+1}}} \cdot \proj_{\ones_{N_{k+1}}}^\perp , \\
	&\frac{1}{N_{k+1}!} \sum_{\Pi_{k+1} \in \permset_{N_{k+1}}} \Pi_{k+1}^\top \mathsf{B}_{N_{k+1}} \Pi_{k+1}  = I - \frac{1}{1 + \epsilon_{N_{k+1}}} \cdot \proj_{\ones_{N_{k+1}}}^\perp = \frac{1}{1 + \epsilon_{N_{k+1}}} \cdot \proj_{\ones_{N_{k+1}}} + \frac{\epsilon_{N_{k+1}}}{1 + \epsilon_{N_{k+1}}} \cdot I_{N_{k+1}}.
\end{align*}
Therefore, the variance term has the upper bound
\begin{align*}
	& \wt{W}_{k+1}^\top \frac{1}{N_{k+1}!} \sum_{\Pi_{k+1} \in \permset_{N_{k+1}}} \Pi_{k+1}^\top \mathsf{B}_{N_{k+1}} \Pi_{k+1} \wt{X}_{k+1} \\
	& \leq \frac{1}{1 + \epsilon_{N_{k+1}}} \cdot \wt{W}_{k+1}^\top \ones_{N_{k+1}} \cdot \frac{1}{N_{k+1}} \ones_{N_{k+1}}^\top \wt{X}_{k+1} + \frac{\epsilon_{N_{k+1}}}{1 + \epsilon_{N_{k+1}}} \wt{W}_{k+1}^\top \wt{X}_{k+1}  \\
	& \leq \frac{1}{1 + \epsilon_{N_{k+1}}} \cdot \lone{\wt{W}_{k+1}} \cdot \prn{\frac{1}{N_{k+1}} \ones_{N_{k+1}}^\top \wt{X}_{k+1} + \epsilon_{N_{k+1}} \linf{\wt{X}_{k+1}}} \\
	& \leq \frac{1}{1 + \epsilon_{N_{k+1}}} \cdot \norm{\wb{W}_{k+1} \times_{k+1} \proj_{\ones_{N_{k+1}}}^\perp }_{\tspace_{k+1}}^2 \cdot \wt{\sigma}_{X, k+1}^2,
\end{align*}
where in the last line we use $\lone{\wt{W}_{k+1}} = \norm{\wb{W}_{k+1} \times_{k+1} \proj_{\ones_{N_{k+1}}}^\perp }_{\tspace_{k+1}}^2$ and the Cauchy-Schwartz inequality $\linf{\wt{X}_{k+1}} \leq \prod_{l=0}^k N_l \cdot \norm{\wb{X}_{k+1} \times_{k+1} \proj_{\ones_{N_{k+1}}}^\perp}_{\infty}^2$. Since $\wt{\sigma}_{X,k+1}^2$ is $\mc{F}_k$-measurable, replacing the variance term by this upper bound allows its exponential to be taken outside the conditional expectation. The proof is complete by reparameterization of $\lambda$ to $\lambda \cdot (1 + \epsilon_{N_{k+1}})$.

\section{Proofs for Section~\ref{sec:master-matrix-valued-bounds}}

\subsection{Proof of Theorem~\ref{thm:matrix-valued-hoeffding}} \label{proof:matrix-valued-hoeffding}
In this proof, we use the filtration $\{\mc{F}_k\}_{k=0}^N$ where $\mc{F}_0 = \sigma (\wb{X})$ and $\mc{F}_{k} = \sigma( \mc{F}_{k-1}, X_k)$. If Assumption~\ref{assmp:commutativity} holds, we also have $\mc{F}_k = \sigma(\wb{X}, \wb{\wt{X}}, \{X_i, \wt{X}_i\}_{i=1}^k)$. We first introduce a lemma for the symmetrized matrix cumulant generating function (CGF) by Rademacher. See \cite[Lemma 4.3]{tropp2012user} for its proof.
\begin{lemma}[Hoeffding CGF bound] \label{lem:rademacher-cgf-bound}
	For a fixed symmetric matrix $A \in \mathbb{S}^n$, let $\epsilon$ be a Rademacher. We have $$\log \E \exp(\epsilon A) \preceq \frac{A^2}{2}.$$
	Here $\log$ is the functional inverse on the positive definite cone $\mathbb{S}_{++}^n$ of $\exp$ on $\mathbb{S}^n$.
\end{lemma}
\paragraph{Step 1: MGF bounds for $\{X_k - X_{\geq k}\}_{k=1}^N$ by the martingale argument} For any $U \in \R^{N \times p \times q}$, $k \in [N]$ and $H_{k-1} \in \mathbb{S}^{p+r}$ such that $H_{k-1} \in \mc{F}_{k-1}$ is adaptable, let $X_{\geq k} := \frac{1}{N+1-k} \sum_{i=k}^N X_i$. By $\E \brk{X_k - X_{\geq k} \mid \mc{F}_{k-1}}= 0$ and the symmetrization lemma (cf.~Lemma~\ref{lem:convex-concave-exp}), 
\begin{align*}
	\mathsf{M}_k(H_{k-1}) &  : = \E \brk{\Tr \exp \brc{H_{k-1} + \lambda \begin{bmatrix} 0 & U_k(X_k - X_{\geq k}) \\ (X_k - X_{\geq k})^\top U_k^\top & 0\end{bmatrix}} \mid \mc{F}_{k-1}} \nonumber \\
	& \leq  \E \brk{\Tr \exp \brc{H_{k-1} + \epsilon \cdot 2\lambda \begin{bmatrix} 0 & U_k(X_k - X_{\geq k}) \\ (X_k - X_{\geq k})^\top U_k^\top & 0\end{bmatrix}} \mid \mc{F}_{k-1}}.
\end{align*}
Using Lieb's inequality and Lemma~\ref{lem:rademacher-cgf-bound}, it further holds that
\begin{align*}
	& \E \brk{\Tr \exp \brc{H_{k-1} + \lambda \begin{bmatrix} 0 & U_k(X_k - X_{\geq k}) \\ (X_k - X_{\geq k})^\top U_k^\top & 0\end{bmatrix}} \mid \mc{F}_{k-1}}  \nonumber \\
	& \leq  \E \brk{\E \brk{\Tr \exp \brc{H_{k-1} + \epsilon \cdot 2\lambda \begin{bmatrix} 0 & U_k(X_k - X_{\geq k}) \\ (X_k - X_{\geq k})^\top U_k^\top & 0\end{bmatrix}} \mid \mc{F}_N} \mid \mc{F}_{k-1}} \nonumber \\
	& \leq \E \brk{\Tr \exp \brc{H_{k-1} + \log \E \brk{ \exp \brc{\epsilon \cdot 2\lambda \begin{bmatrix} 0 & U_k(X_k - X_{\geq k}) \\ (X_k - X_{\geq k})^\top U_k^\top & 0\end{bmatrix}} \mid \mc{F}_N}} \mid \mc{F}_{k-1}} \nonumber \\
	& \leq  \E \brk{\Tr \exp \brc{H_{k-1} +  2\lambda^2 \begin{bmatrix} U_k(X_k - X_{\geq k})(X_k - X_{\geq k})^\top U_k^\top & 0 \\ 0 & (X_k - X_{\geq k})^\top U_k^\top U_k (X_k - X_{\geq k}) \end{bmatrix}}  \mid \mc{F}_{k-1}},
\end{align*}
where in the last inequality we make use of Lemma~\ref{lem:rademacher-cgf-bound} and the monotonicity of $\Tr \exp$. From $\norm{X_k} \leq 1$, it holds
\begin{align*}
	& \begin{bmatrix} U_k(X_k - X_{\geq k})(X_k - X_{\geq k})^\top U_k^\top & 0 \\ 0 & (X_k - X_{\geq k})^\top U_k^\top U_k (X_k - X_{\geq k}) \end{bmatrix} \nonumber \\
	& \preceq \begin{bmatrix} 4U_k U_k^\top & 0 \\ 0 &(X_k - X_{\geq k})^\top \cdot \norm{U_k^\top U_k} I_q \cdot (X_k - X_{\geq k}) \end{bmatrix} \preceq \begin{bmatrix} 4U_k U_k^\top & 0 \\ 0 & 4\norm{U_k^\top U_k} I_r \end{bmatrix} \preceq 4 \norm{U_k}^2 I_{p+r},
\end{align*}
and additionally given Assumption~\ref{assmp:commutativity} holds for $(U, X)$ with its commutativity pair $(\wt{X}, Z)$,
\begin{align*}
	& \begin{bmatrix} U_k(X_k - X_{\geq k})(X_k - X_{\geq k})^\top U_k^\top & 0 \\ 0 & (X_k - X_{\geq k})^\top U_k^\top U_k (X_k - X_{\geq k}) \end{bmatrix} \nonumber \\
	& = \begin{bmatrix} U_k(X_k - X_{\geq k})(X_k - X_{\geq k})^\top U_k^\top & 0 \\ 0 & Z_k^\top (\wt{X}_k - \wt{X}_{\geq k})^\top  (\wt{X}_k - \wt{X}_{\geq k}) Z_k \end{bmatrix} \preceq \begin{bmatrix} 4U_k U_k^\top & 0 \\ 0 & 4 Z_k^\top Z_k \end{bmatrix}.
\end{align*}
This step is exactly where commutativity constraints provide better control compared to that under only exchangeability. Gathering the above results and using the shorthand $\E_k [\cdot ] := \E \brk{\cdot \mid \mc{F}_k}$, we can summarize them into the lemma as follows.
\begin{lemma}
	For any $k = 0, 1, \cdots, N-1$ and the filtration $\{\mc{F}_k\}_{k=0}^N$ defined in this section. We have 
	\begin{align*}
		\mathsf{M}_{k+1}(H_k) & \leq \begin{cases}
 		\E_k \brk{\Tr \exp (H_k + 8 \lambda^2 \norm{U_{k+1}}^2 I_{p+r})}, & \text{in general};	\\
			\E_k \brk{\Tr \exp \prn{H_k + 8 \lambda^2 \begin{bmatrix} U_{k+1} U_{k+1}^\top & 0 \\ 0 &  Z_{k+1}^\top Z_{k+1} \end{bmatrix}}}, & \text{If Assumption~\ref{assmp:commutativity} holds}.
		\end{cases}
	\end{align*}
\end{lemma}
We apply the preceding lemma $N$ times for some $H_0 \in \mathbb{S}^{p+r}$ and $$H_k =  H_0 + \lambda \begin{bmatrix} 0 & \sum_{i=1}^k U_i(X_i - X_{\geq i}) \\ \sum_{i=1}^k (X_i - X_{\geq i})^\top U_i^\top & 0\end{bmatrix}.$$
We then have
\begin{align*}
	& \Ep \brk{\Tr \exp \brc{H_0 + \lambda \begin{bmatrix} 0 & \sum_{k=1}^N U_k(X_k - X_{\geq k}) \\ \sum_{k=1}^N (X_k - X_{\geq k})^\top U_k^\top & 0\end{bmatrix}}} = \E \brk{\mathsf{M}_N(H_{N-1})}\nonumber \\
	&  \leq   \E \brk{\E_N \brk{\Tr \exp (H_{N-1} + 8 \lambda^2 \norm{U_{N}}^2 I_{p+r})}}   = \E \brk{\E_{N-1} \brk{\Tr \exp (H_{N-1} + 8 \lambda^2 \norm{U_{N}}^2 I_{p+r})}} \nonumber \\
	& =\E \brk{\mathsf{M}_{N-1} \prn{H_{N-2} + 8 \lambda^2 \norm{U_{N}}^2 I_{p+r}}} \leq  \E \brk{\mathsf{M}_{N-2} \prn{H_{N-3} + 8 \lambda^2 \norm{U_{N-1}}^2 I_{p+r} + 8 \lambda^2 \norm{U_{N}}^2 I_{p+r}}} \\
	&  \leq \cdots \leq \E \brk{\Tr \exp \prn{H_0 + 8 \lambda^2 \sum_{k=1}^N \norm{U_k}^2 I_{p+r}}}.
\end{align*}
Setting $H_0 = - 8 \lambda^2 \sum_{k=1}^N \norm{U_k}^2 I_{p+r}$ yields 
\begin{align*}
	& \frac{1}{p+r} \Ep \brk{\Tr \exp \brc{\lambda \begin{bmatrix} 0 & \sum_{k=1}^N U_k(X_k - X_{\geq k}) \\ \sum_{k=1}^N (X_k - X_{\geq k})^\top U_k^\top & 0\end{bmatrix} - 8 \lambda^2 \sum_{k=1}^N \norm{U_k}^2 I_{p+r}}} \leq 1.
\end{align*}
If Assumption~\ref{assmp:commutativity} holds, we have similarly
\begin{align*}
	& \frac{1}{p+r} \Ep \brk{\Tr \exp \brc{\lambda \begin{bmatrix} 0 & \sum_{k=1}^N U_k(X_k - X_{\geq k}) \\ \sum_{k=1}^N (X_k - X_{\geq k})^\top U_k^\top & 0\end{bmatrix} - 8 \lambda^2  \begin{bmatrix} \sum_{k=1}^N U_k U_k^\top & 0 \\ 0 &  \sum_{k=1}^N Z_k^\top Z_k \end{bmatrix}}} \leq 1.
\end{align*}
\paragraph{Step 2: MGF bounds for $\{X_k - \wb{X}\}_{k=1}^N$ in the general case} Let $u = (\norm{U_1}, \cdots, \norm{U_N}) \in \R^N$, $U^\top \in \R^{N \times q \times p}$, $Z^\top \in \R^{N \times q' \times r}$ and similarly $X^\top, \wt{X}^\top, W^\top, V^\top$. We identify that $\sum_{k=1}^N U_k (X_k - X_{\geq k}) = \llangle U, X \times_1 (I- \mathsf{B}_N) \rrangle$ and $\sum_{k=1}^N \norm{U_k}^2 = \norm{u}^2$. Setting $U = W \times_1 \Pi $ and $X = X \times_1 \Pi$ yields
\begin{subequations}
	\begin{align}
		\llangle U, X \times_1 (I- \mathsf{B}_N) \Pi   \rrangle & = \llangle W \times_1 \Pi^\top (I-\mathsf{B}_N)^\top  \Pi, X \rrangle, \label{eq:mid-matrix-Hoeffding-U-X-general} \\
			\norm{u}^2 & = \sum_{k=1}^N \norm{W_k}^2. \label{eq:mid-matrix-Hoeffding-u-general}
	\end{align}
\end{subequations}
By convexity of $\Tr \exp$ and substitution of Eqs.~\eqref{eq:mid-matrix-Hoeffding-U-X-general} and \eqref{eq:mid-matrix-Hoeffding-u-general}, we can then average over all $\Pi \in \permset_N$ (see the proof in Appendix~\ref{proof:tensor-hoeffding} and \ref{proof:tensor-bernstein} for calculation details) such that
\begin{align*}
	& \frac{1}{p+r} \Ep \brk{\Tr \exp \brc{\frac{\lambda}{1 + \epsilon_N} \begin{bmatrix} 0 & \llangle W \times_1 \proj_{\ones_N}^\perp, X \rrangle \\ \llangle W \times_1 \proj_{\ones_N}^\perp, X \rrangle^\top & 0\end{bmatrix} - 8 \lambda^2  \cdot \sum_{k=1}^N \norm{W_k}^2 I_{p+r}}} \leq 1,
\end{align*}
\paragraph{Step 3: MGF bounds for $\{X_k - \wb{X}\}_{k=1}^N$ given commutativity conditions} Note that $\sum_{k=1}^N U_k U_k^\top = \llangle U, U^\top \rrangle$ and $\sum_{k=1}^N Z_k^\top Z_k = \llangle Z^\top, Z \rrangle$. For any $\Pi \in \permset_N$, we substitute in $U = W \times_1 \mathsf{A}_N(\mathsf{A}_N^\top \mathsf{A}_N)^\dagger \Pi$ and replace $X$ with $X \times_1 \Pi$. If Assumption~\ref{assmp:commutativity} holds for the commutativity pair $(W, X)$ and $(\wt{X}, V)$, one can easily check that $Z = V \times_1 \mathsf{A}_N(\mathsf{A}_N^\top \mathsf{A}_N)^\dagger \Pi$. Under these substitutions, the following identities hold
\begin{subequations}
\begin{align}
	 \llangle U, X \times_1 (I- \mathsf{B}_N) \Pi   \rrangle & = \llangle W \times_1 \Pi^\top (I-\mathsf{B}_N)^\top \mathsf{A}_N(\mathsf{A}_N^\top \mathsf{A}_N)^\dagger \Pi, X \rrangle = \llangle W \times_1 \Pi^\top (\mathsf{A}_N^\top \mathsf{A}_N)^\dagger \Pi, X \rrangle, \label{eq:mid-matrix-Hoeffding-U-X} \\
	 \llangle U, U^\top \rrangle &  =  \llangle W \times_1 \Pi^\top (\mathsf{A}_N^\top \mathsf{A}_N)^\dagger  \Pi, W^\top \rrangle, \label{eq:mid-matrix-Hoeffding-W-W} \\
	 \llangle Z^\top, Z \rrangle &  =  \llangle V^\top \times_1 \Pi^\top (\mathsf{A}_N^\top \mathsf{A}_N)^\dagger  \Pi, V \rrangle. \label{eq:mid-matrix-Hoeffding-Z-Z}
\end{align}
\end{subequations}
Again using convexity of $\Tr \exp$ and substituting Eqs.~\eqref{eq:mid-matrix-Hoeffding-U-X} to \eqref{eq:mid-matrix-Hoeffding-Z-Z}, we can then average over all $\Pi \in \permset_N$ when the commutativity conditions hold,
\begin{align*}
	& \frac{1}{p+r} \Ep \brk{\Tr \exp \brc{\frac{\lambda}{1 + \epsilon_N} \prn{\begin{bmatrix} 0 & \llangle W \times_1 \proj_{\ones_N}^\perp, X \rrangle \\ * & 0\end{bmatrix} - 8\lambda \begin{bmatrix} \sum_{k=1}^N W_k W_k^\top & 0 \\ 0 &  \sum_{k=1}^N V_k^\top V_k \end{bmatrix}}}} \leq 1,
\end{align*}
where $*$ denotes the transpose of $\llangle W \times_1 \proj_{\ones_N}^\perp, X \rrangle$.

\paragraph{Step 4: Concluding the proof}
The proof is done by invoking Golden-Thompson inequality and substituting $\lambda$ with $\lambda \cdot (1 + \epsilon_N)$. For instance, using the last inequality and denoting the matrix inside the trace exponential operator by $\Omega(\lambda)$, we have when Assumption~\ref{assmp:commutativity} holds,
\begin{align*}
	& \Ep \brk{\wb{\exp }\brc{\lambda \prn{ \llangle W, X \rrangle -  N\cdot \wb{W}\; \wb{X} }}} \\
	& = \frac{1}{p+r} \Ep \brk{\Tr \exp \brc{\lambda \begin{bmatrix} 0 & \llangle W \times_1 \proj_{\ones_N}^\perp, X \rrangle \\ \llangle W \times_1 \proj_{\ones_N}^\perp, X \rrangle^\top & 0\end{bmatrix}}} \nonumber \\
	& = \frac{1}{p+r} \Ep \brk{\Tr \exp \brc{\Omega(\lambda(1 + \epsilon_N)) + 8 \lambda^2(1+\epsilon_N) \cdot \begin{bmatrix} \sum_{k=1}^N W_k W_k^\top & 0 \\ 0 &  \sum_{k=1}^N V_k^\top V_k \end{bmatrix} }} \nonumber \\
	& \stackrel{\mathrm{(i)}}{\leq} \frac{1}{p+r} \Ep \brk{\Tr \exp \brc{\Omega(\lambda(1 + \epsilon_N))} \cdot \exp \brc{8 \lambda^2(1+\epsilon_N) \cdot \begin{bmatrix} \sum_{k=1}^N W_k W_k^\top & 0 \\ 0 &  \sum_{k=1}^N V_k^\top V_k \end{bmatrix} }} \nonumber \\
	& \stackrel{\mathrm{(ii)}}{\leq} \frac{1}{p+r} \Ep \brk{\Tr \exp \brc{\Omega(\lambda(1 + \epsilon_N))}} \cdot \exp \brc{8 \lambda^2 (1 + \epsilon_N) \cdot \max \brc{\norm{\sum_{k=1}^N W_k W_k^\top}, \norm{\sum_{k=1}^N V_k^\top V_k} }} \nonumber \\
	& \leq \exp \brc{8 \lambda^2 (1 + \epsilon_N) \cdot \max \brc{\norm{\sum_{k=1}^N W_k W_k^\top}, \norm{\sum_{k=1}^N V_k^\top V_k} }},
\end{align*}
where in (i) we use Golden-Thompson and (ii) we use the monotonicity of $\Tr \exp$.

\subsection{Proof of Theorem~\ref{thm:matrix-valued-bernstein}} \label{proof:matrix-valued-bernstein}
We start by defining the filtration for the martingale argument for this proof. Let $$\mc{F}_0 = \sigma \prn{\wb{X}, \wb{\wt{X}}, \Sigma, \wt{\Sigma}, \max_{k \in [N]} \norm{X_k - \wb{X}}^2, \max_{k \in [N]} \norm{\wt{X}_k - \wb{\wt{X}}}^2}.$$ We can then define inductively $\mc{F}_k = \sigma(\mc{F}_{k-1}, X_k)$ for $k \in [N]$. We will make use of the following CGF bound from~\cite[Lemma 6.6.2]{tropp2015introduction}.
\begin{lemma}[Bernstein CGF bound] \label{lem:bernstein-cgf-bound}
	For a random symmetric matrix $A \in \mathbb{S}^n$ such that $\E A = 0$ and $\norm{A} \leq L$. We have for $|\lambda| < 3/L$ that
	\begin{align*}
		\log \E \exp (\lambda A) \preceq \frac{\lambda^2}{2 \prn{1 - \frac{|\lambda|L}{3}}} \cdot \mathbb{E} A^2.
	\end{align*}
\end{lemma}
We present a slightly stronger version of Theorem~\ref{thm:matrix-valued-bernstein} in the following, which will be useful in deriving combinatorial matrix sum inequalities (cf.~Cor.~\ref{cor:matrix-valued-combinatorial}).
\begin{manualtheorem}{\ref*{thm:matrix-valued-bernstein}\textquotesingle} \label{thm:matrix-valued-bernstein-stronger}
	Theorem~\ref{thm:matrix-valued-bernstein} holds with $w = \operatorname*{ess\,sup}_{X} (\max_{k \in [N]}\norm{W_1 X_k}, \cdots, \max_{k \in [N]}\norm{W_N X_k})^\top$ and $v = \operatorname*{ess\,sup}_{\wt{X}}  (\max_{k \in [N]} \norm{\wt{X}_k V_1}, \cdots, \max_{k \in [N]} \norm{\wt{X}_kV_N})^\top$ .
\end{manualtheorem}
It is evident Theorem~\ref{thm:matrix-valued-bernstein} directly follows from Theorem~\ref{thm:matrix-valued-bernstein-stronger}, and we devote the remainder of the proof to Theorem~\ref{thm:matrix-valued-bernstein-stronger}.

\paragraph{Step 1: MGF bounds for $\{X_k - X_{\geq k}\}_{k=1}^N$ by the martingale argument} For any $U \in \R^{N \times p \times q}$ such that Assumption~\ref{assmp:commutativity} holds for the associated commutativity pair $(\wt{X}, Z)$, let $u = (\max_{k \in [N]} \norm{U_1X_k}, \cdots, \max_{k \in [N]} \norm{U_N X_k})^\top$ and $z = (\max_{k \in [N]} \norm{\wt{X}_kZ_1}, \cdots, \max_{k \in [N]} \norm{\wt{X}_kZ_N})^\top$ in $\R^N$. Define the following positive semidefinite matrices in $\mc{F}_{k-1}$
\begin{align*}
	\Sigma_{\geq k} = \frac{1}{N+1-k} \sum_{i=k}^N (X_i - \wb{X}) (X_i - \wb{X})^\top, \qquad \wt{\Sigma}_{\geq k} = \frac{1}{N + 1-k} \sum_{i=k}^N (\wt{X}_i - \wb{\wt{X}})^\top (\wt{X}_i - \wb{\wt{X}}).	
\end{align*}
It is clear that $\Sigma_{\geq 1} = \Sigma$ and $\wt{\Sigma}_{\geq 1} = \wt{\Sigma}$. Consider $H_k = \sum_{i=1}^k D_i$ where
\begin{align*}
	D_k = \lambda \begin{bmatrix} 0 & U_k(X_k - X_{\geq k}) \\ (X_k - X_{\geq k})^\top U_k^\top & 0\end{bmatrix} - \underbrace{\frac{\lambda^2}{2\prn{1 - \frac{2|\lambda|}{3}  \min \{\norm{u}_\infty, \norm{z}_\infty\} }}}_{=:c_\lambda} \cdot \underbrace{\begin{bmatrix} U_k \Sigma_{\geq k} U_k^\top & 0 \\ 0 &Z_k^\top \wt{\Sigma}_{\geq k} Z_k \end{bmatrix}}_{=:\Delta_k} .
\end{align*}
With the notation  $\E_k [\cdot ] := \E \brk{\cdot \mid \mc{F}_k}$, we observe that $\E_{k-1} [U_k (X_k - X_{\geq k})] = 0$ and $\norm{U_k (X_k - X_{\geq k})} = \norm{(\wt{X}_k - \wt{X}_{\geq k} )Z_k} \leq 2 \min \{\norm{u}_\infty, \norm{z}_\infty\}$. We can then invoke Lemma~\ref{lem:bernstein-cgf-bound} for $|\lambda| < \frac{3}{2 \min \{\norm{u}_\infty, \norm{z}_\infty\}}$
\begin{align*}
	 & \log \E_{k-1} \exp \brc{ \lambda \begin{bmatrix} 0 & U_k(X_k - X_{\geq k}) \\ (X_k - X_{\geq k})^\top U_k^\top & 0\end{bmatrix} } \nonumber \\
	 & \preceq c_\lambda \begin{bmatrix} U_k \cdot \E_{k-1} \brk{(X_k - X_{\geq k}) (X_k - X_{\geq k})^\top } \cdot U_k^\top & 0 \\ 0 &Z_k^\top \cdot \E_{k-1} \brk{(\wt{X}_k - \wt{X}_{\geq k})^\top (\wt{X}_k - \wt{X}_{\geq k}) } \cdot Z_k \end{bmatrix} \nonumber \\
	 & \preceq  c_\lambda \begin{bmatrix} U_k \cdot \E_{k-1} \brk{(X_k - \wb{X}) (X_k - \wb{X})^\top } \cdot U_k^\top & 0 \\ 0 &Z_k^\top \cdot \E_{k-1} \brk{(\wt{X}_k - \wb{\wt{X}})^\top (\wt{X}_k - \wb{\wt{X}}) } \cdot Z_k \end{bmatrix} \nonumber \\
	 & = c_\lambda \Delta_k.
\end{align*} 
Therefore by Lieb's inequality, for all $k \in [N]$ it follows that
\begin{align*}
	 \E \brk{\Tr \exp (H_k)} & \leq \E \E_{k-1} \brk{\Tr \exp \brc{H_{k-1} - c_\lambda \Delta_k + \log \E_{k-1} \exp \brc{ \lambda \begin{bmatrix} 0 & U_k(X_k - X_{\geq k}) \\ (X_k - X_{\geq k})^\top U_k^\top & 0\end{bmatrix} } }} \nonumber \\
	 & \leq \E \E_{k-1} \brk{\Tr \exp (H_{k-1})}  = \E \brk{\Tr \exp (H_{k-1})}.
\end{align*}
Since $H_0 = 0$, one then has
\begin{align*}
	& \frac{1}{p+r} \Ep \brk{\Tr \exp \brc{\lambda \begin{bmatrix} 0 & \sum_{k=1}^N U_k(X_k - X_{\geq k}) \\ \sum_{k=1}^N (X_k - X_{\geq k})^\top U_k^\top & 0\end{bmatrix} - c_\lambda \sum_{k=1}^N \Delta_k}} \nonumber \\
	& = \frac{1}{p+r} \E \brk{\Tr \exp (H_N)} \leq \frac{1}{p+r} \E \brk{\Tr \exp(H_0)} = 1
\end{align*}

\paragraph{Step 2: MGF bounds for $\{X_k - \wb{X}\}_{k=1}^N$ by Jensen's inequality} For all $\Pi \in \permset_N$, let $U = W \times_1 \Pi $ and $X = X \times_1 \Pi$. We still have Eq.~\eqref{eq:mid-matrix-Hoeffding-U-X-general} and it holds that $\norm{u}_\infty = \norm{w}_\infty, \norm{z}_\infty = \norm{v}_\infty$. As is proven in Appendix~\ref{proof:matrix-valued-hoeffding}, by averaging over all permutations, 
\begin{align*}
	\frac{1}{N!} \sum_{\Pi \in \permset_N} \llangle W \times_1 \Pi, X \times_1 (I- \mathsf{B}_N) \Pi   \rrangle = \frac{1}{N!} \sum_{\pi \in \permmapset_N} \sum_{k=1}^N W_{\pi(k)}(X_{\pi(k)} - X_{\geq k}^\pi) = \frac{1}{1 + \epsilon_N} \llangle W \times_1 \proj_{\ones_N}^\perp, X \rrangle,
\end{align*}
where $X^\pi_{\geq k} := \frac{1}{N+1-k} \sum_{i=k}^N X_{\pi(i)}$. It boils down to controlling
\begin{align*}
	\frac{1}{N!} \sum_{\pi \in \permmapset_N} \sum_{k=1}^N W_{\pi(k)} \Sigma_{\geq k}^\pi W_{\pi(k)}^\top &:= 	\frac{1}{N!} \sum_{\pi \in \permmapset_N} \sum_{k=1}^N W_{\pi(k)} \cdot \brc{ \frac{1}{N+1-k} \sum_{i=k}^N (X_{\pi(i)} - \wb{X}) (X_{\pi(i)} - \wb{X})^\top} \cdot W_{\pi(k)}^\top, \\
	\frac{1}{N!} \sum_{\pi \in \permmapset_N} \sum_{k=1}^N V_{\pi(k)}^\top \wt{\Sigma}_{\geq k}^\pi V_{\pi(k)} &:= 	\frac{1}{N!} \sum_{\pi \in \permmapset_N} \sum_{k=1}^N V_{\pi(k)}^\top  \cdot \brc{\frac{1}{N+1-k} \sum_{i=k}^N (\wt{X}_{\pi(i)} - \wb{\wt{X}})^\top (\wt{X}_{\pi(i)} - \wb{\wt{X}})} \cdot V_{\pi(k)}^\top.
\end{align*}
The next lemma follows.

\begin{lemma} \label{lem:perm-avg-matrix-valued-bernstein}
	For all $N \geq 2$,
	\begin{align*}
		\frac{1}{N!} \sum_{\pi \in \permmapset_N} \sum_{k=1}^N W_{\pi(k)} \Sigma_{\geq k}^\pi W_{\pi(k)}^\top & \preceq \prn{\frac{1}{1 + \epsilon_N} \norm{\Sigma} + \frac{\epsilon_N}{1 + \epsilon_N} \max_{l \in [N]} \norm{X_l - \wb{X}}^2 } \cdot \sum_{k=1}^N W_k W_k^\top =\frac{\wt{\sigma}_X^2}{1 + \epsilon_N} \sum_{k=1}^N W_k W_k^\top, \\
		\frac{1}{N!} \sum_{\pi \in \permmapset_N} \sum_{k=1}^N V_{\pi(k)}^\top \wt{\Sigma}_{\geq k}^\pi V_{\pi(k)} & \preceq \prn{\frac{1}{1 + \epsilon_N} \norm{\wt{\Sigma}} + \frac{\epsilon_N}{1 + \epsilon_N} \max_{l \in [N]} \norm{\wt{X}_l - \wb{\wt{X}}}^2 } \cdot \sum_{k=1}^N V_k^\top V_k  =\frac{\wt{\sigma}_{\wt{X}}^2}{1 + \epsilon_N} \sum_{k=1}^N V_k^\top V_k.
	\end{align*}
\end{lemma}
By Jensen's inequality and the lemma above, we have, using the same $*$ notation,
\begin{align*}
	& \frac{1}{p+r} \Ep \brk{\Tr \exp \brc{\frac{\lambda}{1 + \epsilon_N} \begin{bmatrix} 0 & \llangle W \times_1 \proj_{\ones_N}^\perp, X \rrangle \\ * & 0\end{bmatrix} - \frac{c_\lambda}{1 + \epsilon_N} \sum_{k=1}^N \begin{bmatrix} \wt{\sigma}^2_X W_k W_k^\top & 0 \\ 0 & \wt{\sigma}^2_{\wt{X}} V_k^\top  V_k \end{bmatrix}}} \leq  1.
\end{align*}

\paragraph{Step 3: Concluding the proof} The proof is done by applying Golden-Thompson, substituting $\lambda$ with $\lambda \cdot (1 + \epsilon_N)$ and the following inequality
\begin{align*}
	\begin{bmatrix} \wt{\sigma}_X^2 \sum_{k=1}^N W_k W_k^\top & 0 \\ 0 &  \wt{\sigma}^2_{\wt{X}} \sum_{k=1}^N V_k^\top  V_k \end{bmatrix} \preceq \max \brc{\wt{\sigma}_X^2 \norm{\sum_{k=1}^N W_k W_k^\top}, \wt{\sigma}^2_{\wt{X}} \norm{\sum_{k=1}^N V_k^\top V_k} } I_{p +r}.
\end{align*}

\subsection{Proof of Lemma~\ref{lem:perm-avg-matrix-valued-bernstein}} \label{proof:perm-avg-matrix-valued-bernstein}
We detail the proof for
\begin{align*}
	\frac{1}{N!} \sum_{\pi \in \permmapset_N} \sum_{k=1}^N W_{\pi(k)} \Sigma_{\geq k}^\pi W_{\pi(k)}^\top \preceq \prn{\frac{1}{1 + \epsilon_N} \norm{\Sigma} + \frac{\epsilon_N}{1 + \epsilon_N} \max_{l \in [N]} \norm{X_l - \wb{X}}^2 } \cdot \sum_{k=1}^N W_k W_k^\top
\end{align*}
and the second inequality follows from the same argument. Let $\uniform(\permmapset_N)$ be the uniform distribution on $\permmapset_N$, we can write
\begin{align*}
	& \frac{1}{N!} \sum_{\pi \in \permmapset_N} \sum_{k=1}^N W_{\pi(k)} \Sigma_{\geq k}^\pi W_{\pi(k)}^\top \\*
	& = \E_{\pi \sim \uniform(\permmapset_N)} \brk{\sum_{k=1}^N W_{\pi(k)} \cdot \brc{ \frac{1}{N+1-k} \sum_{i=k}^N (X_{\pi(i)} - \wb{X}) (X_{\pi(i)} - \wb{X})^\top} \cdot W_{\pi(k)}^\top} \nonumber \\
	& = \sum_{k=1}^N  \frac{1}{N+1-k}  \E_{\pi \sim \uniform(\permmapset_N)} \brk{ W_{\pi(k)} (X_{\pi(k)} - \wb{X}) (X_{\pi(k)} - \wb{X})^\top W_{\pi(k)}^\top} \nonumber \\
	& \qquad + \sum_{k=1}^N \sum_{i=k+1}^N  \frac{1}{N+1-k}  \E_{\pi \sim \uniform(\permmapset_N)} \brk{ W_{\pi(k)} (X_{\pi(i)} - \wb{X}) (X_{\pi(i)} - \wb{X})^\top W_{\pi(k)}^\top }.
\end{align*}
Since for any fixed $k \neq i$, $\pi(k)$ is uniform on $[N]$ and $(\pi(k), \pi(i))$ is uniform on $\{(j, l) : j \neq l, j, l \in [N]\}$, we can proceed the calculation
\begin{align*}
	& \frac{1}{N!} \sum_{\pi \in \permmapset_N} \sum_{k=1}^N W_{\pi(k)} \Sigma_{\geq k}^\pi W_{\pi(k)}^\top \\
	& = \sum_{k=1}^N \frac{1}{N(N+1-k)} \cdot \brc{ \sum_{j=1}^N W_j (X_j - \wb{X}) (X_j - \wb{X})^\top W_j^\top + \frac{N-k}{N-1} \sum_{j, l \in [N], j \neq l} W_j (X_l - \wb{X}) (X_l - \wb{X})^\top W_j^\top  } \nonumber \\
	&  = \sum_{k=1}^N \frac{1}{N(N+1-k)} \cdot \brc{\frac{k-1}{N-1}\sum_{j=1}^N W_j (X_j - \wb{X}) (X_j - \wb{X})^\top W_j^\top + \frac{N-k}{N-1}\sum_{j=1}^N \sum_{l=1}^N W_j (X_l - \wb{X}) (X_l - \wb{X})^\top W_j^\top  }.
\end{align*}
Using the following inequalities
\begin{align*}
	\sum_{j=1}^N W_j (X_j - \wb{X}) (X_j - \wb{X})^\top W_j^\top & \preceq \max_{l \in [N]} \norm{X_l - \wb{X}}^2 \cdot \sum_{j=1}^N W_j W_j^\top, \\
	\sum_{j=1}^N \sum_{l=1}^N W_j (X_l - \wb{X}) (X_l - \wb{X})^\top W_j^\top & = N \sum_{j=1}^N W_j \Sigma W_j^\top \preceq N \norm{\Sigma} \cdot \sum_{j=1}^N W_j W_j^\top,
\end{align*}
we can complete the proof by
\begin{align*}
	& \frac{1}{N!} \sum_{\pi \in \permmapset_N} \sum_{k=1}^N W_{\pi(k)} \Sigma_{\geq k}^\pi W_{\pi(k)}^\top \nonumber \\
	& \preceq \brc{\frac{1}{N-1} \sum_{k=1}^N \frac{N-k}{N+1-k} \norm{\Sigma} + \frac{1}{N(N-1)}\sum_{k=1}^N \frac{k-1}{N+1-k}  \max_{l \in [N]} \norm{X_l - \wb{X}}^2 } \cdot \sum_{j=1}^N W_j W_j^\top \nonumber \\
	& = \brc{\frac{N - H_N}{N-1} \norm{\Sigma}+ \frac{N(H_N - 1)}{N(N-1)}  \max_{l \in [N]} \norm{X_l - \wb{X}}^2 } \cdot \sum_{j=1}^N W_j W_j^\top \nonumber \\
	&= \brc{\frac{1}{1 + \epsilon_N} \norm{\Sigma} + \frac{\epsilon_N}{1 + \epsilon_N} \max_{l \in [N]} \norm{X_l - \wb{X}}^2 } \cdot \sum_{j=1}^N W_j W_j^\top.
\end{align*}

\section{Proofs for theoretical implications and applications} \label{sec:proof-theoretical-implications}

\subsection{Proofs for connections to i.i.d.\ setting} \label{proof:tensor-iid}
It is first instructive to show $\sum_{l=1}^{N_K} \<\mathscr{W}_l, \E_{\mathscr{X} \sim P}[\mathscr{X}] \>_{\tspace_{K-1}} = \<W, \mu_P \ones_{\tspace}\>_{\tspace}$. In fact, using that $W$ is balanced along mode-$K$, we verify
\begin{align*}
    \sum_{l=1}^{N_K} \<\mathscr{W}_l, \E_{\mathscr{X} \sim P}[\mathscr{X}] \>_{\tspace_{K-1}} & = \left\< \sum_{l=1}^{N_K}  \mathscr{W}_l, \E_{\mathscr{X} \sim P}[\mathscr{X}] \right\>_{\tspace_{K-1}} = N_K \left\< \wb{W}_{K-1}, \E_{\mathscr{X} \sim P}[\mathscr{X}] \right\>_{\tspace_{K-1}} \nonumber \\
    &= \<\wb{W} \ones_{\tspace},  \E_{\mathscr{X} \sim P}[\mathscr{X}]\>_{\tspace} = \<W \times_1 \proj_{\ones_{N_1}} \times_2 \cdots \times_K \proj_{\ones_{N_K}},  \E_{\mathscr{X} \sim P}[\mathscr{X}]\>_{\tspace} \nonumber \\
    & = \<W ,  \E_{\mathscr{X} \sim P}[\mathscr{X}]\times_1 \proj_{\ones_{N_1}} \times_2 \cdots \times_K \proj_{\ones_{N_K}}\>_{\tspace}= \<W, \mu_P \ones_{\tspace}\>_{\tspace}.
\end{align*}
We return to the main proof. For any $\wt{N}_K \geq N_K$, let $\wt{W} \in \R^{N_1 \times N_2 \cdots \times N_{K-1} \times \wt{N}_K} =: \wt{\tspace}$ be
\begin{align*}
	\wt{W}_{i_1 \cdots i_{K-1} i_K} = \begin{cases}
	    W_{i_1 \cdots i_{K-1} i_K}, & i_K \leq N_K, \\
		0, & i_K > N_K,
	\end{cases}
\end{align*}
and let the random data $\wt{X} \in \wt{\tspace}$ be exchangeable along mode-$K$ such that $\wt{X}_{i_1 \cdots i_{K-1} i_K} = X_{i_1 \cdots i_{K-1} i_K}$ when $i_K \in [N_K]$. We can verify that $\wt{W}$ is balanced along mode-$K$ and
\begin{align*}
	\lambda  \prn{\<\wt{W}, \wt{X} \>_{\wt{\tspace}} - \prod_{l=1}^{K-1} N_l \wt{N}_K \cdot \wb{\wt{W}} \wb{\wt{X}}} & = \lambda \prn{\langle W , X \rangle_{\tspace} - \prod_{l=1}^K N_l \cdot \wb{W} \wb{\wt{X}}} \cas \lambda \prn{\langle W, X \rangle_{\tspace} - \prod_{l=1}^K N_l \cdot \wb{W} \mu_P }.
\end{align*}
Taking $\wt{N}_K \to \infty$, as $\wt{W}_{k} \times_{k} \proj_{\ones_{N_{k}}}^\perp = 0$ and $\norm{\wt{W}_{K} \times_{K} \proj_{\ones_{\wt{N}_K}}^\perp }_{\wt{\tspace}}^2 \to \norm{W}_{\tspace}^2$, we substitute $\epsilon_{\wt{N}_K} \to 0$ into Theorem~\ref{thm:tensor-hoeffding} to recover the Hoeffding bound for independent $X$ along mode-$K$. For the Bernstein bound, additionally making use of the fact that $\linf{\wb{w}_K} \to \linf{w}$ and
\begin{align*}
	\frac{1}{\wt{N}_K} \norm{\wt{X} \times_K \proj_{\ones_{\wt{N}_K}}^\perp}_{\tspace}^2 \to \sigma_P^2 
\end{align*}
completes the proof. The proof for the matrix-valued result is identical. We omit the repetitive details.

\subsection{Proof of Corollary~\ref{cor:matrix-valued-combinatorial}} \label{proof:matrix-valued-combinatorial}
We are able to apply Theorem~\ref{thm:matrix-valued-bernstein} since Assumption~\ref{assmp:commutativity} holds. Using properties of Kronecker product of matrices (see~\cite[Chp.~4]{horn1994topics}), we verify that
\begin{align*}
	\norm{X_k} = \norm{e_{\pi(k)} \otimes I_n} = \norm{e_{\pi(k)}} \cdot \norm{I_n}= 1,
\end{align*}
and $\norm{\wt{X}_k} = 1$. We now compute the variance parameters from
\begin{align*}
	\Sigma & = \frac{1}{N} \sum_{k=1}^N (X_k  - \wb{X}) (X_k - \wb{X})^\top = \frac{1}{N} \sum_{k=1}^N \prn{e_k \otimes I_n} \prn{e_k \otimes I_n}^\top - \prn{\frac{1}{N} \sum_{k=1}^N e_k \otimes I_n} \prn{\frac{1}{N} \sum_{k=1}^N e_k \otimes I_n}^\top \\
	& = \frac{1}{N} \sum_{k=1}^N e_k e_k^\top \otimes I_n  - \prn{\frac{1}{N} \ones_N \otimes I_n}  \prn{\frac{1}{N} \ones_N \otimes I_n}^\top = \frac{1}{N} \cdot  \proj_{\ones_N}^\perp \otimes I_n, \\
	\wt{\Sigma} & = \frac{1}{N} \sum_{k=1}^N (\wt{X}_k - \wb{\wt{X}})^\top (\wt{X}_k - \wb{\wt{X}}) = \frac{1}{N} \sum_{k=1}^N \prn{e_k \otimes I_m} \prn{e_k \otimes I_m}^\top - \prn{\frac{1}{N} \sum_{k=1}^N e_k \otimes I_m} \prn{\frac{1}{N} \sum_{k=1}^N e_k \otimes I_m}^\top \\
	& =  \frac{1}{N} \sum_{k=1}^N e_k e_k^\top \otimes I_m - \prn{\frac{1}{N} \ones_N \otimes I_m}  \prn{\frac{1}{N} \ones_N \otimes I_m}^\top  = \frac{1}{N} \cdot \proj_{\ones_N}^\perp \otimes I_m,
\end{align*}
and $\max_{k \in [N]} \norm{X_k - \wb{X}}^2 = \max_{k \in [N]} \norm{\wt{X}_k - \wb{\wt{X}}}^2 = (1-1/N)^2 + (N-1) \cdot 1/N^2 = 1-1/N$, which implies
\begin{align*}
	\wt{\sigma}_X^2 = \wt{\sigma}_{\wt{X}}^2 = \frac{1}{N} + \epsilon_N \cdot \prn{1 - \frac{1}{N}}.
\end{align*} 
On the other hand, as
\begin{align*}
	\sum_{k=1}^N W_k W_k^\top = \sum_{i, j \in [N]} A_{i,j} A_{i,j}^\top, \qquad \text{and} \qquad \sum_{k=1}^N V_k^\top V_k = \sum_{i, j \in [N]} A_{i,j}^\top A_{i,j}.
\end{align*}
We will invoke the slightly modified Theorem~\ref{thm:matrix-valued-bernstein-stronger} with $\wb{W}\;\wb{X} = 0$ and the preceding calculated parameters. It remains to bounding $\linf{w}$ and $\linf{v}$ defined in Theorem~\ref{thm:matrix-valued-bernstein-stronger}, which are readily available below since
\begin{align*}
	\max_{k \in [N]}\norm{W_i X_k} = \max_{k \in [N]} \norm{\wt{X}_k V_i}  = \max_{j \in[N]} \norm{A_{i,j}} \leq R.
\end{align*}
Lastly, we apply Lemma~\ref{lem:MGF-tail-matrix} for the tail bound that follows. 

Finally, as a byproduct, we observe that our results immediately yield a Bernstein-type tail bound without any commutativity assumptions, obtained as a direct application of the combinatorial Bernstein inequality. Since analogous consequences can also be derived directly from the results in~\cite{mackey2014matrix}, we refrain from presenting them in the main text. We emphasize that the principal Bernstein-type MGF bounds established in the main paper under commutativity conditions have sharper constants and captures empirical variations within $X_k$ through $\wt{\sigma}_X^2$ and $\wt{\sigma}_{\wt{X}}^2$, which are both missing from the following argument. Indeed, observing that by introducing a uniformly random permutation $\pi$ of $[N]$, independent of $X_1,\cdots,X_N$, and using exchangeability to rewrite
\begin{align*}
\sum_{k=1}^N W_kX_k
\stackrel{d}{=}
\sum_{k=1}^N W_kX_{\pi(k)}.
\end{align*}
Conditioning on $X_1,\cdots,X_N$, the right-hand side becomes a combinatorial matrix sum to which a matrix Bernstein inequality can be applied. Assuming $\max_{i \in [n]}\norm{X_k} \leq 1$ and $\max_{k \in [N]} \norm{W_k} \leq L$ gives the following result.
\begin{proposition} Under the assumptions above and assuming $\wb{X}=0$, let 
\begin{align*}
    \sigma^2_W = \max \brc{\norm{\sum_{k=1}^N W_k W_k^\top}, \norm{\sum_{k=1}^N W_k^\top W_k}}, \qquad L =  \max_{k \in [N]} \norm{W_k},
\end{align*}
we have for all $\delta \in (0,1)$,
\begin{align*}
	\P \brc{\norm{\sum_{k=1}^N W_k X_k} \geq (1+\epsilon_N) \sigma_W \sqrt{2 \prn{(N-1)\epsilon_N + 1} \cdot \log \frac{p+r}{\delta}} + \frac{2 (1+\epsilon_N) L}{3} \log \frac{p+r}{\delta}}  \leq \delta.
\end{align*}
\end{proposition}

\subsection{Proof of Corollary~\ref{cor:avg-effect}} \label{proof:avg-effect}
To apply Theorem~\ref{thm:tensor-bernstein}, we will use the bound in Corollary~\ref{cor:tensor-bernstein} and compute the parameters $\wt{\sigma}^2_{X,k}$. From
\begin{align*}
	\prn{\wb{X}_k\times_k \proj_{\ones_{N_k}}^\perp}_{i_1 \cdots i_k} = \ones_{i_l \in I_l, l \in [k-1]} \cdot \brc{\ones_{i_k \in I_k} \frac{n_{k+1} \cdots n_K}{N_{k+1} \cdots N_K} \prn{1 - \frac{n_k}{N_k}} - \ones_{i_k \not \in I_k} \frac{n_k \cdots n_K}{N_k \cdots N_K} },
\end{align*}
we have
\begin{align*}
	\norm{\wb{X}_k \times_k \proj_{\ones_{N_k}}^\perp}_{\tspace_k}^2& = \prod_{l=1}^{k-1} n_l \cdot \brc{\prn{\frac{n_{k+1} \cdots n_K}{N_{k+1} \cdots N_K}}^2 \prn{1 - \frac{n_k}{N_k}}^2 \cdot n_k + \prn{\frac{n_k \cdots n_K}{N_k \cdots N_K} }^2 \cdot \prn{N_k -n_k} } \\
	& = \prod_{l=1}^{k} n_l \cdot \prn{\frac{n_{k+1} \cdots n_K}{N_{k+1} \cdots N_K}}^2 \prn{1 - \frac{n_k}{N_k}}   \\
	\norm{\wb{X}_{k} \times_{k} \proj_{\ones_{N_{k}}}^\perp}_{\infty}^2 & = \ind_{n_k < N_k} \cdot \prn{\max \brc{1 - \frac{n_k}{N_k}, \frac{n_k}{N_k}} \cdot \frac{n_{k+1} \cdots n_K}{N_{k+1} \cdots N_K}}^2.
\end{align*}
Therefore
\begin{align*}
	\wt{\sigma}^2_{X,k} & = \frac{1}{N_k} \cdot \norm{\wb{X}_k \times_k \proj_{\ones_{N_k}}^\perp}_{\tspace_k}^2 + \epsilon_{N_{k}} \cdot \prod_{l=0}^{k-1} N_l \cdot \norm{\wb{X}_{k} \times_{k} \proj_{\ones_{N_{k}}}^\perp}_{\infty}^2 \nonumber \\
	& = \prod_{l=1}^{k-1} n_l \cdot \prn{\frac{n_{k+1} \cdots n_K}{N_{k+1} \cdots N_K}}^2 \prn{1 - \frac{n_k}{N_k}} \cdot \frac{n_k}{N_k}  + \epsilon_{N_k} \cdot \prod_{l=1}^{k-1} N_l \cdot \ind_{n_k < N_k} \cdot \prn{\max \brc{1 - \frac{n_k}{N_k}, \frac{n_k}{N_k}} \cdot \frac{n_{k+1} \cdots n_K}{N_{k+1} \cdots N_K}}^2 \\
	& = \brc{\frac{\prod_{l=1}^k n_l}{\prod_{l=1}^k N_l} \cdot  \prn{1 - \frac{n_k}{N_k}} + \epsilon_{N_k} \cdot \ind_{n_k < N_k} \max \brc{1 - \frac{n_k}{N_k}, \frac{n_k}{N_k}}^2} \cdot \prod_{l=1}^{k-1} N_l \cdot \prn{\frac{n_{k+1} \cdots n_K}{N_{k+1} \cdots N_K}}^2 \\
	& \leq \brc{\frac{\prod_{l=1}^k n_l}{\prod_{l=1}^k N_l} \cdot \prn{1 - \frac{n_k}{N_k}} + \epsilon_{N_k} \cdot \ind_{n_k < N_k}} \cdot \prod_{l=1}^{k-1} N_l \cdot \prn{\frac{n_{k+1} \cdots n_K}{N_{k+1} \cdots N_K}}^2,
\end{align*}
using $\prn{\prod_{l=k+1}^K N_l} \linf{\wb{w}_k} \leq 2B_k$, for $\what{\mu} - \mu = \frac{\prod_{l=1}^K N_l}{\prod_{l=1}^K n_l} \cdot \prn{ \<W, X\>_{\tspace} - \prod_{l=1}^K N_l \cdot \wb{W}\;\wb{X} }$, we can invoke Theorem~\ref{thm:tensor-bernstein} (in particular, Eq~\eqref{eq:tensor-bernstein-c} and Corollary~\ref{cor:tensor-bernstein}) with weights $W-W_0$ and data $2X-\ones_{\tspace}$, and therefore obtain
\begin{align*}
	& \P \Bigg\{\what{\mu} - \mu \geq  \sigma \sqrt{2 \log \frac{1}{\delta} \cdot \max_{k \in [K]} \Delta_k }   +  \frac{2 \max_{k \in [K]} B_k(1 + \epsilon_{N_k})}{3 \alpha_K} \log \frac{1}{\delta} \Bigg\}  \leq \delta, 
\end{align*}
where $\sigma^2 = \prod_{l=1}^K N_l \cdot \prn{\norm{W }_{\tspace}^2 -  \norm{W_{0} }_{\tspace}^2}$, and
\begin{align*}
	\Delta_k & = \frac{1}{\prod_{l=1}^K N_l}  \cdot \frac{\prod_{l=1}^K N_l^2}{\prod_{l=1}^K n_l^2} \cdot \prod_{l=k+1}^K N_l \cdot \prn{1+ \epsilon_{N_{k}}}\wt{\sigma}^2_{X, k} \nonumber \\
	& \leq \frac{\prod_{l=1}^K N_l}{\prod_{l=1}^K n_l^2} \cdot \prod_{l=k+1}^K N_l \cdot \prn{1+ \epsilon_{N_{k}}} \cdot \brc{\frac{\prod_{l=1}^k n_l}{\prod_{l=1}^k N_l} \cdot  \prn{1 - \frac{n_k}{N_k}} + \epsilon_{N_k} \cdot \ind_{n_k < N_k}} \cdot \prod_{l=1}^{k-1} N_l \cdot \prn{\frac{n_{k+1} \cdots n_K}{N_{k+1} \cdots N_K}}^2 \nonumber \\
	& = \frac{\prod_{l=1}^k N_l}{\prod_{l=1}^k n_l^2} \cdot \prn{1+ \epsilon_{N_{k}}} \cdot \brc{\frac{\prod_{l=1}^k n_l}{\prod_{l=1}^k N_l} \cdot \prn{1 - \frac{n_k}{N_k}} + \epsilon_{N_k} \cdot \ind_{n_k < N_k}} \cdot \frac{\prod_{l=1}^{k} N_l}{N_k} \nonumber \\
	& = \frac{1+ \epsilon_{N_{k}}}{N_k} \cdot \brc{\frac{\prod_{l=1}^k N_l}{\prod_{l=1}^k n_l} \cdot  \prn{1 - \frac{n_k}{N_k}} + \epsilon_{N_k} \cdot \ind_{n_k < N_k} \prn{\frac{\prod_{l=1}^k N_l}{\prod_{l=1}^k n_l}}^2} \nonumber \\
	& = \frac{1+ \epsilon_{N_{k}}}{N_k \alpha_k } \cdot \prn{1 - \frac{n_k}{N_k} + \frac{\epsilon_{N_k}}{\alpha_k} \cdot \ind_{n_k < N_k}}.
\end{align*}
This completes the first part of the proof. For the second half, since $N_K \alpha \geq 1$, we can substitute $(n_1, \cdots, n_k) = (N_1, \cdots, N_{K-1}, \alpha N_K)$ into the preceding bound. The proof is done since under this choice
\begin{align*}
	\frac{1+ \epsilon_{N_{k}}}{N_k \alpha_k } \cdot \prn{1 - \frac{n_k}{N_k} + \frac{\epsilon_{N_k}}{\alpha_k} \cdot \ind_{n_k < N_k}} = \ind_{k=K} \cdot \frac{1 + \epsilon_{N_K}}{N_K \alpha} \cdot \prn{1 - \alpha + \frac{\epsilon_{N_K}}{\alpha}}.
\end{align*}

\subsection{Proof of Corollary~\ref{cor:sketching}} \label{proof:sketching}
Let $U= \begin{bmatrix} u_1 & \cdots & u_q \end{bmatrix}^\top = \begin{bmatrix} \wt{u}_1 & \cdots & \wt{u}_q \end{bmatrix}$ and $\theta_k = \begin{bmatrix} \theta_{k, 1} & \cdots \theta_{k, r} \end{bmatrix}$. Recall $U_k^\top U_k = U^\top D_k U$ where $D_k = \diag(d_{k,1}, \cdots, d_{k, q}) \succeq 0$ is a diagonal matrix with at most $q'$ nonzero elements. 

\paragraph{Step 1: Checking commutativity constraints by defining $(\wt{X}, V)$} The commutativity condition~\ref{assmp:commutativity} holds for $W_k = U_k^\top U_k$ and $X_k = \theta_k$. In fact
\begin{align*}
	W_k X_k & =  U^\top D_k U \theta_k = \sum_{l=1}^q d_l u_l u_l^\top \theta_k  = \underbrace{\begin{bmatrix}
			u_1^\top \theta_k \,\, \cdots \,\, u_q^\top \theta_k  & & & \\
			& u_1^\top \theta_k \,\, \cdots \,\, u_q^\top \theta_k & & \\
			& & \ddots & \\
			& & & u_1^\top \theta_k \,\, \cdots \,\, u_q^\top \theta_k
	\end{bmatrix}}_{q \times q^2 r} \underbrace{\begin{bmatrix}
			d_{k,1} u_{l,1} I_r \\
			\vdots \\
			d_{k,q} u_{q, 1} I_r \\
			d_{k,1} u_{l,2} I_r \\
			\vdots \\
			d_{k,q} u_{q,2} I_r \\
			\vdots \\
			d_{k,q} u_{q,q} I_r
	\end{bmatrix}}_{q^2r \times r}  \nonumber \\
	&  = \underbrace{\prn{I_q \otimes \begin{bmatrix} u_1^\top \theta_k & \cdots & u_q^\top \theta_k \end{bmatrix}}}_{:=\wt{X}_k} \underbrace{\begin{bmatrix}
		D_k \wt{u}_1 \otimes I_r \\
		\vdots \\
		D_k \wt{u}_q \otimes I_r
	\end{bmatrix}}_{=:V_k} = \wt{X}_k V_k.
\end{align*}

\paragraph{Step 2: Checking boundedness conditions and weight parameters}

We verify conditions for Theorem~\ref{thm:matrix-valued-bernstein}. For the data $X_k$ and $\wt{X}_k$, we have
\begin{align*}
	\norm{X_k}& =  \norm{\theta_k} \leq \norm{\theta_k}_F \leq 1, \\
	\norm{\wt{X}_k} & = \sqrt{\norm{\wt{X}_k \wt{X}_k^\top}} = \sqrt{\norm{\prn{I_q \otimes \begin{bmatrix} u_1^\top \theta_k & \cdots & u_q^\top \theta_k \end{bmatrix}} \prn{I_q \otimes \begin{bmatrix} u_1^\top \theta_k & \cdots & u_q^\top \theta_k \end{bmatrix}^\top} }} \nonumber \\
	& = \sqrt{\norm{I_q \otimes \begin{bmatrix} u_1^\top \theta_k & \cdots & u_q^\top \theta_k \end{bmatrix}\begin{bmatrix} u_1^\top \theta_k & \cdots & u_q^\top \theta_k \end{bmatrix}^\top}} \nonumber \\
	& = \sqrt{\norm{I_q \otimes \sum_{l=1}^q u_l^\top \theta_k \theta_k^\top u_l}} = \sqrt{\Tr \prn{U \theta_k \theta_k^\top U}} = \norm{\theta_k}_F \leq 1.
\end{align*}
For the weight matrices, we have $\min \{\linf{w}, \linf{v}\} \leq \linf{w} = \max_{k \in [N]} \norm{W_k} = \max_{k\in [N]} \norm{D_k} = L$, and
\begin{align*}
	\norm{\sum_{k=1}^N W_k W_k^\top} & = \norm{\sum_{k=1}^N D_k^2} \leq M, \\
	\norm{\sum_{k=1}^N V_k^\top V_k} & = \norm{\sum_{k=1}^N \sum_{l=1}^q \prn{D_k \wt{u}_l \otimes I_r}^\top \prn{D_k \wt{u}_l \otimes I_r} } = \norm{\sum_{k=1}^N \sum_{l=1}^q \wt{u}_l^\top D_k^2 \wt{u}_l \otimes I_r } \nonumber \\
	& = \sum_{l=1}^q \wt{u}_l^\top \prn{\sum_{k=1}^N D_k^2} \wt{u}_l = \Tr \prn{U^\top \sum_{k=1}^N D_k^2 U} = \Tr \prn{\sum_{k=1}^N D_k^2}.
\end{align*}

\paragraph{Step 3: Bounding the data variance parameters} Next, we calculate the variance parameters $\wt{\sigma}_X^2$ and $\wt{\sigma}_{\wt{X}}^2$. Straightforwardly as $X_k = \theta_k$,
\begin{align*}
	\wt{\sigma}_X^2 = \norm{\frac{1}{N} \sum_{k=1}^N (\theta_k - \wb{\theta}) (\wb{\theta} - \theta_k)^\top} + \epsilon_N \max_{k \in [N]} \norm{\theta_k - \wb{\theta}}^2.
\end{align*}
While for $\wt{X}$, since
\begin{align*}
	\wt{\Sigma} & = \frac{1}{N} \sum_{k=1}^N (\wt{X}_k - \wb{\wt{X}})^\top (\wt{X}_k - \wb{\wt{X}}) \nonumber \\
	& = \frac{1}{N} \sum_{k=1}^N \prn{I_q \otimes \begin{bmatrix} u_1^\top \theta_k & \cdots & u_q^\top \theta_k \end{bmatrix}}^\top \prn{I_q \otimes \begin{bmatrix} u_1^\top \theta_k & \cdots & u_q^\top \theta_k \end{bmatrix}} \nonumber \\
    & \qquad - \prn{I_q \otimes \begin{bmatrix} u_1^\top \wb{\theta} & \cdots & u_q^\top \wb{\theta} \end{bmatrix}}^\top \prn{I_q \otimes \begin{bmatrix} u_1^\top \wb{\theta} & \cdots & u_q^\top \wb{\theta} \end{bmatrix}} \nonumber \\
	& = \frac{1}{N} \sum_{k=1}^N I_q \otimes \begin{bmatrix} \theta_k^\top u_i u_j^\top \theta_k \end{bmatrix}_{i, j \in [q]} - I_q \otimes \begin{bmatrix} \wb{\theta}^\top u_i u_j^\top \wb{\theta} \end{bmatrix}_{i, j \in [q]},
\end{align*}
we have
\begin{align*}
	\sigma^2_{\wt{X}} = \norm{\frac{1}{N} \sum_{k=1}^N I_q \otimes \begin{bmatrix} \theta_k^\top u_i u_j^\top \theta_k \end{bmatrix}_{i, j \in [q]} - I_q \otimes \begin{bmatrix} \wb{\theta}^\top u_i u_j^\top \wb{\theta} \end{bmatrix}_{i, j \in [q]}} = \norm{\frac{1}{N} \sum_{k=1}^N \begin{bmatrix} (\theta_k - \wb{\theta})^\top u_i u_j^\top (\theta_k - \wb{\theta}) \end{bmatrix}_{i, j \in [q]}}.
\end{align*}
Consequently, the trace bound for positive semidefinite matrices gives
\begin{align*}
	\sigma^2_{\wt{X}} & = \norm{\frac{1}{N} \sum_{k=1}^N \begin{bmatrix} (\theta_k - \wb{\theta})^\top u_i u_j^\top (\theta_k - \wb{\theta}) \end{bmatrix}_{i, j \in [q]}} \nonumber \\
	& \leq \frac{1}{N} \sum_{k=1}^N \sum_{i=1}^q \norm{(\theta_k - \wb{\theta})^\top u_i}^2 = \Tr \prn{U \prn{\frac{1}{N} \sum_{k=1}^N (\theta_k - \wb{\theta})(\theta_k - \wb{\theta})^\top} U^\top} = \frac{1}{N} \sum_{k=1}^N \norm{\theta_k - \wb{\theta}}_F^2.
\end{align*}
Combining with
\begin{align*}
	 \norm{\wt{X}_k - \wb{\wt{X}}}^2 & = \norm{I_q \otimes \begin{bmatrix} u_1^\top (\theta_k - \wb{\theta}) & \cdots & u_q^\top (\theta_k - \wb{\theta}) \end{bmatrix}}^2 \nonumber \\
	 & =  \norm{\begin{bmatrix} u_1^\top (\theta_k - \wb{\theta}) & \cdots & u_q^\top (\theta_k - \wb{\theta}) \end{bmatrix}}^2 = \sum_{l=1}^q \norm{u_l^\top (\theta_k - \wb{\theta})}^2 \nonumber \\
	 & = \Tr \prn{U (\theta_k - \wb{\theta}) (\theta_k - \wb{\theta}) U^\top} = \norm{\theta_k - \wb{\theta}}_F^2,
\end{align*}
it follows that
\begin{align*}
	\max \brc{\wt{\sigma}_X^2, \wt{\sigma}_{\wt{X}}^2} \leq \frac{1}{N} \sum_{k=1}^N \norm{\theta_k - \wb{\theta}}_F^2 + \epsilon_N \max_{k \in [N]} \norm{\theta_k - \wb{\theta}}_F^2 = \wt{\sigma}_\theta^2.
\end{align*}

\paragraph{Step 4: Concluding the proof} By writing $Y_j=\theta_j-\wb{\theta}$, we use Lemma~\ref{lem:perm-avg-matrix-valued-bernstein} in Step~2 of Appendix~\ref{proof:matrix-valued-bernstein} to bound the upper-left block by $M\wt{\sigma}_\theta^2 I_q/(1+\epsilon_N)$. For the lower-right block on the left-hand side of the second inequality in that lemma, we have
\begin{align*}
 \Ep_\pi \brk{\sum_{k=1}^N \frac{\sum_{i=k}^N Y_{\pi(i)}^\top W_{\pi(k)}^2 Y_{\pi(i)}}{N+1-k}}
 \preceq \frac{H_N}{N}\sum_j Y_j^\top \prn{\sum_l W_l^2} Y_j
 \preceq \frac{q'M\wt{\sigma}_\theta^2}{1+\epsilon_N} I_r.
\end{align*}
In the first inequality, we extend the inner sum to all $i\in[N]$ and use that $\pi(k)$ is uniform on $[N]$. In the second, we use $\sum_l W_l^2\preceq M I_q$, $H_N\leq q'$, and $(1+\epsilon_N)N^{-1}\sum_j\norm{Y_j}_F^2\leq\wt{\sigma}_\theta^2$. We rescale $\lambda$ as in Step~3 of the same proof and apply Lemma~\ref{lem:MGF-tail-matrix} to obtain
\begin{align*}
	\P \brc{\norm{\sum_{k=1}^N U_k^\top U_k \theta_k - \wb{\theta}} \geq \wt{\sigma}_\theta \sqrt{2 (1+\epsilon_N) q'M \log \frac{q+r}{\delta}}  + \frac{2L}{3} (1+\epsilon_N) \log \frac{q+r}{\delta} } \leq \delta.
\end{align*}
This completes the proof.

\end{document}